\documentclass{article}

\usepackage{amsmath, amssymb, eucal, latexsym, multicol, amsthm, bm, placeins}
\usepackage{ifpdf}
\usepackage{xcolor}
\usepackage{epsfig}
\usepackage{bm}
\usepackage{stmaryrd}
\usepackage{fullpage}
\usepackage{enumerate}

\usepackage{stmaryrd}
\usepackage{url}

\usepackage[abs]{overpic}

\usepackage{amssymb,epsfig,xspace}
\usepackage{ifpdf}

\usepackage{bm}
\usepackage{bbold}
\usepackage{mathrsfs}

\usepackage{subfigure}

\newtheorem{theorem}{Theorem}[section]
\newtheorem{lemma}[theorem]{Lemma}
\newtheorem{proposition}[theorem]{Proposition}
\newtheorem{corollary}[theorem]{Corollary}

\newtheorem{remark}[theorem]{Remark}
\newtheorem{example}[theorem]{Example}
\newtheorem{algorithm}[theorem]{Algorithm}

\theoremstyle{definition}

\newcommand{\uumlaut}{{\"u}} 

\newcommand{\eps}{\varepsilon}

\newcommand{\Cst}{L}

\newcommand{\R}{\mathbb R}
\newcommand{\N}{\mathbb N}
\newcommand{\Z}{\mathbb Z}

\newcommand{\LL}{\mathbb \Lambda}
\newcommand{\UU}{\mathbb U}

\newcommand{\cF}{\mathcal F}

\newcommand{\ff}{f}
\newcommand{\lappie}{\chi}

\newcommand{\low}{r}
\newcommand{\upp}{R}

\newcommand{\Upper}{C_{\rm upp}}
\newcommand{\Lower}{c_{\rm low}}

\newcommand{\be}{\begin{equation}}
\newcommand{\ee}{\end{equation}}

\DeclareFontFamily{U}{mathx}{\hyphenchar\font45}
\DeclareFontShape{U}{mathx}{m}{n}{
      <5> <6> <7> <8> <9> <10>
      <10.95> <12> <14.4> <17.28> <20.74> <24.88>
      mathx10
      }{}
\DeclareSymbolFont{mathx}{U}{mathx}{m}{n}
\DeclareFontSubstitution{U}{mathx}{m}{n}
\DeclareMathAccent{\widecheck}{0}{mathx}{"71}
\DeclareMathAccent{\wideparen}{0}{mathx}{"75}

\newcommand{\tria}{{\tau}}
\newcommand{\bbT}{{\mathcal T}}
\newcommand{\nodes}{\mathcal{N}}
\newcommand{\cM}{\mathcal{M}}

\newcommand{\cP}{\mathcal{P}}
\newcommand{\cE}{\mathcal{E}}
\newcommand{\cD}{\mathcal{D}}
\newcommand{\cA}{\mathcal{A}}
\newcommand{\cB}{\mathcal{B}}

\newcommand{\stab}{C_{\rm st}}

\DeclareMathOperator{\ran}{ran}
\DeclareMathOperator{\meas}{meas}
\DeclareMathOperator*{\argmin}{argmin}

\DeclareMathOperator{\supp}{supp}

\DeclareMathOperator{\divv}{div}

\newcommand{\cL}{\mathcal L}
\newcommand{\Lis}{\cL\mathrm{is}}

\newcommand{\refine}{\mbox{\rm \texttt{refine}}}
\newcommand{\Mark}{\mbox{\rm \texttt{mark}}}
\newcommand{\estimate}{\mbox{\rm \texttt{estimate}}}
\newcommand{\solve}{\mbox{\rm \texttt{solve}}}
\newcommand{\afem}{\mbox{\rm \texttt{afem}}}
\newcommand{\nipu}{\mbox{\rm \texttt{nested-inexact-preconditioned-Uzawa}}}

\newcommand{\err}{\mbox{\rm Err}}
\usepackage{scalerel}

\newenvironment{algotab}%
{\par\begin{samepage}%
\begin{tabbing}\ttfamily%
 \hspace*{5mm}\=\hspace{3ex}\=\hspace{3ex}\=\hspace{3ex}\=\hspace{3ex}%
\=\hspace{3ex}\=\hspace{3ex}\=\hspace{3ex}\=\hspace{3ex}\kill}%
{\end{tabbing}\end{samepage}}

\title{An optimal adaptive Fictitious Domain Method}

\date{\today}

\author{
Stefano~Berrone\thanks{Dipartimento di Scienze Matematiche,
Politecnico di Torino,
Corso Duca degli Abruzzi 24,
I-10129 Torino, Italy (stefano.berrone@polito.it )} 
\and
Andrea~Bonito%
\thanks{Department of Mathematics, Texas A\&M Unicersity, College Station, TX 77843, USA (bonito@math.tamu.edu)}
\and Rob~Stevenson%
\thanks{Korteweg-de Vries Institute for Mathematics,
University of Amsterdam,
P.O. Box 94248,
1090 GE Amsterdam, The Netherlands
(r.p.stevenson@uva.nl)}
\and Marco~Verani\thanks{MOX-Dipartimento di Matematica, Politecnico di Milano, P.zza Leonardo Da Vinci 32, I-20133 Milano, Italy (marco.verani@polimi.it).}}

\begin{document}
\maketitle
\begin{abstract} 
We consider a Fictitious Domain formulation of an elliptic partial differential equation and approximate the resulting saddle-point system using an inexact preconditioned Uzawa iterative algorithm.
Each iteration entails the approximation of an elliptic problems performed using adaptive finite element methods.
We prove that the overall method converges with the best possible rate and illustrate numerically our theoretical findings.
\end{abstract}

\section{Introduction} \label{Sintro}
In many engineering applications the efficient numerical solution of partial differential equations on 
complex geometries is of paramount importance. In this respect, one crucial issue is the construction of the computational grid. To face this problem, one can basically resort to two different types of approaches. In the first approach, a mesh is constructed on a sufficiently accurate approximation of the exact physical domain (see, e.g.,  isoparametric finite elements \cite{Ciarlet:book}, isogeometric analysis \cite{IGA:book}, or Arbitrary Lagrangian-Eulerian formulation \cite{DGH,HAC,HLZ}), while in the second approach one embeds the physical domain into a simpler computational mesh whose elements can intersect the boundary of the given domain. Clearly, the mesh generation process is extremely simplified in the second approach, while the imposition of boundary conditions requires extra work. 
The second approach is in particular useful
when the domain changes
during the computation, such as in free-boundary and shape optimization problems.

Among the huge variety of methods sharing the philosophy of the second approach, let us mention here the Immersed Boundary methods (see, e.g., \cite{Peskin:acta}), the Penalty Methods (see, e.g., \cite{Babuska:penalty}), the Fictitious Domain/Embedding Domain Methods (see, e.g., \cite{Borgers-Widlund:1990,BG03}) and the Cut Element method (see, e.g. \cite{BH10,BH12}).

Following up on our  earlier work \cite{20.195}, we consider the Fictitious Domain Method with Lagrange multiplier introduced in \cite{Glowinski:1994, GiraultGlowinski1995} (see also \cite{Babuska:1972} for the pioneering work inspiring this approach). In this approach, the physical domain $\wideparen{\Omega}$ with boundary $\gamma$ is embedded into a simpler and larger domain $\Omega$ (the fictitious domain), the right-hand side is extended to the fictitious domain and the boundary conditions on $\gamma$ are appended through the use of a Lagrange multiplier. The Fictitious Domain Method gives rise to a symmetric saddle point problem whose exact primary solution restricted to $\wideparen{\Omega}$ corresponds to the solution of the original problem. 

Even for smooth data, generally the solution of this saddle point problem is non-smooth. 
Indeed, when posed on a non-smooth, non-convex domain, generally already the solution of the original PDE will be non-smooth.
Depending on the extension of the data, the solution of the extended problem might even be more singular.
To achieve nevertheless the best possible convergence rate allowed by the polynomial orders of the applied trial spaces, we will apply an adaptive solution method.

Convergence and optimality of adaptive methods has been demonstrated for elliptic problems, but much less is known for saddle point problems.
Exceptions are given by the special cases of mixed discretizations of Poisson's problem (see e.g \cite{19.95,38.45, 37.45,140}), and the pseudostress-velocity formulation of the Stokes problem (see \cite{37.46,139.5}),
where optimal rates were established by demonstrating that  
the finite element approximation for the flux or pseudo-stress is near-best in the sense that it provides the quasi-orthogonality axiom from \cite{35.93555}.

 In this work we focus on Fictitious Domain Method on a two-dimensional domain with the application of piecewise constant trial spaces for the Lagrange multiplier $\lambda$ and continuous piecewise linears for the primary variable $u$. 
In the spirit of the method no kind of alignment is assumed between the partitions of $\gamma$, and the restriction to $\gamma$ of the partitions of the fictitious domain.
Following an idea from \cite{18.8}, we solve the saddle-point problem with a nested inexact preconditioned Uzawa iteration (see Algorithm \ref{a:fd}): an iterative scheme hinging upon three nested loops.
The outer loop adjusts the Galerkin approximation space for the Schur complement equation that determines $\lambda$.
The intermediate loop solves this Galerkin system by a damped Richardson iteration. Each iteration of the latter involves solving an elliptic problem on the fictitious domain whose solution is approximated in the inner loop.
For sufficiently smooth data, it holds that $\lambda \in L_2(\gamma)$. Therefore, in view of the orders of the trial spaces there is no (qualitative) benefit in applying locally refined partitions on $\gamma$ for the approximation of $\lambda$.
The arising `inner' elliptic problems will be solved with an adaptive finite element method (afem). A complication is that the forcing functional for these problems involves a weighted integral on $\gamma$ meaning that the data is not in $L_2(\Omega)$.  We apply the afem from \cite{45.47} that allows for data in $H^{-1}(\Omega)$.
Since the Schur complement operator of our saddle point problem is an operator of order $-1$, the Richardson iteration requires a preconditioner. We will apply a biorthogonal wavelet preconditioner.
The overall method will be proven to converge with the best possible rate (see Theorem~\ref{final}).

At the end of this paper, it will be shown that our results apply verbatim to the $d$-dimensional setting.
A difference though is the following:
the extension of the original PDE to the fictitious domain yields a solution that is generally non-smooth over the interface.
As we will demonstrate this has the consequence that, in three and more dimensions,  best (isotropic) local refinements provide a rate that is generally lower than for a smooth solution (in 3 dimensions, $\frac{1}{4}$ vs. $\frac{1}{3}$).
This problem can be cured by constructing a proper extension of the right-hand side to the fictitious domain which will be studied in forthcoming work (cf. \cite{Mommer-ima}).

The outline of the paper is as follows. In Sect.~\ref{ficcie} we recall the Fictitious Domain Method. 
In Sect.~\ref{saddle}--\ref{NestedUz}, we consider the solution of an abstract, infinite dimensional symmetric saddle point problem by the Uzawa iteration. We discuss the reduction of the saddle-point problem to its Schur complement (Sect.~\ref{saddle}), preconditioning of this Schur complement (Sect.~\ref{PrecUz}), a posteriori error estimation (Sect.~\ref{Apostest}), and the inexact preconditioned Uzawa iteration combined with a nested iteration technique (Sect.~\ref{NestedUz}).
The inexactness of the iteration refers to the fact that the application of the Schur complement is approximated by replacing the exact inverse of the `left upper block operator' by a call of an (adaptive) finite element solver.

The results in Sect.~\ref{saddle}--\ref{NestedUz} provide a framework for the development of optimal adaptive routines for solving general symmetric saddle point problems. In this context, note that 
any problem $\argmin_{u \in \mathbb{A}}\|Bu-f\|_{\mathbb{B}'}$, where for Hilbert spaces $\mathbb{A}$, $\mathbb{B}$, $B\colon \mathbb{A}\rightarrow \ran B \subseteq \mathbb{B}'$ is boundedly invertible and $f \in \mathbb{B}'$, can be reformulated as the well-posed symmetric saddle point problem $\left[\begin{array}{@{}cc@{}} R & B\\B'&0\end{array} \right]\left[\begin{array}{@{}c@{}}y \\ u\end{array} \right]=\left[\begin{array}{@{}c@{}}f \\ 0\end{array} \right]$, with $R$ being the Riesz mapping on $\mathbb{B}$ (e.g. \cite{45.44}).

In Sect.~\ref{s:inner}, we consider the afem from \cite{45.47} for solving Poisson's problem with $H^{-1}(\Omega)$ data.
We show convergence and optimality of a variant that avoids an inner loop for reducing data oscillation.
In Sect.~\ref{Svaryingrhs}, we apply this afem for solving the `inner' elliptic problems in 
the inexact preconditioned Uzawa iteration applied to the fictitious domain problem, and show that the overall method converges with the best possible rate.
In Sect.~\ref{Numres}, we report on numerical experiments obtained with our adaptive Fictitious Domain solver.
Finally, general space dimensions and/or higher order approximations will be discussed in Sect.~\ref{Sgeneralizations}.

In this work, by $C \lesssim D$ we will mean that $C$ can be bounded by a multiple of $D$, independently of parameters which C and D may depend on. Obviously, $C \gtrsim D$ is defined as $D \lesssim C$, and $C\eqsim D$ as $C\lesssim D$ and $C \gtrsim D$.

For normed linear spaces $\mathbb{A}$ and $\mathbb{B}$, $\cL(\mathbb{A},\mathbb{B})$ will denote the space of bounded linear mappings $\mathbb{A} \rightarrow \mathbb{B}$ endowed with the operator norm $\|\cdot\|_{\cL(\mathbb{A},\mathbb{B})}$. The subset of invertible operators in $\cL(\mathbb{A},\mathbb{B})$  with inverses in $\cL(\mathbb{B},\mathbb{A})$
will be denoted as $\Lis(\mathbb{A},\mathbb{B})$.

\section{Fictitious domain method} \label{ficcie}
On a {\em two-dimensional} domain $\wideparen{\Omega} \subset \R^2$ with Lipschitz continuous boundary $\gamma$, 
and $\wideparen{f} \in L_2(\wideparen{\Omega}) \hookrightarrow H^{-1}(\wideparen{\Omega})$, $g \in H^1(\gamma) \hookrightarrow H^{\frac{1}{2}}(\gamma)$, we consider the Poisson problem
\be \label{16}
\left\{\begin{array}{rcll} -\Delta \wideparen{u} & \!=\! & \wideparen{f} & \text{ on }\wideparen{\Omega},\\
 \wideparen{u} & \!=\! & g & \text{ on } \gamma.
 \end{array}
 \right.
 \ee
On a Lipschitz $\Omega \subset \R^2$ with $\wideparen{\Omega} \Subset \Omega$, $f \in L_2(\Omega)$ being an $L_2$-bounded extension of $\wideparen{f}$, and the bilinear forms $a(u,v):=\int_\Omega \nabla u \cdot \nabla v\,dx$, $b(v,\lambda):=- \int_\gamma v \lambda \,ds$, we consider the problem of finding $(u,\lambda) \in H^1_0(\Omega) \times H^{-\frac{1}{2}}(\gamma)$ such that
\begin{equation} \label{17}
\begin{split}
a(u,v)+b(v,\lambda) & = \int_\Omega f v \,dx \quad (v \in H^1_0(\Omega)),\\
b(u,\mu) & = -  \int_\gamma g \mu\,d s \quad (\mu \in H^{-\frac{1}{2}}(\gamma)),
\end{split}
\end{equation}
where $\int_\gamma g \mu\,d s$ should be read as the unique extension of the $L_2(\gamma)$-scalar product to the duality pairing on $H^{\frac{1}{2}}(\gamma)\times H^{-\frac{1}{2}}(\gamma)$.
It is well-known that this saddle-point defines a boundedly invertible mapping between $H^1_0(\Omega) \times H^{-\frac{1}{2}}(\gamma)$ and its dual, the main ingredient being the fact that 
$\inf\{\|v\|_{H^1(\Omega)}\colon v|_{\gamma}=\mu\}$ defines an equivalent norm on $H^{\frac{1}{2}}(\gamma)=(H^{-\frac{1}{2}}(\gamma))'$. 
Setting $\breve{\Omega}:=\Omega \setminus \overline{\wideparen{\Omega}}$, and applying integration-by-parts to both terms in  $a(u,v)=\int_{\wideparen{\Omega}} \nabla u \cdot \nabla v\,dx+\int_{\breve{\Omega}} \nabla u \cdot \nabla v\,dx$, one infers that $u|_{\wideparen{\Omega}}=\wideparen{u}$, being the solution of \eqref{16},
that $\breve{u}:=u|_{\breve{\Omega}}$ solves $-\Delta \breve{u}=f$ on $\breve{\Omega}$, $\breve{u}=g$ on $\gamma$, and $\breve{u}=0$ on $\partial{\Omega}$,
and finally that $\lambda=
\frac{\partial \breve{u}}{\partial \vec{n}}|_{\gamma}-\frac{\partial \wideparen{u}}{\partial \vec{n}}|_{\gamma}$, where $\vec{n}$ is the normal to $\gamma$ exterior to $\wideparen{\Omega}$.

Since these Poisson problems on both Lipschitz domains $\wideparen{\Omega}$ and $\breve{\Omega}$ have forcing terms in $L_2$ and Dirichlet boundary data in $H^1$, \cite[Ch.~5, Thm.~1.1]{234.5} 
shows that
\begin{equation} \label{170}
\lambda \in L_2(\gamma), \text{ with } \|\lambda\|_{L_2(\gamma)} \lesssim \|f\|_{L_2(\Omega)}+\|g\|_{H^1(\gamma)}.
\end{equation}

We are going to approximate the solution $(u,\lambda)$ of \eqref{17} by functions from finite element spaces, where we consider the lowest order case by taking {\em continuous piecewise linears} for the approximation for $u$, and {\em piecewise constants} for the approximation for $\lambda$. 

 Taking into account the two-dimensional domain and the orders of the finite element spaces, the error measured in $H^1(\Omega)$-norm of the best approximation for $u$ can be expected to be generally at best of order $N^{-\frac{1}{2}}$, where $N$ denotes the dimension of the finite element space on $\Omega$.
In view of \eqref{170}, the error measured in $H^{-\frac{1}{2}}(\gamma)$-norm of the best approximation for $\lambda$ from the space of piecewise constants w.r.t. a \emph{quasi-uniform} partition of $\gamma$ into $N$ pieces is of order $N^{-\frac{1}{2}}$. Since apparently no overall (qualitative) advantage can be obtained from the application of locally refined partitions on $\gamma$, we will consider a sequence of \emph{uniform} dyadically refined partitions on $\gamma$.

\section{Symmetric Saddle point problem} \label{saddle}
The variational problem that arises from the fictitious domain method is an example of a symmetric saddle point problem, that in this and the following three sections will be studied in an abstract setting.

Let $\UU$ and $\LL$ Hilbert spaces.
For a bilinear, bounded, symmetric, and coercive $a:\UU \times \UU \rightarrow \R$, a bilinear and bounded $b:\UU \times \LL \rightarrow \R$ with \mbox{$\inf_{0 \neq \mu \in \LL} \sup_{0 \neq w \in \UU} \frac{b(w,\mu)}{\|w\|_\UU \|\mu\|_\LL}>0$} (`inf-sup' condition), given $(f,g) \in \UU' \times\LL'$ we
consider the problem of finding \framebox{$(u,\lambda) \in \UU \times \LL$} that satisfies
\be \label{2}
a(u,v)+b(v,\lambda)+b(u,\mu)=f(v) - g(\mu) \quad ((v,\mu) \in \UU \times \LL).
\ee
It is well-known that under aforementioned conditions on $a$ and $b$, 
$$
(u,\lambda) \mapsto ((v,\mu) \mapsto a(u,v)+b(v,\lambda)+b(u,\mu)) \in \Lis(\UU \times \LL,(\UU \times \LL)').
$$

With $A \in \Lis(\UU,\UU')$, $B \in \cL(\UU,\LL')$ defined by $(Au)(v)=a(u,v)$, $(Bu)(\lambda)=b(u,\lambda)$,  equivalent formulations of \eqref{2} are given by
$$
\left[
\begin{array}{@{}cc@{}}
A & B'\\
B & 0
\end{array}
\right]
\left[
\begin{array}{@{}c@{}}
u\\
\lambda
\end{array}
\right]
=
\left[
\begin{array}{@{}c@{}}
f\\
-g
\end{array}
\right],
$$
and
$$
\left[
\begin{array}{@{}cc@{}}
A & B'\\
0 & S
\end{array}
\right]
\left[
\begin{array}{@{}c@{}}
u\\
\lambda
\end{array}
\right]
=
\left[
\begin{array}{@{}c@{}}
f\\
B A^{-1}f+g
\end{array}
\right],
$$
where $S:=B A^{-1}B' \in \cL(\LL,\LL')$ is the \emph{Schur complement} operator. Obviously $S=S'$, and furthermore, as demonstrated by the next lemma, $S$ is coercive (so in particular $S \in \Lis(\LL,\LL')$).

\begin{lemma} It holds that $(S\mu)(\mu)=\sup_{0 \neq v \in \UU} \frac{b(v,\mu)^2}{a(v,v)} \eqsim \|\mu\|_\LL^2$ ($ \mu \in \LL$).
\end{lemma}

\begin{proof} 
Let $R_\UU: \UU \rightarrow \UU'$ denote the Riesz map defined by $(R_\UU v)(w)=\langle w, v\rangle_\UU$.
Writing $\tilde{B}'=R_\UU^{-1} B'$, $\tilde{A}= R_\UU^{-1} A$, we have
\begin{align*}
\sup_{0 \neq v \in \UU} \frac{b(v,\mu)^2}{a(v,v)} &=
\sup_{0 \neq v \in \UU} \frac{(B' \mu)(v)^2}{(A v)(v)}=
\sup_{0 \neq v \in \UU} \frac{\langle v,\tilde{B}' \mu\rangle_\UU^2}{\langle v,\tilde{A}v \rangle_\UU}=
\sup_{0 \neq w \in \UU} \frac{\langle w,\tilde{A}^{-\frac{1}{2}}\tilde{B}' \mu\rangle_\UU^2}{\langle w,w\rangle_\UU}\\
&=
\langle \tilde{A}^{-\frac{1}{2}} \tilde{B}' \mu,\tilde{A}^{-\frac{1}{2}} \tilde{B}' \mu\rangle_\UU =
\langle A^{-1} B' \mu,R^{-1}_\UU B' \mu\rangle_\UU
=(S \mu)(\mu) \quad(\mu \in \LL).
\end{align*} 
The second statement follows from the coercivity of $a$, the boundedness of $b$, and the inf-sup condition.
\end{proof}

As we reserved $(u,\lambda)$ to denote the exact solution of the saddle point problem, in the remainder of this section we fix three more notations (i)-(iii) that we use throughout this paper.

(i). For a finite dimensional (or more generally, closed) subspace $\LL_\sigma \subset \LL$, where $\sigma$ runs over a collection ${\mathcal S}$, for $\chi \in \LL$ we let \framebox{$\chi_\sigma \in \LL_\sigma$} denote its \emph{Galerkin approximation} defined by
\be \label{not1}
(S \chi_\sigma)(\mu)=(S \chi)(\mu) \quad(\mu \in \LL_\sigma).
\ee
This $\chi_\sigma$ is the best approximation to $\chi$ from $\LL_\sigma$ w.r.t. to the `energy-norm' $\mu \mapsto \sqrt{(S \mu)(\mu)}$.

(ii). Given a $\chi \in \LL$, we let \framebox{$u^\chi \in \UU$} denote the solution of 
\be \label{not3}
a(u^\chi,v)=f(v)- b(v,\chi) \quad (v \in \UU),
\ee
i.e., $u^\chi=A^{-1}(f-B' \chi)$. 

Notice that $u^\lambda=u$. 
Furthermore, we note that given a $\LL_\sigma \subset \LL$, the pair $(u^{\lambda_\sigma},\lambda_\sigma) \in \UU\times \LL_\sigma$ solves the \emph{semi-discrete 
saddle point problem}
\be \label{not4}
a(u^{\lambda_\sigma},v)+b(v,\lambda_\sigma)+b(u^{\lambda_\sigma},\mu)=f(v) - g(\mu) \quad ((v,\mu) \in \UU \times \LL_\sigma).
\ee
\begin{remark}
Well-posedness of the original saddle-point problem implies this for the semi-discrete one, uniform in $\sigma \in {\mathcal S}$. In other words,
$$
(u,\lambda) \mapsto ((v,\mu) \mapsto a(u,v)+b(v,\lambda)+b(u,\mu)) \in \Lis(\UU \times \LL_\sigma,(\UU \times \LL_\sigma)'),
$$
with both the norm of the operator and that of its inverse being uniformly bounded.
\end{remark}

(iii). For a finite dimensional (or more generally, closed) subspace $\UU_\tria \subset \UU$, where $\tria$ runs over a collection $\bbT$, for $w \in \UU$ we let \framebox{$w_\tau \in \UU_\tau$}
denote its \emph{Galerkin approximation} defined by
\be \label{not2}
a(w_\tau,v)=a(w,v) \quad(v \in \UU_\tau),
\ee
being the best approximation to $w$ from $\UU_\tau$ w.r.t. $v \mapsto \sqrt{a(v,v)}$.

\begin{remark} \label{LBB} Since we never solve any \emph{fully} discrete saddle-point problem, i.e., a system \eqref{2} in which the test- and trial space $\UU\times \LL$ is replaced by $\UU_\tria\times \LL_\sigma$, 
a Ladyzhenskaya-Babu\v{s}ka-Brezzi (LBB) condition ensuring stability of the latter will never enter our considerations.
\end{remark}

\section{Preconditioned Uzawa iteration} \label{PrecUz}
With $I_\sigma:\LL_\sigma \rightarrow \LL$ being the trivial embedding, and $I'_\sigma:\LL' \rightarrow \LL_\sigma'$ its adjoint, the Galerkin approximation $\lambda_\sigma \in \LL_\sigma$ for $\lambda$ solves
\be \label{gal}
S_\sigma \lambda_\sigma = I_\sigma' (B A^{-1} f +g),\quad \text{where }S_\sigma:=I_\sigma' S I_\sigma \in \Lis(\LL_\sigma,\LL'_\sigma).
\ee
At some occasions, $I_\sigma$ will be omitted from the notation.

Although $S_\sigma$ is a mapping between finite dimensional spaces, its matrix representation cannot be computed. Since on the other hand the application of $S_\sigma$ can be mimicked by approximating the application of $A^{-1}$, for solving \eqref{gal} we will resort to an iterative method. In order to do so,
we need  a (uniform) `preconditioner': Let $M_\sigma \in \Lis(\LL_\sigma, \LL'_\sigma)$ be such that $M_\sigma=M_\sigma'$, and, for some constants
$\low, \upp>0$
\be \label{precond}
\low \|\mu\|_\LL^2  \leq (M_\sigma \mu)(\mu) \leq \upp \|\mu\|_\LL^2 \quad (\mu \in \LL_\sigma,\, \sigma \in {\mathcal S}).
\ee
W.r.t. the scalar product $(\mu,\chi) \mapsto (M_\sigma \mu)(\chi)$ on $\LL_\sigma \times \LL_\sigma$, the operator $M_\sigma^{-1} S_\sigma:\LL_\sigma \rightarrow \LL_\sigma$ is symmetric, coercive, and uniformly boundedly invertible. 

For solving \eqref{gal}, we  consider the \emph{damped, preconditioned Richardson iteration} that, for given $\lambda_\sigma^{(0)} \in \LL_\sigma$, produces $(\lambda_\sigma^{(j)})_{j \geq 0} \subset \LL_\sigma$  defined by
\begin{align} \nonumber
\lambda_\sigma^{(j+1)}:&=\lambda_\sigma^{(j)}+\beta M_\sigma^{-1} I_\sigma' (BA^{-1}f + g-S I_\sigma \lambda_\sigma^{(j)})\\
&= \lambda_\sigma^{(j)}+\beta M_\sigma^{-1} I_\sigma' (B u^{\lambda_\sigma^{(j)}} +g) \label{Uzawa}
\end{align}
(cf. \eqref{not3}), in the latter form known as the (damped) \emph{preconditioned Uzawa iteration}.
Taking a constant $\beta \in \Big(0,\frac{2}{\sup_{\sigma \in {\mathcal S}}\rho(M_\sigma^{-1} S_\sigma)}\Big)$, in each step of \eqref{Uzawa} the error measured in the norm on $\LL_\sigma$ associated to either $S_\sigma$ or $M_\sigma$ is reduced by at least the factor
\be \label{alpha_rho}
\rho:=\sup_{\sigma \in {\mathcal S}} \rho(I-\beta M_\sigma^{-1} S_\sigma)<1.
\ee
With the optimal choice
\begin{equation}
\beta=\frac{2}{\sup_{\sigma \in {\mathcal S}} \rho(M_\sigma^{-1}
  S_\sigma)+(\sup_{\sigma \in {\mathcal S}} \rho(M_\sigma
  S_\sigma^{-1}))^{-1}},
\label{beta.opt}
\end{equation}
 it holds that $\rho=\frac{\kappa-1}{\kappa+1}$ where $\kappa:=\sup_{\sigma \in {\mathcal S}} \rho(M_\sigma^{-1} S_\sigma)\sup_{\sigma \in {\mathcal S}} \rho(M_\sigma S_\sigma^{-1})$.

To reformulate \eqref{Uzawa} in coordinates, let $\Phi_\sigma$ be a basis for $\LL_{\sigma}$. 
We set $\cF_{\sigma}:\R^{\dim \LL_{\sigma}} \rightarrow \LL_{\sigma}:{\bf c} \mapsto {\bf c}^\top \Phi_\sigma$, so that, equipping $\R^{\dim \LL_{\sigma}}  \eqsim (\R^{\dim \LL_{\sigma}})'$ with the standard Euclidean scalar product $\langle\,,\,\rangle$, its adjoint
$\cF_{\sigma}':\LL_{\sigma}' \rightarrow \R^{\dim \LL_{\sigma}}$ is the mapping $f \mapsto f(\Phi_\sigma)$.
Setting $\bm{\lambda}_{\sigma}^{(j)}:=\cF_{\sigma}^{-1} \lambda_{\sigma}^{(j)}$, i.e, $\bm{\lambda}_{\sigma}^{(j)}$ is the coordinate vector of $\lambda_{\sigma}^{(j)}$ w.r.t. $\Phi_\sigma$, an equivalent formulation of \eqref{Uzawa} reads as 
\begin{align*}
\bm{\lambda}_{\sigma}^{(j+1)} &= \bm{\lambda}_{\sigma}^{(j)}+\beta (\cF_{\sigma}' M_{\sigma} \cF_{\sigma})^{-1} \cF_{\sigma}' I_{\sigma}' (B u^{\lambda_{\sigma}^{(j)}} +g)\\
&= \bm{\lambda}_{\sigma}^{(j)}+\beta {\bf M}_\sigma^{-1}
 (B u^{\lambda_{\sigma}^{(j)}} +g)(\Phi_\sigma).
\end{align*}
with {\em preconditioner} ${\bf M}_\sigma:=\cF_{\sigma}' M_{\sigma} \cF_{\sigma}$.

The analysis of a practical scheme where $u^{\lambda_\sigma^{(j)}}$ is replaced by a (Galerkin) approximation from a finite dimensional subspace of $\UU$ is postponed to Sect.~\ref{NestedUz}.

\begin{example} With $R_\LL:\LL \rightarrow \LL'$ being the Riesz map defined by $(R_\LL q)(r)=\langle r,q\rangle_\LL$, the Riesz map $R_{\LL_\sigma}:\LL_\sigma \rightarrow \LL_\sigma'$ is given by $I_\sigma' R I_\sigma$. For the choice $M_\sigma=R_{\LL_\sigma}$ (which obviously satisfies \eqref{precond}), for $\chi \in \LL$, $\mu \in \LL_\sigma$ we have
$$
\langle M_\sigma^{-1} I_\sigma' R_\LL \chi,\mu\rangle_\LL=(I_\sigma' R_\LL \chi)(\mu)=(R_\LL \chi)(\mu)=\langle \chi,\mu\rangle_\LL,
$$
i.e., $M_\sigma^{-1} I_\sigma' R_\LL=Q_\sigma$, being the $\LL$-orthogonal projector onto $\LL_\sigma$. So with this choice of $M_\sigma$, the second line in \eqref{Uzawa} reads as 
$$
\lambda_\sigma^{(j+1)}:=\lambda_\sigma^{(j)}+\beta Q_\sigma R_\LL^{-1} (B u^{(j)} +g).
$$
This choice of $M_\sigma$ seems only practically feasible when $\LL$ is an $L_2$-space.

In the setting of a stationary Stokes problem, it holds that $\UU=H^1_0(\Omega)^n$, $\LL=L_{2}(\Omega)/\R$, and 
$R_\LL^{-1}B=\divv$. So with $M_\sigma=R_{\LL_\sigma}$, and writing $R_\LL^{-1} g$ simply as $g$,
the second line in \eqref{Uzawa} reads as $\lambda_\sigma^{(j+1)}:=\lambda_\sigma^{(j)}+\beta Q_\sigma (\divv u^{(j)} +g)$.
From $\|\divv \cdot\|_{L_2(\Omega)} \leq \|\nabla \cdot\|_{L_2(\Omega)^{n^2}}$ on $\UU$, one infers that in this case one can take $\beta=1$, see \cite{239.2}.
\end{example}

\begin{example} \label{ex1}
In the case of the fictitious domain method introduced in Sect.~\ref{ficcie}, we have $\LL=H^{-\frac{1}{2}}(\gamma)$ so that a non-trivial preconditioner is required.
Uniform preconditioners of \emph{multi-level type} of \emph{linear complexity} even on locally refined partitions have recently been proposed: 
Preconditioners of (additive) \emph{subspace correction type} were constructed for two- or three-dimensional domains $\wideparen{\Omega}$ in \cite{75.065} or \cite{75.256}.
Within the framework of \emph{operator preconditioning} (\cite{138.26}), preconditioners for two- and three-dimensional domains are constructed in \cite{249.97,249.98}.

We now consider the special setting where $\wideparen{\Omega} \subset \R^2$ and $\{0\}=\LL_{\sigma_0} \subset \LL_{\sigma_1} \subset \cdots \subset \LL$ is
 a sequence of spaces of piecewise constant functions w.r.t. to a sequence of \emph{uniformly} dyadically refined partitions $\sigma_1 \prec \sigma_2 \prec \cdots$ of $\gamma=\partial\wideparen{\Omega}$, with  $\sigma_1=\sigma_\bot $ is some fixed `bottom' partition.
 In this case, we can follow \cite{242.81} and construct a \emph{wavelet preconditioner} based on a compactly supported and piecewise constant wavelet basis for $H^{-\frac{1}{2}}(\R /\Z)$.
 All wavelets with `levels' less or equal to $i$ span all piecewise constants w.r.t. a partition of $[0,1]$ into $2^{-(i-1)}\#\sigma_\bot$ equally-sized subintervals.
Lifting this basis to $\gamma$, the uniform preconditioner $M_{\sigma_i} \in \Lis(\LL_{\sigma_i},\LL_{\sigma_i}')$ is defined by ${\bf M}_{\sigma_i}^{-1}={\bf T}_i {\bf T}_i^\top$, where ${\bf T}_i$ is the basis transformation from the wavelet basis to the canonical single scale basis $\Phi_{\sigma_i}$ for $\LL_{\sigma_i}$, which can be performed in linear complexity (see, e.g., the appendix of \cite{BBSV17} for more details). 
 This is the strategy adopted in the numerical experiments proposed in Section~\ref{Numres}.
\end{example}

Relevant references for Uzawa iterations in possibly infinite dimensional settings include \cite{34.67,48.4,18.8,18.675,174,75.066}.
At some places in the literature, $\LL$ is (implicitly) identified with its dual using the Riesz map. 
Although appropriate for $L_2$ type spaces, it may obscure the need for a preconditioner in other cases.

\section{A posteriori error estimation} \label{Apostest}
The preconditioned Uzawa scheme yields some approximation $\chi \in \LL_\sigma$ to $\lambda_\sigma$, the latter being the Galerkin approximation to $\lambda$ from $\LL_\sigma$.
To asses the quality of both of these approximations we derive a posteriori error estimators for 
$\|\lambda_\sigma -\chi\|_\LL$ and $\|\lambda-\lambda_\sigma\|_\LL$.
It is natural to expect that such estimators depends on $u^\chi$ or $u^{\lambda_\sigma}$.
However, since only their approximation $\tilde u$ is available, we derive instead estimators in terms of $\tilde u$ and show that they are reliable and efficient under the assumption that the error in $\tilde u$ is sufficiently small in a relative sense.

\begin{proposition} \label{apost} For $\sigma \in {\mathcal S}$, let $\chi \in \LL_\sigma$ and $\tilde u \in \UU$ be approximations to $\lambda_\sigma$ and $u^{\chi}$, respectively. Then it holds that
\begin{align} \label{first1}
\|\lambda_\sigma-\chi\|_\LL \eqsim \|u^{\lambda_\sigma}-u^\chi\|_\UU  \eqsim \sup_{0 \neq \mu \in \LL_\sigma}
\frac{b(u^\chi,\mu)+g(\mu)}{(M_\sigma \mu)(\mu)^{\frac{1}{2}}},\\ \label{first2}
\left|\sup_{0 \neq \mu \in \LL_\sigma}
\frac{b(u^\chi,\mu)+g(\mu)}{(M_\sigma \mu)(\mu)^{\frac{1}{2}}}-\sqrt{\langle {\bf M}^{-1}_\sigma {\bf r},{\bf r}\rangle}  \right| \lesssim 
\|u^\chi-\tilde u\|_\UU,
\end{align}
where ${\bf r}:= (B \tilde u+g)(\Phi_\sigma)$, and furthermore that
\begin{align} \label{second1}
&\|\lambda -\lambda_\sigma\|_\LL \eqsim \|u-u^{\lambda_\sigma}\|_\UU \eqsim \|Bu^{\lambda_\sigma}+g\|_{\LL'},\\ \label{second2}
&\left|\|Bu^{\lambda_\sigma}+g\|_{\LL'}-\|B \tilde u+g\|_{\LL'}\right| \lesssim  \|u^{\lambda_\sigma}-\tilde u\|_\UU.
\end{align}
 
So if$\frac{\|u^\chi-\tilde u\|_\UU}{\sqrt{\langle {\bf M}^{-1}_\sigma {\bf r},{\bf r}\rangle}}$ or $\frac{ \|u^{\lambda_\sigma}-\tilde u\|_\UU}{\|B \tilde u+g\|_{\LL'}}$ are sufficiently small, then $\|\lambda_\sigma-\chi\|_\LL \eqsim  \sqrt{\langle {\bf M}^{-1}_\sigma {\bf r},{\bf r}\rangle} $ or $\|\lambda -\lambda_\sigma\|_\LL \eqsim \|B \tilde u+g\|_{\LL'}$.
\end{proposition}

\begin{remark} 
In applications, $\tilde u$ will be a Galerkin approximation to $u^{\chi}$. For our fictitious domain application, in Sect.~\ref{apost_inner} an a posteriori error estimator for $\|u^{\chi}-\tilde u\|_\UU$  or $\|u^{\lambda_\sigma}-\tilde u\|_\UU$ (modulo `data oscillation') will be given to assess the smallness
of  $\frac{\|u^\chi-\tilde u\|_\UU}{\sqrt{\langle {\bf M}^{-1}_\sigma {\bf r},{\bf r}\rangle}}$ or $\frac{\|u^{\lambda_\sigma}-\tilde u\|_\UU}{\|B \tilde u+g\|_{\LL'}}$.
\end{remark}

\begin{proof}[Proof of Proposition~\ref{apost}] The validity of the first $\eqsim$-symbol in \eqref{first1} follows from
$$
\sup_{0 \neq v \in \UU}\frac{a(u^{\lambda_\sigma}-u^\chi,v)}{\|v\|_\UU}=\sup_{0 \neq v \in \UU}\frac{b(\chi-\lambda_\sigma,v)}{\|v\|_\UU},
$$
the boundedness and coercivity of $a$, and the boundedness and `inf-sup condition' satisfied by $b$.
 The well-posedness, uniform in $\sigma \in {\mathcal S}$, of the semi-discrete saddle-point problem shows that
\begin{align*}
\|\lambda_\sigma-\chi\|_\LL + \|u^{\lambda_\sigma}-u^\chi\|_\UU &\eqsim \sup_{0 \neq (v,\mu) \in \UU \times \LL_\sigma}
\frac{a(u^{\lambda_\sigma}-u^\chi,v)+b(v,\lambda_\sigma-\chi)+b(u^{\lambda_\sigma}-u^\chi,\mu)}{\|v\|_\UU+\|\mu\|_\LL}\\
& =\sup_{0 \neq \mu \in \LL_\sigma} \frac{g(\mu)+b(u^\chi,\mu)}{\|\mu\|_\LL} \eqsim 
\sup_{0 \neq \mu \in \LL_\sigma} \frac{g(\mu)+b(u^\chi,\mu)}{(M_\sigma \mu)(\mu)^{\frac{1}{2}}}
\end{align*}
by \eqref{precond}. The boundedness of $b$ shows that
$$
\left|\sup_{0 \neq \mu \in \LL_\sigma} \frac{g(\mu)+b(u^\chi,\mu)}{(M_\sigma \mu)(\mu)^{\frac{1}{2}}}-\sup_{0 \neq \mu \in \LL_\sigma} \frac{g(\mu)+b(\tilde u,\mu)}{(M_\sigma \mu)(\mu)^{\frac{1}{2}}}\right| \lesssim \|u^\chi-\tilde u\|_\UU.
$$
The proof of \eqref{first2} is completed by 
\begin{align*}
\sup_{0 \neq \mu \in \LL_\sigma}
\frac{g(\mu)+b(\tilde u,\mu)}{(M_\sigma \mu)(\mu)^{\frac{1}{2}}}
& = \sup_{0 \neq \mu \in \LL_\sigma}\frac{(B \tilde u+g)(\mu)}{(M_\sigma \mu)(\mu)^{\frac{1}{2}}} \\
&\stackrel{\mu=\cF_\sigma {\bf m}}{=} \sup_{0 \neq {\bf m} \in \R^{\dim \LL_\sigma}} \frac{\langle {\bf r},{\bf m}\rangle}{\langle {\bf M}_\sigma{\bf m},{\bf m}\rangle^{\frac12}} = \|{\bf M}_\sigma^{-\frac{1}{2}} {\bf  r}\|.
\end{align*}

Using the same arguments one infers the first $\eqsim$-symbol in \eqref{second1} and
$$
\|\lambda -\lambda_\sigma\|_\LL + \|u-u^{\lambda_\sigma}\|_\UU \eqsim  \|Bu^{\lambda_\sigma}+g\|_{\LL'}.
$$
Now \eqref{second2}  is obvious.
\end{proof}

\section{Nested inexact preconditioned Uzawa iteration} \label{NestedUz}
Returning to the preconditioned Uzawa iteration \eqref{Uzawa}, in order to arrive at an implementable method we will allow for $u^{\lambda_\sigma^{(j)}}$ to be replaced by an approximation.
Furthermore, eventually aiming at a method of optimal computational complexity, we will combine the preconditioned Uzawa iteration with the concept of \emph{nested iteration}: 
Let 
$\{0\}=\LL_{\sigma_0} \subset \LL_{\sigma_1} \subset \cdots \subset \LL$ be such that for some constants $\zeta>1$, $\Cst=\Cst(f,g)>0$ (with $\Cst(\xi f, \xi g)=|\xi| \Cst(f, g)$), it holds that
\be \label{10}
\|\lambda -\lambda_{\sigma_i}\|_\LL \leq \Cst \zeta^{-i}.
\ee

We consider the \emph{nested inexact preconditioned Uzawa iteration} that, with $\lambda_{\sigma_0}^{(K)}=\lambda_{\sigma_0}=0$, for $i=1,2,\cdots$ produces $(\lambda_{\sigma_i}^{(j)})_{0 \leq j \leq K}$ defined by
$$
 \lambda_{\sigma_i}^{(j)}=\left\{
\begin{array}{ll} 
\lambda_{\sigma_{i-1}}^{(K)} & j=0, \\
\lambda_{\sigma_{i}}^{(j-1)}+\beta M_{\sigma_i}^{-1} I_{\sigma_i}' (B u^{(i,j-1)} + g) & 1 \leq j \leq K,
\end{array}
\right.
$$
where $u^{(i,j-1)} \in \UU$ is such that
\be \label{11}
\|u^{\lambda_{\sigma_i}^{(j-1)}}-u^{(i,j-1)}\|_\UU \leq \Cst \zeta^{-i}.
\ee
In the next two sections, such $u^{(i,j-1)}$ will be found as Galerkin approximations to $u^{\lambda_{\sigma_i}^{(j-1)}}$ w.r.t. adaptively generated partitions.
Below, for $K$ a sufficiently large constant, we derive an upper bound for $\|\lambda_{\sigma_i}-\lambda^{(K)}_{\sigma_i}\|_\LL$ that is of the same order as the upper bound for $\|\lambda -\lambda_{\sigma_i}\|_\LL$ from \eqref{10}.

\begin{lemma} \label{basic}
With $\beta$ and $\rho$ from \eqref{alpha_rho}, 
given a constant $M>\frac{\beta \|B\|_{\cL(\UU,\LL')}}{(1-\rho)\low}$, let $K=K(M)$ be a sufficiently large constant such that ${\textstyle \frac{1}{\sqrt{\low}}} \big[ \rho^K \sqrt{\upp}((1+\zeta)+M\zeta) +{\textstyle \frac{1}{1-\rho}} {\textstyle \frac{\beta}{\sqrt{\low}}} \|B\|_{\cL(\UU,\LL')} \big] \leq M$. Then, assuming \eqref{10} and \eqref{11}, we have
$$
\|\lambda_{\sigma_i}-\lambda^{(j)}_{\sigma_i}\|_\LL \leq 
 {\textstyle \frac{1}{\sqrt{\low}}} \big[ \rho^j \sqrt{\upp}((1+\zeta)+M\zeta) +{\textstyle \frac{1}{1-\rho}} {\textstyle \frac{\beta}{\sqrt{\low}}} \|B\|_{\cL(\UU,\LL')} \big] \Cst \zeta^{-i} \quad\big(\lesssim \Cst \zeta^{-i}\big)
$$
$(i \geq 1,\, 0 \leq j \leq K)$, and so in particular,
$$
\|\lambda_{\sigma_i}-\lambda^{(K)}_{\sigma_i}\|_\LL \leq M  \Cst \zeta^{-i} \qquad(i \geq 0).
$$
Furthermore, for $1 \leq j \leq K$
$$
\|u-u^{(i,j-1)}\|_\UU \leq \|A^{-1}B^{'}\|_{\cL(\LL,\UU)}\big(\|\lambda-\lambda_{\sigma_{i}}\|_\LL+\|\lambda_{\sigma_{i}}-\lambda_{\sigma_i}^{(j-1)}\|_\LL\big)+
\Cst \zeta^{-i} \lesssim \Cst \zeta^{-i}.
$$
\end{lemma}

\begin{proof} For $i \geq 1$, define $\|\mu\|_{\sigma_i}:=(M_{\sigma_i} \mu)(\mu)^{\frac{1}{2}}$ ($\mu \in \Lambda_{\sigma_i}$). Then, for $1 \leq j \leq K$,
$$
\|\lambda_{\sigma_i}-\lambda_{\sigma_i}^{(j)}\|_{\sigma_i} \leq \rho \|\lambda_{\sigma_i}-\lambda_{\sigma_i}^{(j-1)}\|_{\sigma_i} +{\textstyle \frac{\beta}{\sqrt{\low}}} \|B\|_{\cL(\UU,\LL')} \Cst \zeta^{-i},
$$
where to arrive at the last term we used \eqref{11} and that the norm on $\LL_{\sigma_i}'$ dual to $\|\cdot\|_{\sigma_i}$ is at most a factor $1/\sqrt{\low}$ larger that the norm on $\LL_{\sigma_i}'$ dual to $\|\cdot\|_\LL$.
By \eqref{10} and induction, we have 
\begin{align*}
\|\lambda_{\sigma_i}-\lambda_{\sigma_{i}}^{(0)}\|_{\sigma_i} & = \|\lambda_{\sigma_i}-\lambda_{\sigma_{i-1}}^{(K)}\|_{\sigma_i} \\ &\leq \sqrt{\upp}(\|\lambda_{\sigma_i}-\lambda\|_\LL+\|\lambda-\lambda_{\sigma_{i-1}}\|_\LL+\|\lambda_{\sigma_{i-1}}-\lambda_{\sigma_{i-1}}^{(K)}\|_\LL)\\
& \leq \sqrt{\upp}((1+\zeta)+M\zeta)\Cst\zeta^{-i},
\end{align*}
and so for $0 \leq j \leq K$,
\be \label{9}
\begin{split}
\|\lambda_{\sigma_i}-\lambda_{\sigma_i}^{(j)}\|_\LL & \leq {\textstyle \frac{1}{\sqrt{\low}}} \|\lambda_{\sigma_i}-\lambda_{\sigma_i}^{(j)}\|_{\sigma_i}\\ &\leq
 {\textstyle \frac{1}{\sqrt{\low}}} \big[ \rho^j \sqrt{\upp}((1+\zeta)+M\zeta) +{\textstyle \frac{1}{1-\rho}} {\textstyle \frac{\beta}{\sqrt{\low}}} \|B\|_{\cL(\UU,\LL')} \big] \Cst \zeta^{-i},
\end{split}
\ee
which completes the proof of the first two statements by definition of $M$.

The second statement follows from
\begin{align*}
\|u-&u^{(i,j-1)}\|_\UU \leq \|u-u^{\lambda_{\sigma_i}^{(j-1)}}\|_\UU+\|u^{\lambda_{\sigma_i}^{(j-1)}}-u^{(i,j-1)}\|_\UU\\
& \leq \|A^{-1}B^{'}\|_{\cL(\LL,\UU)}\big(\|\lambda-\lambda_{\sigma_{i}}\|_\LL+\|\lambda_{\sigma_{i}}-\lambda_{\sigma_i}^{(j-1)}\|_\LL\big)+
\Cst \zeta^{-i} 
\end{align*}
together with \eqref{10} and \eqref{9}. \qedhere
\end{proof}

\section{Inner elliptic solver}\label{s:inner}
Inside the nested inexact preconditioned Uzawa iteration, we need to find a sufficiently accurate approximation $u^{(i,j-1)}$ for $u^{\lambda^{(j-1)}_{\sigma_i}}$, cf. \eqref{11}.
This $u^{\lambda^{(j-1)}_{\sigma_i}}$ is the solution in $\UU$ of the elliptic problem $a(u^\chi,v)=f(v)-b(v,\chi)$ ($v \in \UU$), cf. \eqref{not3}, with $\chi$ reading as $\lambda_{\sigma_i}^{(j-1)}$.
In the application of the fictitious domain method, this problem reads as solving $u^{\lappie} \in H^1_0(\Omega)$ that satisfies
\be \label{ellip}
\int_\Omega \nabla u^{\lappie} \cdot \nabla v \,dx=\int_\Omega \ff v \, dx+\int_\gamma \lappie v \, ds \quad(v \in H^1_0(\Omega)).
\ee
Recall that $\Omega \subset \R^2$, $\gamma \subset \Omega$ is a Lipschitz curve, and $\ff \in L_2(\Omega)$.
For the moment, we consider this problem for some arbitrary, but fixed $\chi \in L_2(\Omega)$.
The discussion how to deal with the fact that $\chi=\lambda^{(j-1)}_{\sigma_i}$ varies with $i$ and $j$ will be postponed to Sect.~\ref{Svaryingrhs}.

For solving \eqref{ellip} we will apply an adaptive linear finite element method. The adaptive triangulations will be generated by newest vertex bisection.

\subsection{Newest vertex bisection}
We recall some properties of newest vertex bisection. Proofs can be found on several places in the literature, e.g. in \cite{21, 249.86}.
Let $\tria_\bot$ 
be a fixed conforming `bottom' triangulation of $\Omega$. Let the assignment of the newest vertices in $\tria_\bot$ be such that if for $T,T' \in \tria_\bot$ the edge $T \cap T'$ is opposite to the newest vertex in $T$, then it is opposite to the newest vertex in $T'$. In \cite{21}, it was shown that such an assignment always exists.

The infinite family of triangulations that can be created from $\tria_\bot$ by newest vertex bisection is \emph{uniformly shape regular} (only dependent on $\tria_\bot$).
The subset of this family of triangulations that additionally is \emph{conforming} will be denoted as $\bbT$.
For $\tria, \tria^* \in \bbT$, we write $\tria \preceq \tria^*$ ($\tria \prec \tria^*$) if $\tria^*$ is a (strict) refinement of $\tria$.
For $\tria, \tria^* \in \bbT$, we will denote the smallest common refinement of $\tria$ and $\tria^*$ as $\tria \oplus \tria^*$. It is a triangulation in $\bbT$, and
$$
\# \tria \oplus \tria^*  \leq \# \tria + \#\tria^*- \# \tria_\bot.
$$

For any collection $\omega$ of triangles, let $\nodes(\omega)$ the set of vertices of $T \in \omega$.
For $ \tria \in \bbT$ and $z \in \nodes(\tria)$, let $\phi_z=\phi_{\tria,z}$ denote the continuous piecewise linear function w.r.t. $\tria$ that satisfies $\phi_{z}(z')=\delta_{z z'}$ ($z' \in \nodes(\tria)$).
We denote by $\Gamma(\tria)$ the set of all edges of $\tria$ that are not on $\partial\Omega$.
We set $\omega_z=\omega_{\tria,z}:=\supp \phi_z$, and let $\Gamma(\omega_z)$ denote the collection of edges of $\tria$ that are not on $\partial \omega_z$.

For $\tria \in \bbT$ and $\cM \subset \nodes(\tria)$, we let
$$\refine(\tria,\cM)$$
denote the procedure that produces the smallest triangulation in $\bbT$ in which for any $z \in \cM$ any $\tria \ni T \subset \omega_z$ has been replaced by at least four subtriangles.
The following theorem is an easy consequence of \cite[Thm. 2.4]{21}.

\begin{theorem} \label{bdd} Let $(\tria_k)_{k \geq 0}$  defined by $\tria_0=\tria_\bot$ and $\tria_{k+1} :=\refine(\tria_k,\cM_k)$ for some $\cM_k \subset \nodes(\tria_k)$. Then
$$
\# \tria_k -\# \tria_\bot \lesssim  \sum_{j=0}^{k-1} \# \cM_j.
$$
\end{theorem}

\subsection{A posteriori error estimation for the `inner' elliptic problem} \label{apost_inner}
Standard a posteriori error estimation for the Poisson problem requires the forcing function to be in $L_2(\Omega$). Our problem \eqref{ellip} does not satisfy this condition because of its second forcing term.
We will therefore use results from \cite{45.47} about a posteriori error estimation for general forcing functions in $H^{-1}(\Omega)$, and their implementable specializations to forcing functions of types
$v \mapsto \int_\Omega h v \,dx$ and $v \mapsto \int_\gamma h v \,ds$ where, for some $p>1$, $h \in L_p(\Omega)$ or $h \in L_p(\gamma)$, respectively.
In view of our application, however, for simplicity we consider the case $p=2$ only.

For $\tria \in \bbT$, we set $\UU_\tria:=\{w \in H_0^1(\Omega) \colon w|_T \in \cP_1(T)\}$.
We let
$$
\solve(\tria,\ff,\lappie)
$$
denote the procedure that computes the Galerkin approximation $u^{\lappie}_\tria$ from $\UU_\tria$ to the solution $u^{\lappie}$ of \eqref{ellip} .
For $U \in \UU_\tria$, $z \in \nodes(\tria)$, we set
\begin{align*}
j(U,\tria,z) &:= \Big(\sum_{e \in \Gamma(\omega_{\tria,z})} |e|^2 \llbracket \nabla U \cdot {\bf n}_e \rrbracket^2\Big)^{\frac{1}{2}},\\
d_\Omega(\ff,\tria,z)&:= \Big(|\omega_{\tria,z}| \int_\Omega |\ff|^2 \phi_{\tria,z} \,d x\Big)^{\frac{1}{2}},\\
d_\gamma(\lappie,\tria,z)&:= \Big(|\omega_{\tria,z}|^{\frac{1}{2}} \int_\gamma |\lappie|^2 \phi_{\tria,z} \,ds\Big)^{\frac{1}{2}},\\
e(U,\ff,\lappie,\tria,z)&:=\Big(j(U,\tria,z)^2+d_\Omega(\ff,\tria,z)^2+d_\gamma(\lappie,\tria,z)^2\Big)^{\frac{1}{2}},
\end{align*}
where $\llbracket \nabla U \cdot {\bf n}_e \rrbracket$ denotes the jump in the normal derivative of $U$ over $e$, $|e|:=\meas(e)$, and $|\omega_{\tria,z}|:=\max_{\tria \ni T \subset \omega_z} \meas(T)$.
For $\cM \subset \nodes(\tria)$ we set
\begin{align} \nonumber
J(U,\tria,\cM)&:=\big(\sum_{z \in \cM} j(U,\tria,z)^2\big)^{\frac{1}{2}}\\ \nonumber
\cD_\Omega(\ff,\tria,\cM)&:=\big(\sum_{z \in \cM} d_\Omega(\ff,\tria,z)^2\big)^{\frac{1}{2}},\\ \nonumber
\cD_\gamma(\lappie,\tria,\cM)&:=\big(\sum_{z \in \cM} d_\gamma(\lappie,\tria,z)^2\big)^{\frac{1}{2}},\\ \nonumber
\cD(\ff,\lappie,\tria,\cM)&:=\big(\cD_\Omega(\ff,\tria,\cM)^2+\cD_\gamma(\lappie,\tria,\cM)^2\big)^{\frac{1}{2}},\\ \label{ellipt_est}
\cE(U,\ff,\lappie,\tria,\cM&):=\big(\sum_{z \in \cM} e(U,\ff,\lappie,\tria,z)^2\big)^{\frac{1}{2}}.
\end{align}
In the last five notations, we will sometimes drop the argument $\cM$ from the left hand side in case it is equal to $\nodes(\tria)$.
In the last notation, sometimes we drop the argument $U$ at both sides in case it is equal to $u^{\lappie}_\tria$.

Finally, we set
$$
\err(f,\chi,\tria):=\big(|u^{\lappie}-u^{\lappie}_\tria|_{H^1(\Omega)}^2+\cD(\ff,\lappie,\tria)^2\big)^{\frac{1}{2}},
$$
which is sometimes called the \emph{total error}.
At a number places it will be used that $u^{\lappie}_\tria$ is the best approximation to $u^{\lappie}$ from $\UU_\tria$ w.r.t. semi-norm $|\cdot|_{H^1(\Omega)}$.

\begin{remark}\label{r:area}
Since neighboring triangles in $\tria \in \bbT$ have uniformly comparable sizes, and the valence of any $z \in \nodes(\tria)$ is uniformly bounded, it holds that $|\omega_{\tria,z}| \eqsim \meas(\omega_{\tria,z})$.  In \cite{45.47} the last expression is taken as the definition of $|\omega_{\tria,z}|$.
We have chosen for the current definition of $|\omega_{\tria,z}|$ because of its property that 
for $\cM \subset \nodes(\tria)$, $\bbT \ni \tria^* \succeq \refine(\tria,\cM)$,
$z \in \cM$, and $z^* \in \nodes(\tria^*)$ with $\omega_{\tria^*,z^*} \subset \omega_{\tria,z}$, it holds that 
$|\omega_{\tria^*,z^*}| \leq\frac{1}{4} |\omega_{\tria,z}|$, which will be used to demonstrate Lemma~\ref{estimatorreduction}.
(In contrast, note that under these premises, for $z \in \partial\Omega$ it is possible 
that $\meas(\omega_{\tria^*,z^*})=\meas(\omega_{\tria,z})$).
\end{remark}

Given $\tria \in \bbT$, $U \in \UU_\tria$, $\ff \in L_2(\Omega)$, and $\lappie \in L_2(\gamma)$, we let
$$
\estimate(U,\ff,\lappie,\tria)
$$
denote the procedure that computes $(e(U,\ff,\lappie,\tria,z))_{z \in \nodes(\tria)}$. 

In view of \eqref{ellip} setting $h(v):=\int_\Omega \ff v \, dx+\int_\gamma \lappie v \, ds$, from applications of Sobolev's embedding theorem and Poincar\'{e}'s inequality one may infer that
\be \label{1}
\|h\|_{H^{-1}(\omega_z)}:=\sup_{0 \neq v \in H^1_0(\omega_z)}\frac{h(v)}{|v|_{H^1(\omega_z)}} \lesssim \Big(d_\Omega(\ff,\tria,z)^2+d_\gamma(\lappie,\tria,z)^2\Big)^{\frac{1}{2}}
\ee
(cf. \cite[Sect.~7.1]{45.47}).

With the forcing term in \eqref{ellip} reading as an \emph{arbitrary} $h \in H^{-1}(\Omega)$, and denoting the resulting solution simply by $u$, the following two lemmas were shown in \cite{45.47}:

\begin{lemma}[{\cite[Lemma~3.2]{45.47}}, localized upper bound] For $\tria \preceq \tria^* \in \bbT$, it holds that
$$
|u_{\tria^*}-u_\tria|_{H^1(\Omega)} \lesssim \Big(  \sum_{z \in \nodes(\tria\setminus \tria^*)} j(u_\tria,\tria,z)^2 + \|h\|_{H^{-1}(\omega_z)}^2\Big)^{\frac{1}{2}},
$$
and so in particular
$$
|u-u_\tria|_{H^1(\Omega)} \lesssim \Big(  \sum_{z \in \nodes(\tria)} j(u_\tria,\tria,z)^2 + \|h\|_{H^{-1}(\omega_z)}^2\Big)^{\frac{1}{2}}.
$$
\end{lemma}

\begin{lemma}[{\cite[Lemma~3.3]{45.47}}, local lower bound] \label{lower} For $\tria \in \bbT$, $z \in \nodes(\tria)$, $U \in \UU_\tria$, it holds that
$$
j(U,\tria,z) \lesssim |u-U|_{H^1(\omega_z)}+\|h\|_{H^{-1}(\omega_z)}.
$$
\end{lemma}

Returning to our specific $h(v)=\int_\Omega \ff v \, dx+\int_\gamma \lappie v \, ds$, from \eqref{1} and the previous two lemmas we infer the following two results:
\begin{lemma}[localized upper bound]  \label{localupp} There exists a constant $\Upper$ such that for $\tria \preceq \tria^* \in \bbT$, it holds that
$$
|u^{\lappie}_{\tria^*}-u^{\lappie}_\tria|_{H^1(\Omega)} \leq \Upper \cE(\ff,\lappie,\tria,\nodes(\tria\setminus \tria^*)),
$$
and so in particular,
$$
|u^{\lappie}-u^{\lappie}_\tria|_{H^1(\Omega)} \leq \Upper \cE(\ff,\lappie,\tria).
$$
\end{lemma}

\begin{lemma}[global lower and upper bounds] \label{equiv} There exists a constant $\Lower>0$ such that for $\tria \in \bbT$
$$
\Lower \cE(\ff,\lappie,\tria) \leq 
\err(\ff,\lappie,\tria) \leq \sqrt{(\Upper^2+1)}\, \cE(\ff,\lappie,\tria).
$$
\end{lemma}

\subsection{Contraction property}
Further results about the a posteriori estimator established in \cite{45.47} will be combined with standard arguments in adaptive finite element theory to show that a weighted sum of the squared error in the Galerkin solution and the squared error estimator contracts when employing bulk chasing.

Whereas the adaptive finite element method investigated in \cite{45.47} involves an inner loop to reduce data oscillation, this loop will be avoided in our adaptive method.

\begin{lemma}[stability of the jump estimator] \label{stability} There exists a constant $\stab$ such that for $\tria \in \bbT$, $U,W \in \UU_\tria$, it holds that
$$
|J(U,\tria)-J(W,\tria)| \leq \stab |U-W|_{H^1(\Omega)}.
$$
\end{lemma}

\begin{proof} Application of triangle inequalities shows that $|J(U,\tria)-J(W,\tria)|\leq J(U-W,\tria)$. Now the result follows from an application of Lemma~\ref{lower} with `$h$'$=0$, and thus `$u$'$=0$, and `$U$'=$U-W$.
\end{proof}

The next lemma shows reduction of the estimator when employing bulk chasing under the unrealistic assumption that the discrete solution does not change. 
This assumption will be removed later.

\begin{lemma} \label{estimatorreduction} For $\tria \in \bbT$, $\cM \subset \nodes(\tria)$, $U \in \UU_\tria$, and $\bbT \ni \tria^* \succeq \refine(\tria,\cM)$, it holds that
$$
\cE(U,\ff,\lappie,\tria^*)^2 \leq  \cE(U,\ff,\lappie,\tria)^2-  \frac{1}{2}\cE(U,\ff,\lappie,\tria,\cM)^2.
$$
Furthermore, for $\bbT \ni \tria^* \succeq \tria$, it holds that $\cD(\ff,\lappie,\tria^*) \leq \cD(\ff,\lappie,\tria)$.
\end{lemma}

\begin{proof} For convenience of the reader we collect the arguments for these statement from the proofs of \cite[Lemmas~4.1,~7.1, and Theorem 7.5]{45.47}.

Since the normal derivative of $U$ exhibits jumps only on inter-element boundaries of $\tria$, and the latter belong to exactly two $\omega_z$'s for $z \in \nodes(\tria)$, we have
$$
J(U,\tria^*)^2=2\sum_{e \in \Gamma(\tria)} \Big(\sum_{\{e^* \in \Gamma(\tria^*)\colon e^* \subset e\}} |e^*|^2\Big) \llbracket \nabla U \cdot {\bf n}_e \rrbracket^2.
$$
On the other hand, we have
$$
J(U,\tria)^2=2\sum_{e \in \Gamma(\tria)} |e|^2 \llbracket \nabla U \cdot {\bf n}_e \rrbracket^2.
$$
For any $e \in \Gamma(\tria)$ we have $\sum_{\{e^* \in \Gamma(\tria^*)\colon e^* \subset e\}} |e^*|^2 \leq |e|^2$. Since for $e \in \Gamma(\omega_z)$ for some $z \in \cM$, $\sum_{\{e^* \in \Gamma(\tria^*)\colon e^* \subset e\}} |e^*|^2 \leq \frac{1}{2}|e|^2$, one infers that
\be \label{3}
J(U,\tria^*)^2 \leq \frac{1}{2} J(U,\tria^*,\cM)^2+ J(U,\tria^*,\nodes(\tria)\setminus \cM)^2.
\ee

Next we consider the data oscillation estimators.
Since $\phi_{\tria,z}=\sum_{z^* \in \nodes(\tria^*)} \phi_{\tria,z}(z^*) \phi_{\tria^*,z^*}$, $\sum_{z \in \nodes(\tria)} \phi_{\tria,z}(z^*)=1$ for any $z^*$,
$\phi_{\tria,z} \geq 0$, and
$\phi_{\tria,z}(z^*) \neq 0$ only if $\omega_{\tria^*,z^*} \subset \omega_{\tria,z}$, we have
\be \label{4}
\begin{split}
\hspace*{-3em}\cD_\Omega(\ff,\tria^*)^2=&\sum_{z^* \in \nodes(\tria^*)} |\omega_{\tria^*,z^*}| \int_\Omega |\ff|^2 \phi_{\tria^*,z^*}\,dx\\
 =& \sum_{z^* \in \nodes(\tria^*)} \sum_{z \in \nodes(\tria)} \phi_{\tria,z}(z^*) |\omega_{\tria^*,z^*}| \int_\Omega |\ff|^2 \phi_{\tria^*,z^*}\,dx\\
 =& \sum_{z \in \nodes(\tria)} \sum_{\{z^* \in \nodes(\tria^*)\colon \omega_{\tria^*,z^*} \subset \omega_{\tria,z}\}} \phi_{\tria,z}(z^*) |\omega_{\tria^*,z^*}| \int_\Omega |\ff|^2 \phi_{\tria^*,z^*}\,dx\\
 \leq & \frac{1}{4} \sum_{z \in \cM } |\omega_{\tria,z}| \int_\Omega |\ff|^2 \sum_{z^* \in \nodes(\tria^*)}\phi_{\tria,z}(z^*) \phi_{\tria^*,z^*}\,dx
\\ &+\sum_{z \in \nodes(\tria) \setminus \cM } |\omega_{\tria,z}| \int_\Omega |\ff|^2 \sum_{z^* \in \nodes(\tria^*)} \phi_{\tria,z}(z^*) \phi_{\tria^*,z^*}\,dx\\
=& \frac{1}{4} \sum_{z \in \cM } |\omega_{\tria,z}| \int_\Omega |\ff|^2 \phi_{\tria,z}\,dx
+\sum_{z \in \nodes(\tria) \setminus \cM } |\omega_{\tria,z}| \int_\Omega |\ff|^2 \phi_{\tria,z}\,dx\\
=& \frac{1}{4} \cD_\Omega(\ff,\tria,\cM)^2+\cD_\Omega(\ff,\tria,\nodes(\tria) \setminus \cM)^2.
\end{split}
\ee
Notice that we used our definition of $ |\omega_{\tria^*,z^*}|$, see Remark~\ref{r:area}, to obtain the above inequality.

Since exactly the same arguments show that
\be \label{6}
\cD_\gamma(\lappie,\tria^*)^2 \leq \frac{1}{2} \cD_\gamma(\lappie,\tria,\cM)^2+\cD_\gamma(\lappie,\tria,\nodes(\tria) \setminus \cM)^2,
\ee
and combining the latter with \eqref{3} and \eqref{4} completes the proof of the first statement.

The second statement is an easy consequence of \eqref{4} and \eqref{6} for $\cM=\emptyset$.
\end{proof}

For $(e_z)_{z \in \nodes(\tria)} \subset \R$ and $\theta \in (0,1]$, we let
$$
\cM:=\Mark((e_z)_{z \in \nodes(\tria)} ,\theta)
$$
denote the procedure that outputs a smallest $\cM \subset \nodes(\tria)$ that satisfies the {\em bulk chasing condition} $\sum_{z \in \cM} e_z^2 \geq \theta^2 \sum_{z \in \nodes(\tria)} e_z^2$.

\begin{corollary}[contraction] \label{contraction} Given a constant $\theta \in (0,1]$, there exists constants $\upsilon>0$ and $\alpha<1$ such that for $\tria \in \bbT$, $\cM:=\Mark((e(\ff,\lappie,\tria,z)_{z \in \nodes(\tria)}, \theta)$,
and $\bbT \ni \tria^* \succeq \refine(\tria,\cM)$, it holds that
$$
|u^{\lappie}-u^{\lappie}_{\tria^*}|_{H^1(\Omega)}^2+ \upsilon \cE(\ff,\lappie,\tria^*)^2 \leq \alpha \Big(|u^{\lappie}-u^{\lappie}_\tria|_{H^1(\Omega)}^2+ \upsilon \cE(\ff,\lappie,\tria)^{2}\Big) .
$$
\end{corollary}

\begin{proof}
This proof follows the arguments introduced in \cite{35.936}.

Applications of Lemma~\ref{stability} and that of Young's inequality show that for any $\delta>0$,
$$
\cE(\ff,\lappie,\tria^*)^2 \leq (1+\delta) \cE(u^{\lappie}_\tria,\ff,\lappie,\tria^*)^2+(1+\delta^{-1})\stab|u^{\lappie}_{\tria^*}-u^{\lappie}_\tria|_{H^1(\Omega)}^2.
$$
Using that $\cE(u^{\lappie}_\tria,\ff,\lappie,\tria^*)^2 \leq (1-\frac{1}{2}\theta^2)  \cE(\ff,\lappie,\tria)^2$ by Lemma~\ref{estimatorreduction},
choosing $\delta$ such that $(1+\delta)(1-\frac{1}{2}\theta^2)= (1-\frac{1}{4}\theta^2)$,
using that
$$
|u^{\lappie}-u^{\lappie}_{\tria^*}|_{H^1(\Omega)}^2=|u^{\lappie}-u^{\lappie}_\tria|_{H^1(\Omega)}^2-|u^{\lappie}_{\tria^*}-u^{\lappie}_\tria|_{H^1(\Omega)}^2,
$$
and taking $\upsilon$ such that $\upsilon (1+\delta^{-1})\stab =1$, we find that
\begin{align*}
|u^{\lappie}-u^{\lappie}_{\tria^*}|_{H^1(\Omega)}^2+ &\upsilon \cE(\ff,\lappie,\tria^*)^2 \leq 
|u^{\lappie}-u^{\lappie}_\tria|_{H^1(\Omega)}^2+ \upsilon (1-\frac{1}{4}\theta^2)\cE(\ff,\lappie,\tria)^2\\
&\leq \Big(1-\frac{\theta^2/4}{1+\Upper/\upsilon}\Big)\Big(|u^{\lappie}-u^{\lappie}_\tria|_{H^1(\Omega)}^2+ \upsilon \cE(\ff,\lappie,\tria)^2\Big)
\end{align*}
by an application of Lemma~\ref{localupp}.
\end{proof}

\subsection{Convergence with the best possible rate}
For $s>0$ we define the approximation class $\cA^s$ as the collection of $w \in H^1_0(\Omega)$ for which
$$
|w|_{\cA^s}:=\sup_{N \in \N} N^s \min_{\{\tria\in \bbT \colon \#\tria-\#\tria_{\bot}\leq N\}} |w-w_\tria|_{H^1(\Omega)} <\infty.
$$
Classical estimates show that for $s \leq \frac{1}{2}$ , $H^1_0(\Omega) \cap H^{1+2s}(\Omega) \subset \cA^s$ where it is sufficient to consider uniform refinements of $\tria_{\bot}$.
Obviously the class $\cA^s$ contains many more
functions, which is the reason to consider adaptive methods in the first
place. 
As shown in \cite{22}, for $s \in (0,\frac{1}{2}]$, the Besov space $B^{1+2s}_{\tau,q}(\Omega)$ is contained in $\cA^s$ for any $q>0$, $\tau>(s+\frac{1}{2})^{-1}$.
Although $\cA^s$ is non-empty for any $s>0$ as it contains $\UU_\tria$ for any $\tria \in \bbT$, even for $C^\infty(\Omega)$-functions only for $s \leq \frac{1}{2}$ membership in $\cA^s$ is guaranteed.
For that reason, it is no real restriction to consider only $s \in (0,\frac{1}{2}]$ in the following.

Besides the approximated classes $\cA^s$, we need approximation classes for both data terms of the inner elliptic problem \eqref{ellip}.
For $\ff \in L_2(\Omega)$ and $s>0$, we say that $\ff \in \cB^s_\Omega$ when
$$
|\ff|_{\cB_\Omega^s}:=\sup_{N \in \N} N^s \min_{\{\tria\in \bbT \colon \#\tria-\#\tria_{\bot}\leq N \}} \cD_\Omega(\ff,\tria) <\infty.
$$
Similarly, for $\lappie \in L_2(\gamma)$, we say that $\lappie \in \cB^s_\gamma$ when
$$
|\lappie|_{\cB_\gamma^s}:=\sup_{N \in \N} N^s \min_{\{\tria\in \bbT \colon \#\tria-\# \tria_{\bot}\leq N\}} \cD_\gamma(\lappie,\tria) <\infty.
$$
The approximation classes $\cB_\Omega^s$ and $\cB_\gamma^s$ for the data should not be confused with Besov spaces.

The next, crucial result shows that the data oscillation terms $\cD_\Omega(\ff,\tria)$ and $\cD_\gamma(\lappie,\tria)$ can be reduced at rate $\frac{1}{2}$.
Knowing this result, standard arguments introduced in \cite{249.86} will show that the usual adaptive finite element method driven by bulk chasing on the estimator $\cE$ converges with the best possible rate $s \in (0,\frac{1}{2}]$.

\begin{theorem}[{\cite[Theorems 7.3 and 7.4]{45.47}}] \label{rhsclasses} Functions $\ff \in L_2(\Omega)$ and $\lappie \in L_2(\gamma)$ are in $\cB^{\frac{1}{2}}_\Omega$ and $\cB^{\frac{1}{2}}_\gamma$, respectively, with $|\ff|_{\cB_\Omega^{\frac{1}{2}}} \lesssim \|\ff\|_{L_2(\Omega)}$ and $|\lappie|_{\cB_\gamma^{\frac{1}{2}}} \lesssim \|\lappie\|_{L_2(\gamma)}$, only dependent on $\tria_{\bot}$ and, for the second case, the length of $\gamma$.
\end{theorem}

The next lemma will be the key to bound the minimal number of nodes needed to satisfy the bulk chasing criterion, as it  is realized by the routine $\Mark$.
It shows that when $\tria^*$ is a sufficiently deep refinement of $\tria$ such that its total error is less than or equal to a certain multiple of the total error on $\tria$, then 
the set of vertices of the triangles that were refined when going from $\tria$ to $\tria^*$ satisfies the bulk chasing criterion.

\begin{lemma}[bulk chasing property] \label{boundonmarkedcells}
Setting
$$
\theta_*:={\textstyle \frac{\Lower}{\sqrt{1+\Upper^2}}},
$$
for $\theta \in (0,\theta_*)$ and any $\bbT \ni \tria^* \succeq \tria$ with 
\begin{equation} \label{5}
\err(\ff,\lappie,\tria^*)^2 \leq \big[1-{\textstyle \frac{\theta^2}{\theta_*^2}}\big] \err(\ff,\lappie,\tria)^2,
\end{equation}
it holds that
$$
\cE(\ff,\lappie,\tria,\nodes(\tria\setminus \tria^*)) \geq \theta \cE(\ff,\lappie,\tria).
$$
\end{lemma}

\begin{proof} Noting that each $T \in \tria$ that contains a $z \in \nodes(\tria) \setminus \nodes(\tria \setminus \tria^*)$ is in $\tria^*$, one infers that
$$
\cD(\ff,\lappie,\tria)^2 \leq  \cD(\ff,\lappie,\tria,\nodes(\tria\setminus \tria^*))^2+\cD(\ff,\lappie,\tria^*)^2.
$$
Now from lemmas \ref{localupp} and \ref{equiv}, and the assumption on $\tria^*$, we obtain that
\begin{align*}
\theta^2 (1+\Upper^2)&\cE(\ff,\lappie,\tria)^2 \leq {\textstyle \frac{\theta^2}{\theta^2_*}} \err(\ff,\lappie,\tria)^2
\\
&\leq \err(\ff,\lappie,\tria)^2-\err(\ff,\lappie,\tria^*)^2\\
&\leq |u^{\lappie}_{\tria^*}-u^{\lappie}_\tria|_{H^1(\Omega)}^2 +\cD(\ff,\lappie,\tria,\nodes(\tria\setminus \tria^*))^2\\
&\leq (1+\Upper^2) \cE(\ff,\lappie,\tria,\nodes(\tria\setminus \tria^*))^2
\end{align*}
being the statement of the lemma.
\end{proof}

\begin{corollary} \label{c:marked}For $\theta \in (0,\theta_*)$, $u^{\lappie} \in \cA^s$ for some $s \in (0,\frac{1}{2}]$, 
$\tria \in \bbT$, and $\cM=\Mark(e(\ff,\lappie,\tria,z)_{z \in \nodes(\tria)},\theta)$, it holds that
\be \label{Mbound}
\# \cM \lesssim C_s(u^{\lappie},\ff,\lappie) \err(\ff,\lappie,\tria)^{-\frac{1}{s}},
\ee
where
\begin{equation}\label{e:Cs}
C_s(u^{\lappie},\ff,\lappie):=\big(|u^{\lappie}|_{\cA^s}^{\frac{1}{s}}+\|\ff\|_{L_2(\Omega)}^{\frac{1}{s}}+\|\lappie\|_{L_2(\gamma)}^{\frac{1}{s}}\big).
\end{equation}
\end{corollary}

\begin{proof} Since $u^{\lappie} \in \cA^s$, $\ff\in \cB_\Omega^{\frac{1}{2}}$, $\lappie\in \cB_\gamma^{\frac{1}{2}}$, there exist $\tria_u,\tria_\ff,\tria_\lappie \in \bbT$ such that
\be \label{7}
\max\Big(|u^{\lappie}-u^{\lappie}_{\tria_u}|_{H^1(\Omega)},\cD_\Omega(\ff,\tria_\ff),\cD_\gamma(\lappie,\tria_\lappie)\Big) \leq {\textstyle \sqrt{\frac{1}{3}\big[1- \frac{\theta^2}{\theta_*^2}\big] }}\,\err(\ff,\lappie,\tria)=:\hat{E},
\ee
and
$$
\#\tria_u-\#\tria_{\bot} \leq |u|_{\cA^s}^{\frac{1}{s}} \hat{E}^{-\frac{1}{s}},\,
\#\tria_\ff-\#\tria_{\bot} \leq |\ff|_{\cB_\Omega^{\frac{1}{2}}}^{\frac{1}{2}} \hat{E}^{-\frac{1}{2}},
\#\tria_\lappie-\#\tria_{\bot} \leq |\lappie|_{\cB_\gamma^{\frac{1}{2}}}^{\frac{1}{2}} \hat{E}^{-\frac{1}{2}}.
$$
Since the left hand sides of the last two inequalities are either $0$ or $\geq 1$, we also have
\be \label{8}
\#\tria_u-\#\tria_{\bot} \leq |u|_{\cA^s}^{\frac{1}{s}} \hat{E}^{-\frac{1}{s}},\,
\#\tria_\ff-\#\tria_{\bot} \leq |\ff|_{\cB_\Omega^{\frac{1}{2}}}^{\frac{1}{s}} \hat{E}^{-\frac{1}{s}},
\#\tria_\lappie-\#\tria_{\bot} \leq |\lappie|_{\cB_\gamma^{\frac{1}{2}}}^{\frac{1}{s}} \hat{E}^{-\frac{1}{s}}.
\ee
From \eqref{7} and the monotonicity of $\cD(\ff,\lappie,\tria)$ and $|u^{\lappie}-u^{\lappie}_\tria|_{H^1(\Omega)}$ as function of $\tria$, it follows that
$\tria^*:=\tria\oplus\tria_u\oplus\tria_\ff\oplus\tria_\lappie$ satisfies \eqref{5}. 
In view of the bulk chasing property given by Lemma~\ref{boundonmarkedcells}, and because $\cM$ is a set of minimal cardinality that realizes the bulk chasing criterion, we infer that
\begin{align*}
\#\cM \leq \#\nodes(\tria\setminus \tria^*) \lesssim \#( \tria\setminus \tria^*) \leq \#\tria^* -\# \tria
\\
\leq \#\tria_u-\#\tria_{\bot} +\#\tria_\ff-\#\tria_{\bot}+\#\tria_\lappie-\#\tria_{\bot}
\end{align*}
where the third inequality is a consequence of the fact that each $T \in \tria \setminus \tria^*$ has been bisected at least once.
Now from \eqref{8}, Theorem~\ref{rhsclasses}, and
\be \label{extra}
\hat{E}^{-\frac{1}{s}}= \Big({\textstyle \sqrt{\frac{1}{3}\big[1- \frac{\theta^2}{\theta_*^2}\big]}\Big)^{-\frac{1}{s}}}\err(\ff,\lappie,\tria)^{-\frac{1}{s}} \eqsim \err(\ff,\lappie,\tria)^{-\frac{1}{s}},
\ee
the proof is completed.
\footnote{Noting that $\Big({\textstyle \sqrt{\frac{1}{3}\big[1- \frac{\theta^2}{\theta_*^2}\big]}\Big)^{-\frac{1}{s}}} \rightarrow \infty$ if, and only if, $\theta \rightarrow \theta_*$ or $s\rightarrow 0$, we conclude that
the constant `hidden' in the $\lesssim$-symbol in \eqref{extra},  and thus in  \eqref{Mbound},  depends on the value of $\theta$ or $s$ when they tend to $\theta_*$ or $0$, respectively. Consequently, this holds true for all results that are going to derived from Corollary~\ref{c:marked}.}
 \end{proof}

The next result guarantees that the nested sequence $(\tria_k)_{k}$ produced by
this adaptive finite element method reduces the total error at the best possible rate.

\begin{theorem}[convergence with optimal rate] \label{rate}
Let $\theta \in (0,\theta_*)$, and $u^{\lappie} \in \cA^s$ for some $s \in (0,\frac{1}{2}]$.
Then with $\tria_k$ denoting the partition after $k$ iterations of the $\solve-\estimate-\Mark-\refine$ loop started with $\tria_0=\tria_\bot$, it holds that
$$
\# \tria_k -\#\tria_0 \lesssim C_s(u^{\lappie},\ff,\lappie)\err(\ff,\lappie,\tria_k)^{-\frac{1}{s}},
$$
where $C_s(u^{\lappie},\ff,\lappie)$ is given by \eqref{e:Cs}.
\end{theorem}
\begin{proof} With $\cM_i$ denoting the set of nodes that are marked in $\nodes(\tria_i)$, applications of Theorem~\ref{bdd} and Corollary~\ref{c:marked} yield
$$
\# \tria_k -\#\tria_{\bot}   \lesssim \sum_{i=0}^{k-1} \# \cM_i \lesssim C_s(u^{\lappie},\ff,\lappie) \sum_{i=0}^{k-1} \err(\ff,\lappie,\tria_i)^{-\frac{1}{s}}.
$$
Hence, the equivalence  between $\err$ and $\cE$ provided by Lemma~\ref{equiv} together with the contraction property from Corollary~\ref{contraction} imply
\begin{align*}
\# \tria_k -\#\tria_{\bot} \lesssim &
C_s(u^{\lappie},\ff,\lappie) \sum_{i=0}^{k-1} \Big(\sqrt{|u^{\lappie}-u^{\lappie}_{\tria_i}|_{H^1(\Omega)}^2+\upsilon \cE(\ff,\lappie,\tria_i)^2}\Big)^{-\frac{1}{s}} \\
& \eqsim C_s(u^{\lappie},\ff,\lappie) \Big(\sqrt{|u^{\lappie}-u^{\lappie}_{\tria_{k-1}}|_{H^1(\Omega)}^2+\upsilon \cE(\ff,\lappie,\tria_{k-1})^2}\Big)^{-\frac{1}{s}}.
\end{align*}
By invoking Lemma~\ref{equiv} again, as well as the second statement of Lemma~\ref{estimatorreduction}, we arrive at
$$
\# \tria_k -\#\tria_{\bot} \lesssim  C_s(u^{\lappie},\ff,\lappie) \err(\ff,\lappie,\tria_{k-1})^{-\frac{1}{s}}  \leq C_s(u^{\lappie},\ff,\lappie) \err(\ff,\lappie,\tria_k)^{-\frac{1}{s}}. \qedhere
$$
\end{proof}

\section{The adaptive finite element method as an inner solver in Uzawa} \label{Svaryingrhs}
We have seen that for $f \in L_2(\Omega)$, and \emph{fixed} $\chi \in L_2(\gamma)$, the adaptive finite element method for solving \eqref{ellip} converges with the best possible rate. That is, whenever $u^\chi \in \cA^s$ for some $s \in (0,\frac{1}{2}]$, the Galerkin approximations converge to $u^\chi$ with rate $s$.
Now we return to the \emph{sequence} of problems \eqref{ellip}, where $\chi$ runs over the set of all intermediate approximations $\lambda^{(j-1)}_{\sigma_i}$ of $\lambda$.
These elliptic problems have to be approximated inside the Uzawa iteration.
We aim at showing that whenever $u=u^\lambda \in \cA^s$, the sequence of all approximations that we generate inside the nested inexact preconditioned Uzawa iteration converge to $u$ with \emph{this} rate $s$. 

Therefore, it is needed to optimally  bound the number of cells selected by any call of $\Mark$ in terms of $|u|_{\cA^s}$ (and that of $\|f\|_{L_2(\Omega)}$ and $\|g\|_{H^1(\gamma)}$), instead of applying the obvious bound involving $|u^\chi|_{\cA^s}$.
Indeed with $\chi$ running over the $\lambda^{(j-1)}_{\sigma_i}$, we do not know whether these $u^\chi \in \cA^s$ 
(let alone whether $\sup_{\chi} |u^\chi|_{\cA^s} \lesssim |u|_{\cA^s}$).

In the following Lemma~\ref{lem1} we will manage to achieve this goal for calls of $\Mark$ (and thus of $\refine$, $\solve$ and $\estimate$) that are made as long as the (total) error in the current Galerkin approximation for $u^\chi$ is bounded from below by a positive constant multiple of
$|u-u^\chi|_{H^1(\Omega)} \eqsim \|\lambda -\chi\|_{H^{-\frac{1}{2}}(\gamma)}$, cf. \eqref{condition}.
Fortunately, when this condition is violated, the approximation for $u^\chi$ will be sufficiently accurate for its use inside the Uzawa iteration so that there is no need for another call of $\Mark$.
The bound on the number of cells selected by $\Mark$ from Lemma~\ref{lem1} will depend on $\|\chi\|_{L_2(\gamma)}$.
In Lemma~\ref{lem3} it will be shown that for $\chi$ running over all $\lambda^{(j-1)}_{\sigma_i}$, the norms $\|\chi\|_{L_2(\gamma)}$ will be uniformly bounded by a multiple of $\|f\|_{L_2(\Omega)}+\|g\|_{H^1(\gamma)}$.

\begin{lemma} \label{lem1}
Let $\theta \in (0,\theta_*)$, and $u \in \cA^s$ for some $s \in (0,\frac{1}{2}]$.
Then for $\chi \in L_2(\gamma)$ and
$\tria \in \bbT$ with
\be \label{condition}
\err(f,\chi,\tria) \gtrsim |u-u^\chi|_{H^1(\Omega)},
\ee 
for $\cM=\Mark(e(\ff,\chi,\tria,z)_{z \in \nodes(\tria)},\theta)$ it holds that
\be \label{result}
\# \cM \lesssim C_s(u,\ff,\chi) \err(\ff,\chi,\tria)^{-\frac{1}{s}}.
\footnote{Without the condition \eqref{condition}, $C_s(u,\ff,\chi)$ in \eqref{result} would have to be read as the undesirable factor $C_s(u^\chi,\ff,\chi)$, cf. Lemma~\ref{c:marked}.}
\ee
\end{lemma}

\begin{proof}
Since $u \in \cA^s$, $\ff\in \cB_\Omega^{\frac{1}{2}}$, $\chi\in \cB_\gamma^{\frac{1}{2}}$, there exist $\tria_u,\tria_\ff,\tria_\chi \in \bbT$ such that
\be \label{12}
\max\Big(|u-u_{\tria_u}|_{H^1(\Omega)},\cD_\Omega(\ff,\tria_\ff),\cD_\gamma(\chi,\tria_\chi)\Big) \leq \err(\ff,\chi,\tria),
\ee
and
\be \label{13}
\begin{split}
\#\tria_u-\#\tria_{\bot} &\leq |u|_{\cA^s}^{\frac{1}{s}} \err(\ff,\chi,\tria)^{-\frac{1}{s}},\\
\#\tria_\ff-\#\tria_{\bot} &\leq |\ff|_{\cB_\Omega^{\frac{1}{2}}}^{\frac{1}{s}} \err(\ff,\chi,\tria)^{-\frac{1}{2}},\\
\#\tria_\chi-\#\tria_{\bot} &\leq |\chi|_{\cB_\gamma^{\frac{1}{2}}}^{\frac{1}{s}} \err(\ff,\chi,\tria)^{-\frac{1}{2}}.
\end{split}
\ee

Let $\tria^*:=\tria_u\oplus\tria_\ff\oplus\tria_\chi$.
Then by
$$
|u^\chi-u^\chi_{\tria^*}|_{H^1(\Omega)} \leq |u^\chi-u_{\tria^*}|_{H^1(\Omega)} \leq |u-u_{\tria^*}|_{H^1(\Omega)}+|u-u^\chi|_{H^1(\Omega)},
$$
$|u-u^\chi|_{H^1(\Omega)} \lesssim \err(f,\chi,\tria)$ by assumption, and $\tria \mapsto \cD(\cdot,\cdot,\tria)$ being monotone non-increasing by Lemma~\ref{estimatorreduction},
we have $\err(\ff,\chi,\tria^*) \lesssim \err(\ff,\chi,\tria)$.

Lemma~\ref{equiv} guarantees that
$$
\sqrt{|u^\chi-u^\chi_{\tria^*}|_{H^1(\Omega)}^2+\nu \cE(f,\chi,\tria^*)^2} \eqsim \err(\ff,\chi,\tria^*).
$$
Hence, the contraction property (Corollary~\ref{contraction}) indicates that the left hand side reduces by a constant factor $\alpha<1$
by each application of the cycle $\estimate-\Mark-\refine-\solve$.
Therefore by applying a fixed, sufficiently large number of those cycles shows
 that there exists a $\breve{\tria} \in \bbT$ with $\# \breve{\tria} \lesssim \# \tria^*$ and
$$
\err(\ff,\chi,\breve{\tria})^2 \leq  \big[1-{\textstyle \frac{\theta^2}{\theta_*^2}}\big] \err(\ff,\chi,\tria)^2.
$$

For $\bar{\tria}:=\tria \oplus \breve{\tria}$, we have $\bar{\tria} \succeq \tria$ and
$\err(\ff,\chi,\bar{\tria})^2 \leq \err(\ff,\chi,\breve{\tria})^2 \leq \big[1-{\textstyle \frac{\theta^2}{\theta_*^2}}\big] \err(\ff,\chi,\tria)^2$,
so that from the bulk chasing property given by Lemma~\ref{boundonmarkedcells} combined with the minimal cardinality property of the set $\cM$,  it follows that
\begin{align*}
\#\cM & \leq \#\nodes(\tria\setminus \bar{\tria}) \lesssim \#( \tria\setminus \bar{\tria}) \leq \#\bar{\tria} -\# \tria \leq \#\breve{\tria} -\# \tria_{\bot} \\
& \lesssim \#\tria_u-\#\tria_{\bot}+ \#\tria_\ff-\#\tria_{\bot}+\#\tria_\chi-\#\tria_{\bot} \leq C_s(u,\ff,\chi) \err(\ff,\chi,\tria)^{-\frac{1}{s}},
\end{align*}
by \eqref{13}, and Theorem~\ref{rhsclasses}.
\end{proof}

Instead of adaptively solving  the elliptic problems \eqref{ellip} for $\chi=\lambda^{(j-1)}_{\sigma_i}$ for each $i$ and $j$ starting from $\tria_{\bot}$, we will use the final partition
produced for the approximation of $u^{\lambda^{(j)}_{\sigma_i}}$ as the initial partition for the approximation for
$u^{\lambda^{(j+1)}_{\sigma_i}}$ when $j<K$, and for $u^{\lambda^{(0)}_{\sigma_{i+1}}}$ otherwise.

We consider the following $\solve-\estimate-\Mark-\refine$  iteration, that starts from {\em some} given initial triangulation $\tria_0 \in \bbT$, thus \emph{not} necessarily equal to $\tria_{\bot}$, and that is completed by a stopping criterion. \medskip

\begin{algorithm}\label{a:AFEM}{\rm
\begin{algotab}
\> \\
\> $[\tria_k,\,$\=$u_{\tria_{k}}^\chi]=\afem(\tria_0,f,\chi,\eps)$: \\
\>\> $u_{\tria_{0}}^\chi= \solve(\tria_0,f,\chi)$\\
\>\> $(e(f,\chi,\tria_0,z))_{z \in \nodes(\tria_0)}=\estimate(u_{\tria_0}^\chi,f,\chi)$\\
\>\> $k=0$ \\
\>\> \texttt{while} \= $\Upper \cE(f,\chi,\tria_{k}) >\eps$ \texttt{do}\\
\>\>\> $\cM_{k}=\Mark((e(f,\chi,\tria_{k},z))_{z \in \nodes(\tria_{k})},\theta)$\\
\>\>\> $\tria_{k+1}=\refine(\tria_{k},\cM_k)$\\
\>\>\> $u_{\tria_{k+1}}^\chi= \solve(\tria_{k+1},f,\chi)$\\
\>\>\> $(e(f,\chi,\tria_{k+1},z))_{z \in \nodes(\tria_{k+1})}=\estimate(u_{\tria_{k+1}}^\chi,f,\chi)$\\
\>\>\> $k \leftarrow k+1$\\
\>\> \texttt{enddo}
\end{algotab}}
\end{algorithm}

In the following lemma, essentially it is shown that the approximations produced by $\afem$ converge to $u^\chi$ with a rate that is the best possible for approximating $u$ as long as the tolerance $\eps \gtrsim |u-u^\chi|_{H^1(\Omega)}$.

\begin{lemma} \label{lem2}
Let $\theta \in (0,\theta_*)$, $u \in \cA^s$ for some $s \in (0,\frac{1}{2}]$, $\chi \in L_2(\gamma)$,
$\tria_0 \in \bbT$, and $\eps>0$ with
$$
\eps \gtrsim |u-u^\chi|_{H^1(\Omega)}.
$$
Let
$
\tria_0 \prec \cdots \prec \tria_m \subset \bbT
$
denote the sequence of triangulations that is produced by the call $\afem(\tria_0,f,\chi,\eps)$, and for
 $0 \leq k \leq m-1$, let $\cM_k \subset \nodes(\tria_k)$ denote the sets of nodes that were marked.
Then
$$
\sum_{k=0}^{m-1} \# \cM_k \lesssim C_s(u,f,\chi) \eps^{-1/s},
$$
and $|u^\chi-u^\chi_{\tria_m}|_{H^1(\Omega)} \leq \eps$,
where $C_s(u,f,\chi)$ is given by \eqref{e:Cs}.
\end{lemma}

\begin{proof} The last statement is valid by Lemma~\ref{localupp} because the algorithm terminates as a consequence of Corollary~\ref{contraction}.

For $0 \leq k < m$, $\err(f,\chi,\tria_k) \eqsim \Upper \cE(f,\chi,\tria_k) >\eps\gtrsim |u-u^\chi|_{H^1(\Omega)}$,
where the strict inequality holds for otherwise the algorithm would have stopped at iteration $k$.
By Lemma~\ref{lem1}, we deduce that  $\# \cM_k \lesssim C_s(u,\ff,\chi) \err(\ff,\chi,\tria_k)^{-\frac{1}{s}}$.
As in the proof of Theorem~\ref{rate}, from Lemma~\ref{equiv} and Corollary~\ref{contraction} we infer that
$$
\sum_{k=0}^{m-1} \# \cM_k \lesssim C_s(u,\ff,\chi) \err(\ff,\chi,\tria_{m-1})^{-\frac{1}{s}} \leq C_s(u,\ff,\chi) \Upper^{\frac{1}{s}} \eps^{-\frac{1}{s}}.\qedhere
$$
\end{proof}

To use the results that were derived in the abstract setting discussed in Sect.~\ref{saddle}, recall that in our fictitious domain setting we have
 $\UU=H^1_0(\Omega)$, $\LL=H^{-\frac{1}{2}}(\gamma)$, and
$\{0\}=\LL_{\sigma_0} \subset \LL_{\sigma_1} \subset \cdots \subset \LL$ 
is the sequence of spaces of piecewise constant functions w.r.t. to uniform dyadically refined partitions $\sigma_1 \prec \sigma_2 \prec \cdots$ of $\gamma$. 
Since $\lambda \in L_2(\gamma)$ with $\|\lambda\|_{L_2(\gamma)} \lesssim \|f\|_{L_2(\Omega)}+\|g\|_{H^1(\gamma)}$, \eqref{10} reads as
$$
\|\lambda -\lambda_{\sigma_i}\|_{H^{-\frac{1}{2}}(\gamma)} \leq \Cst 2^{-i/2},
$$
i.e., $\zeta=\sqrt{2}$, and $\Cst=\Cst(f,g) \eqsim  \|f\|_{L_2(\Omega)}+\|g\|_{H^1(\gamma)}$.

We are now ready to use the routine $\afem$ as an inner solver in the nested inexact preconditioned Uzawa iteration. With constants $\beta$ and $K=K(M)$  as in Lemma~\ref{basic}, it reads as follows: \medskip

\begin{algorithm}
\label{a:fd}{\rm
\begin{algotab}
\> \\
\>\nipu$(f,g)$\\
\> $\lambda_{\sigma_0}^{(K)}:=0$, $\tria_{0,K}:=\tria_\bot$ \\
\> \texttt{for} \= $i=1,2,\ldots$ \texttt{do}\\
\>\> $\lambda_{\sigma_{i}}^{(0)}:=\lambda_{\sigma_{i-1}}^{(K)}$, $\tria_{i,0}:=\tria_{i-1,K}$\\
\>\> \texttt{for} \= $j=1$ \texttt{to} $K$ \texttt{do}\\
\>\>\> $[\tria_{i,j},u_{\tria_{i,j}}^{\lambda_{\sigma_i}^{(j-1)}}]:=\afem(\tria_{i,j-1},f,\lambda_{\sigma_i}^{(j-1)},\Cst \zeta^{-i})$\\
\>\>\> $\lambda_{\sigma_{i}}^{(j)}:=\lambda_{\sigma_{i}}^{(j-1)}+\beta M_{\sigma_i}^{-1} I_{\sigma_i}' (B u_{\tria_{i,j}}^{\lambda_{\sigma_i}^{(j-1)}}+g)$\\
\>\> \texttt{endfor}\\
\> \texttt{endfor}
\end{algotab}}
\end{algorithm}

In order to remove the dependence on $\chi=\lambda_{\sigma_i}^{(j-1)}$ of the upper bounds derived in Lemmas~\ref{lem1} and \ref{lem2}, we need uniform boundedness of the $\|\lambda_{\sigma_{i}}^{(j)}\|_{L_2(\gamma)}$:

\begin{lemma} \label{lem3} For the sequence $((\lambda_{\sigma_{i}}^{(j)})_{1 \leq j \leq K})_{i \geq 1}$ produced by the above algorithm it holds that $\|\lambda_{\sigma_{i}}^{(j)}\|_{L_2(\gamma)} \lesssim \Cst=\Cst(f,g)$.
\end{lemma}

\begin{proof} With $Q_{\sigma_i}$ denoting the $L_2(\gamma)$-orthogonal projector onto $\LL_{\sigma_i}$, we estimate
\begin{align*}
\|\lambda_{\sigma_{i}}^{(j)}\|_{L_2(\gamma)} & \leq \|\lambda\|_{L_2(\gamma)}  +\|\lambda-\lambda_{\sigma_{i}}^{(j)}\|_{L_2(\gamma)}\\
& \leq \|\lambda\|_{L_2(\gamma)}+ \|\lambda-Q_{\sigma_i}\lambda\|_{L_2(\gamma)}+\|Q_{\sigma_i} \lambda-\lambda_{\sigma_{i}}^{(j)}\|_{L_2(\gamma)}\\
&\leq 2 \|\lambda\|_{L_2(\gamma)}+\|Q_{\sigma_i} \lambda-\lambda_{\sigma_{i}}^{(j)}\|_{L_2(\gamma)}\\
&\lesssim 2 \|\lambda\|_{L_2(\gamma)}+2^{i/2} \|Q_{\sigma_i} \lambda-\lambda_{\sigma_{i}}^{(j)}\|_{H^{-\frac{1}{2}}(\gamma)}
\end{align*}
by the application of the inverse inequality $\|\cdot\|_{L_2(\Omega)} \lesssim 2^{i/2} \|\cdot\|_{H^{-\frac{1}{2}}(\Omega)}$ on $\LL_{\sigma_i}$ (e.g., see \cite[Thm.~4.6]{56.5}).
The proof is completed by $\|\lambda\|_{L_2(\Omega)} \lesssim \Cst=\Cst(f,g)$ and
$\|Q_{\sigma_i} \lambda-\lambda_{\sigma_{i}}^{(j)}\|_{H^{-\frac{1}{2}}(\gamma)} \leq \|(I-Q_{\sigma_i}) \lambda\|_{H^{-\frac{1}{2}}(\gamma)} +\|\lambda-\lambda_{\sigma_{i}}^{(j)}\|_{H^{-\frac{1}{2}}(\gamma)} \lesssim \Cst 2^{-i/2}$, for the second term using Lemma~\ref{basic} together with \eqref{10}. 
\end{proof}

We are ready to prove that the sequence $((u_{\tria_{i,j}}^{\lambda_{\sigma_i}^{(j-1)}})_{1\leq j \leq K})_{i \geq 1}$ converges to $u$ with the best possible rate:
\begin{theorem} \label{final} Let $\theta \in (0,\theta_*)$, $u \in \cA^s$ for some $s \in (0,\frac{1}{2}]$ and assume that $K$ is sufficiently large constant as specified in Lemma~\ref{basic}.
Then for $i \geq 1$,
\be \label{14}
\frac{\max(\|\lambda-\lambda^{(j)}_{\sigma_i}\|_{H^{-\frac{1}{2}}(\gamma)},\|u-u_{\tria_{i,j}}^{\lambda_{\sigma_i}^{(j-1)}}\|_{H^1(\Omega)})}{\|f\|_{L_2(\Omega)}+\|g\|_{H^1(\gamma)}} \lesssim  2^{-i/2},\quad(1 \leq j \leq K),
\ee
and
\be \label{18}
\# \tria_{i,j} -\#\tria_\bot \lesssim 
\Big(\Big(\frac{|u|_{\cA^s}}{\|f\|_{L_2(\Omega)}+\|g\|_{H^1(\gamma)}}\Big)^{1/s}+2\Big)
\Big(\frac{ \|u-u_{\tria_{i,j}}^{\lambda_{\sigma_i}^{(j-1)}}\|_{H^1(\Omega)}}{\|f\|_{L_2(\Omega)}+\|g\|_{H^1(\gamma)}}\Big)^{-1/s}.
\ee
\end{theorem}

\begin{proof} The first statements follow from \eqref{10} and Lemma~\ref{basic} with $u^{(i,j)}=u_{\tria_{i,j}}^{\lambda_{\sigma_i}^{(j-1)}}$ and $\zeta= \sqrt{2}$.

With the number of triangulations created inside the $\afem(\tria_{i,j-1},f,\lambda_{\sigma_i}^{(j-1)},\Cst 2^{-i/2})$ denoted as $m_{i,j-1}$, let
$\cM^{(i,{j-1})}_{0},\ldots, \cM^{(i,{j-1})}_{m_{i,{j-1}}-1}$ denote the sequence of marked cells that is generated.
Since $\|\lambda_{\sigma_i}^{(j-1)}\|_{L_2(\gamma)} \lesssim \Cst$ by Lemma~\ref{lem3}, and $\|u-u^{\lambda_{\sigma_i}^{(j-1)}}\|_{H^1(\Omega)} \eqsim \|\lambda-\lambda_{\sigma_i}^{(j-1)}\|_{H^{-\frac{1}{2}(\gamma)}}\leq \Cst 2^{-i/2}$,
Lemma~\ref{lem2} shows that
$$
\sum_{k=0}^{m_{i,{j-1}}-1} \# \cM^{(i,{j-1})}_{k} \lesssim \big(|u|_{\cA^s}^{1/s}+\|f\|_{L_2(\Omega)}^{1/s} +L^{1/s}\big) \Cst^{-1/s} (2^{i/2})^{1/s}.
$$
Now an application of Theorem~\ref{bdd}, and the fact that, thanks to the optimal preconditioning, $K$ is a constant independent of $i$, 
show that 
\begin{equation} \label{15}
\begin{split}
\# \tria_{i,j}-\#\tria_\bot &\lesssim
\sum_{\breve{\jmath}=1}^j
\sum_{k=0}^{m^{(i,{\breve{\jmath}-1})}-1} \# \cM^{(i,{\breve{\jmath}-1})}_{k}
+
\sum_{\breve{\imath}=1}^{i-1}
\sum_{\breve{\jmath}=1}^K
\sum_{k=0}^{m^{(\breve{\imath},{\breve{\jmath}-1})}-1} \# \cM^{(\breve{\imath},{\breve{\jmath}-1})}_{k}\\
& \lesssim 
\Big(\frac{|u|_{\cA^s}^{1/s}+\|f\|_{L_2(\Omega)}^{1/s}}{L^{1/s}} +1\Big)
(2^{i/2})^{1/s} \\
& \lesssim 
\big((L^{-1} |u|_{\cA^s})^{1/s} +2\big) (L^{-1} \|u-u_\tria^{\lambda_{\sigma_i}^{(j-1)}}\|_{H^1(\Omega)})^{-1/s}. 
\end{split} 
\end{equation}
\end{proof}

\begin{remark}  Theorem~\ref{final} shows that the sequence $((u_{\tria_{i,j}}^{\lambda_{\sigma_i}^{(j-1)}})_{1\leq j \leq K})_{i \geq 1}$ converges to $u$ with the best possible rate, or equivalently, that $\# \tria_{i,j}$ is of the best possible order.
The latter even holds true if we read $\# \tria_{i,j}$ as the sum of the cardinality of $\tria_{i,j}$ and that of all preceding ones starting from $\tria_\bot$. This follows from \eqref{15}, $1 \leq j \leq K$, and $\sup_{i \geq 1} \max_{1 \leq j \leq K}
 m_{i,j-1} <\infty$. The latter is a consequence of the fact that the argument $\tria=\tria_{i,j-1}$ in the call $\afem(\tria_{i,j-1},f,\lambda_{\sigma_i}^{(j-1)},\Cst 2^{-i/2})$ is such that for
 $j>1$, $|u^{\lambda^{(j-2)}_{\sigma_i}}-u^{\lambda^{(j-2)}_{\sigma_i}}_\tria|_{H^1(\Omega)} \leq \Cst 2^{-i/2}$, and for $j=0$,
 $|u^{\lambda^{(K)}_{\sigma_{i-1}}}-u^{\lambda^{(K)}_{\sigma_{i-1}}}_\tria|_{H^1(\Omega)} \leq \Cst 2^{-(i-1)/2}$, and so, by the first inequality in \eqref{14},
in both cases $\inf_{U \in \UU_\tria} |u^{\lambda^{(j-1)}_{\sigma_i}}-U|_{H^1(\Omega)} \lesssim 2^{-i/2}$. As we have seen, this means that a uniformly bounded number of iterations of $\solve-\estimate-\Mark-\refine$ suffices to obtain a Galerkin approximation to $u^{\lambda^{(j-1)}_{\sigma_i}}$ that meets the tolerance $\Cst 2^{-i/2}$.

The statement proven in this remark is the first step in a proof of optimal computational complexity of a method in which 
the exact Galerkin solutions are replaced by inexact ones, following the analysis given in \cite{249.86}.
 \end{remark}
 
 \begin{remark}{(Cost of subdividing $\gamma$)}.
For the overall computational cost of the method, the costs of the repeated updates of the approximate Lagrange multiplier as well as their evaluations when used as right hand sides of the $\afem$ algorithm need to be accounted for.
 Both are proportional to the dimension of the spaces $\dim \Lambda_{\sigma_i} \eqsim 2^i$ or equivalently to the cardinality of the underlying mesh $\#\sigma_i$.
In view of \eqref{14}, we deduce that
$\dim \Lambda_{\sigma_i} \lesssim \Cst^2 \|u-u_{\tria_{i,j}}^{\lambda_{\sigma_i}^{(j-1)}}\|_{H^1(\Omega)}^{-2}$,
which is smaller than the estimate \eqref{18} derived for $\# \tria_{i,j}$ ($s \in (0,1/2 \rbrack$). The overall computational cost is therefore dominated by the approximation of $u$ in $\afem$.
\end{remark}

\section{Numerical Illustrations} \label{Numres}
\subsection{A posteriori error estimation} 
To assess the performances of Algorithm~\ref{a:fd}, we derive a-posteriori estimators for $|u-u_{\tria_{i,K}}^{\lambda_{\sigma_i}^{(K-1)}}|_{H^1(\Omega)}$ and $\|\lambda -\lambda_{\sigma_i}^{(K-1)}\|_{H^{-\frac{1}{2}}(\gamma)}$, and report on their values.
Notice that we expect $\lambda_{\sigma_i}^{(K)}$ to be more accurate than $\lambda_{\sigma_i}^{(K-1)}$ but we cannot get a computational estimate for the error in the former. 

We start with $|u-u_{\tria_{i,K}}^{\lambda_{\sigma_i}^{(K-1)}}|_{H^1(\Omega)}$.
From \eqref{second1} in Proposition~\ref{apost}, it follows that $|u-u^{\lambda_{\sigma_i}}|_{H^1(\Omega)} \eqsim \|g-u^{\lambda_{\sigma_i}}\|_{H^{\frac12}(\gamma)}$, with
the Aronszajn-Slobodeckij norm $\|w\|_{H^{\frac12}(\gamma)}^2:=
\|w\|_{L_2(\gamma)}^2+|w|_{H^{\frac12}(\gamma)}^2$ and
$|w|_{H^{\frac12}(\gamma)}^2:=\int_\gamma \int_\gamma \frac{|w(\xi)-w(\eta)|^2}{|\xi-\eta|^2}d\xi d\eta$.

To be able to compute, or accurately approximate, the error estimator in linear complexity, we localize the double integral.
As shown by B. Faermann in \cite{72.5}, using that $g-u^{\lambda_{\sigma_i}} \perp_{L_2(\gamma)} \LL_{\sigma_i}$ 
it holds that 
$\|g-u^{\lambda_{\sigma_i}}\|_{H^{\frac12}(\gamma)} \eqsim |g-u^{\lambda_{\sigma_i}}|_{\frac{1}{2},\sigma_i,{\rm loc}}$, where
$|w|_{\frac{1}{2},\sigma_i,{\rm loc}}^2:=\sum_{I \in \sigma_i} |w|^2_{H^{\frac12}(I\cup I_R)}$
and $I_R=I_R(I) \in \sigma_i$ is the interval next to $I$ in clockwise direction.

By triangle-inequalities and the trace theorem, we arrive at
\begin{equation} \label{een}
\begin{split}
|u-u_{\tria_{i,K}}^{\lambda_{\sigma_i}^{(K-1)}}|_{H^1(\Omega)} &\leq 
|u-u^{\lambda_{\sigma_i}}|_{H^1(\Omega)}+|u^{\lambda_{\sigma_i}}-u_{\tria_{i,K}}^{\lambda_{\sigma_i}^{(K-1)}}|_{H^1(\Omega)}\\
& \eqsim |g-u^{\lambda_{\sigma_i}}|_{\frac{1}{2},\sigma_i,{\rm loc}}+|u^{\lambda_{\sigma_i}}-u_{\tria_{i,K}}^{\lambda_{\sigma_i}^{(K-1)}}|_{H^1(\Omega)}\\
& \lesssim  |g-u_{\tria_{i,K}}^{\lambda_{\sigma_i}^{(K-1)}}|_{\frac{1}{2},\sigma_i,{\rm loc}}+|u^{\lambda_{\sigma_i}}-u_{\tria_{i,K}}^{\lambda_{\sigma_i}^{(K-1)}}|_{H^1(\Omega)}.
\end{split}
\end{equation}

Let $M_{\sigma_i}$ be a preconditioner as in \eqref{precond}, $\Phi_{\sigma_i}$ be a basis for $\LL_{\sigma_i}$,
and ${\bf r}:=\langle g-u_{\tria_{i,K}}^{\lambda_{\sigma_i}^{(K-1)}},\Phi_{\sigma_i} \rangle_{L_2(\gamma)}$.
From \eqref{first1}-\eqref{first2} in Proposition~\ref{apost} we have
\begin{equation} \label{twee}
\begin{split}
 |u^{\lambda_{\sigma_i}}-u_{\tria_{i,K}}^{\lambda_{\sigma_i}^{(K-1)}}|_{H^1(\Omega)} & \leq |u^{\lambda_{\sigma_i}}-u^{\lambda_{\sigma_i}^{(K-1)}}|_{H^1(\Omega)}+|u^{\lambda_{\sigma_i}^{(K-1)}}-u_{\tria_{i,K}}^{\lambda_{\sigma_i}^{(K-1)}}|_{H^1(\Omega)}\\
& \lesssim \sqrt{\langle {\bf M}^{-1}_{\sigma_i} {\bf r},{\bf r}\rangle}+|u^{\lambda_{\sigma_i}^{(K-1)}}-u_{\tria_{i,K}}^{\lambda_{\sigma_i}^{(K-1)}}|_{H^1(\Omega)}.
\end{split}
\end{equation}

Finally, an application of Lemma~\ref{localupp} shows that
\begin{equation} \label{drie}
|u^{\lambda_{\sigma_i}^{(K-1)}}-u_{\tria_{i,K}}^{\lambda_{\sigma_i}^{(K-1)}}| \leq \cE(u_{\tria_{i,K}}^{\lambda_{\sigma_i}^{(K-1)}},f,\lambda_{\sigma_i}^{(K-1)},\tria_{i,K}).
\end{equation}

Combining \eqref{een}, \eqref{twee}, \eqref{drie}, yields the computable upper bound
\begin{equation} \label{upper}
\begin{split}
|u-&u_{\tria_{i,K}}^{\lambda_{\sigma_i}^{(K-1)}}|_{H^1(\Omega)} \leq \\
&\lesssim \underbrace{|g-u_{\tria_{i,K}}^{\lambda_{\sigma_i}^{(K-1)}}|_{\frac{1}{2},\sigma_i,{\rm loc}}}_{E_\text{outer}:=}+\underbrace{\sqrt{\langle {\bf M}^{-1}_{\sigma_i} {\bf r},{\bf r}\rangle}}_{E_\text{Uzawa}:=}+\underbrace{\cE(u_{\tria_{i,K}}^{\lambda_{\sigma_i}^{(K-1)}},f,\lambda_{\sigma_i}^{(K-1)},\tria_{i,K})}_{E_\text{inner}:=}.
\end{split}
\end{equation}

Notice that when $E_\text{Uzawa}+E_\text{inner} \lesssim E_\text{outer}$, it
even holds that
$$
|u-u_{\tria_{i,K}}^{\lambda_{\sigma_i}^{(K-1)}}|_{H^1(\Omega)} \eqsim E_\text{outer}+E_\text{Uzawa}+E_\text{inner}.
$$
Indeed, this follows from the estimate 
\begin{equation}\label{estim:outer}
E_\text{outer} =|g-u_{\tria_{i,K}}^{\lambda_{\sigma_i}^{(K-1)}}|_{\frac{1}{2},\sigma_i,{\rm loc}} \le \sqrt{2}\,
 |g-u_{\tria_{i,K}}^{\lambda_{\sigma_i}^{(K-1)}}|_{H^{\frac{1}{2}}(\gamma)} \lesssim |u-u_{\tria_{i,K}}^{\lambda_{\sigma_i}^{(K-1)}}|_{H^1(\Omega)}
\end{equation}
by the trace theorem.
 
\begin{remark} Concerning the terminology, recall that in Lemma~\ref{equiv} we have seen that the inner Galerkin error $|u^{\lambda_{\sigma_i}^{(K-1)}}-u_{\tria_{i,K}}^{\lambda_{\sigma_i}^{(K-1)}}|_{H^1(\Omega)}$ is equivalent to
$E_\text{inner}$ up to the data oscillation term $\cD(f,\lambda_{\sigma_i}^{(K-1)},\tria_{i,K})$.
Furthermore, \eqref{first1}-\eqref{first2} in Proposition~\ref{apost} show that if $E_\text{inner}/E_\text{Uzawa}$ is sufficiently small, then $\|\lambda_{\sigma_i}-\lambda_{\sigma_i}^{(K-1)}\|_{H^{-\frac{1}{2}}(\gamma)} \eqsim |u^{\lambda_{\sigma_i}}-u^{\lambda_{\sigma_i}^{(K-1)}}|_{H^1(\Omega)}$ $\eqsim E_\text{Uzawa}$, which thus is properly called the Uzawa error.
Similarly,  if additionally $E_\text{Uzawa}/E_\text{outer}$ is sufficiently small, then $E_\text{outer} \eqsim |g-u^{\lambda_{\sigma_i}}|_{\frac{1}{2},\sigma_i,{\rm loc}} \eqsim |u-u^{\lambda_{\sigma_i}}|_{H^1(\Omega)} \eqsim \|\lambda-\lambda_{\sigma_i}\|_{H^{-\frac{1}{2}}(\gamma)}$ being the outer Galerkin error.
\end{remark}
\bigskip

Proceeding with the estimate of  $\|\lambda -\lambda_{\sigma_i}^{(K-1)}\|_{H^{-\frac{1}{2}}(\gamma)}$, the Galerkin orthogonality w.r.t. the energy inner product $(\chi,\mu)\mapsto (S\mu)(\chi)$ yields
\begin{align*}
\|\lambda -\lambda_{\sigma_i}^{(K-1)}\|_{H^{-\frac{1}{2}}(\gamma)} & \eqsim
\|\lambda -\lambda_{\sigma_i}\|_{H^{-\frac{1}{2}}(\gamma)}+\|\lambda_{\sigma_i} -\lambda_{\sigma_i}^{(K-1)}\|_{H^{-\frac{1}{2}}(\gamma)}\\
& \eqsim |u-u^{\lambda_{\sigma_i}}|_{H^1(\Omega)}+|u^{\lambda_{\sigma_i}}-u^{\lambda_{\sigma_i}^{(K-1)}}|_{H^1(\Omega)} \\
 &\lesssim E_\text{outer}+E_\text{Uzawa}+E_\text{inner}.
\end{align*}
Recalling \eqref{estim:outer}, we obtain
\begin{align*}
E_\text{outer} \lesssim |u-u_{\tria_{i,K}}^{\lambda_{\sigma_i}^{(K-1)}}|_{H^1(\Omega)}
& \leq |u-u^{\lambda_{\sigma_i}^{(K-1)}}|_{H^1(\Omega)}+|u^{\lambda_{\sigma_i}^{(K-1)}}-u_{\tria_{i,K}}^{\lambda_{\sigma_i}^{(K-1)}}|_{H^1(\Omega)}\\
&\lesssim \|\lambda -\lambda_{\sigma_i}^{(K-1)}\|_{H^{-\frac{1}{2}}(\gamma)}+E_\text{inner},
\end{align*}
and infer that if $E_\text{Uzawa} \lesssim E_\text{outer}$ and $E_\text{inner} /E_\text{outer}$ is sufficiently small, then 
$$
\|\lambda -\lambda_{\sigma_i}^{(K-1)}\|_{H^{-\frac{1}{2}}(\gamma)} \eqsim E_\text{outer}+E_\text{Uzawa}+E_\text{inner}.
$$

\begin{remark} 
It is tempting to circumvent the somewhat cumbersome computation of the localized Aronszajn-Slobodeckij semi-norm $|\cdot|_{\frac{1}{2},\sigma_i,{\rm loc}}$ by the following approach: For $w \in L_1(\gamma)$,
let $P_{\sigma_i}w$ be the continuous piecewise linear function on $\gamma$ w.r.t. the partition $\sigma_i$ defined on each of its vertices $\nu$ as the average of $w$ over the union of the two elements of $\sigma_i$ that contain $\nu$.
Using that $P_{\sigma_i}$ locally preserves constants, standard techniques show that $\|P_{\sigma_i}\|_{\cL(L_2(\gamma),L_2(\gamma))} \lesssim 1$, 
$\|P_{\sigma_i}\|_{\cL(H^1(\gamma),H^1(\gamma))} \lesssim 1$, $\|I-P_{\sigma_i}\|_{\cL(H^1(\gamma),L_2(\gamma))} \lesssim 2^{-i}$, and as a consequence,
$\|P_{\sigma_i}\|_{\cL(H^{\frac12}(\gamma),H^{\frac12}(\gamma))} \lesssim 1$ and
$\|I-P_{\sigma_i}\|_{\cL(H^1(\gamma),H^{\frac12}(\gamma))} \lesssim 2^{-i/2}$.
Using the orthogonality
$g-u^{\lambda_{\sigma_i}} \perp_{L_2(\gamma)} \LL_{\sigma_i}$,
we arrive at
\begin{align*}
|u-u^{\lambda_{\sigma_i}}|_{H^1(\Omega)} &\eqsim \|g-u^{\lambda_{\sigma_i}}\|_{H^{\frac12}(\gamma)}
=\|(I-P_{\sigma_i})(g-u^{\lambda_{\sigma_i}})\|_{H^{\frac12}(\gamma)}\\
&\lesssim
2^{-i/2}\|g-u_{\tria_{i,K}}^{\lambda_{\sigma_i}^{(K-1)}}\|_{H^1(\gamma)}+
\|u^{\lambda_{\sigma_i}}-u_{\tria_{i,K}}^{\lambda_{\sigma_i}^{(K-1)}}\|_{H^{\frac12}(\gamma)},
\end{align*}
which, in view of \eqref{een}, yields
$$
\max\big(|u-u_{\tria_{i,K}}^{\lambda_{\sigma_i}^{(K-1)}}|_{H^1(\Omega)},\|\lambda -\lambda_{\sigma_i}^{(K-1)}\|_{H^{-\frac{1}{2}}(\gamma)} \big) \lesssim \tilde{E}_\text{outer}+E_\text{Uzawa}+E_\text{inner},
$$
where $\tilde{E}_\text{outer}:=2^{-i/2}\|g-u_{\tria_{i,K}}^{\lambda_{\sigma_i}^{(K-1)}}\|_{H^1(\gamma)}$.

The approach of estimating the $H^{\frac{1}{2}}(\gamma)$-norm of a residual by a weighted $H^1(\gamma)$-norm was
introduced in \cite{35.93551} and is often used in the BEM community.
In the current context, however this turns out not to be appropriate.
In our experiments the modified estimator greatly overestimates the error and it even does not reduce when the iterations proceed. 
The reason is that the trace of $u_{\tria_{i,K}}^{\lambda_{\sigma_i}^{(K-1)}}$ is piecewise polynomial w.r.t. an irregular partition of $\gamma$, that moreover is locally much finer than $\sigma_i$.
\end{remark}

\subsection{Setting}\label{S:test-1}
We explore the convergence and optimality properties of the nested inexact preconditioned Uzawa algorithm (Algorithm~\ref{a:fd}).
We consider the L-shaped domain
$\wideparen{\Omega}=(-1,1)^2 \setminus (-1,0)^2$, set $g=0$ and choose $\wideparen{f}\in L_2(\wideparen{\Omega})$ such that the solution $u$ to  \eqref{16} in polar coordinates $(r,\phi)$ centered at $(0,0)$ reads
$$\wideparen{u}(r,\phi) = h(r) r^{2/3} \sin(2/3(\phi+\pi/2)),$$
where
$$
h(r)=\frac{w(3/4-r)}{w(r-1/4)+w(3/4-r)}\quad \textrm{with}\quad w(r)=\left\{
\begin{array}{ll}
r^2 & \mbox{if}\ r>0\\0 & \mbox{else.}
\end{array}\right.
$$
The fictitious domain formulation \eqref{17} is obtained by embedding
$\wideparen{\Omega}$ in the square domain $\Omega=(-1.5,1.5)^2$ and by letting $\wideparen{f}$ to be the zero extension of $f \in L_2(\Omega)$.
Note that in that case, the solution $(u,\lambda)$ of \eqref{17} satisfies $u|_{\Omega \setminus \wideparen{\Omega}}=0$ and 
$$\lambda= \frac{\partial \wideparen{u}}{\partial \vec{n}}\big|_\gamma= \frac{2}{3} h(r) \ r^{-1/3} \in H^s(\gamma), \ s<\frac 1 6.$$

Recall that the approximations of $u$ are continuous piecewise linear polynomials w.r.t. locally refined partitions of $\Omega$ while the approximations of $\lambda$ consist of piecewise constant polynomials w.r.t uniform dyadically refined partitions $\sigma_\bot=\sigma_1 \prec \sigma_2 \prec \cdots$ of $\gamma$, where $\# \sigma_i=2^{i+2}$.
\medskip 

\subsection{Performances of the Wavelet Preconditioner}\label{ss:precon}
We start by assessing the efficiency of the wavelet preconditioner $M_{\sigma_i}^{-1}$ introduced in Example~\ref{ex1}.
It is an approximate inverse of $S_{\sigma_i} \in \Lis(\LL_{\sigma_i},\LL_{\sigma_i}')$ and its quality is characterized by a uniform bound on
\begin{equation} \label{100}
\kappa:=\sup_{i} \rho(M_{\sigma_i}^{-1} S_{\sigma_i}) \sup_{i} \rho(M_{\sigma_i} S_{\sigma_i}^{-1})
=\sup_{i} \kappa(M_{\sigma_i}^{-1} S_{\sigma_i}),
\end{equation}
where for an invertible $C$, $\kappa(C)$ is the spectral condition number defined by $\kappa(C):=\rho(C)\rho(C^{-1})$.
The equality in \eqref{100} follows from
the nesting $\LL_{\sigma_i} \subset \LL_{\sigma_{i+1}}$ and the multi-level character of the preconditioner.

Unfortunately, the exact computation of $\kappa(M_{\sigma_i}^{-1} S_{\sigma_i})$ is impossible because the evaluation of $S_{\sigma_i}$ requires the inverse of the infinite dimensional $A \in \Lis(\UU,\UU')$. 
Instead, we monitor the computable quantity $\kappa(M_{\sigma_i}^{-1} S_{\sigma_i \tau_i})$, where for a partition $\tria_i \in {\mathcal T}$ of $\Omega$,  $S_{\sigma_i \tau_i}$ is an approximation of $S_{\sigma_i}$.
We propose to define $S_{\sigma_i \tau_i}:=B_{\sigma_i \tau_i} A^{-1}_{\tau_i} B_{\sigma_i \tau_i}'$, where $B_{\sigma_i \tau_i} \in \cL(\UU_{\tau_i},\LL_{\sigma_i}')$ and $A_{\tria_i} \in \Lis(\UU_{\tria_i},\UU_{\tria_i}')$ are defined by $(B_{\sigma_i \tau_i}w)(\mu)=b(w,\mu)$ ($w \in \UU_{\tria_i}$, $\mu \in \LL_{\sigma_i}$) and $(A_{\tria_i}w)(v)=a(w,v)$ ($w,v\in \UU_{\tria_i}$), respectively. 
Given $\sigma_i$, we know that $S_{\sigma_i \tau_i} \rightarrow  S_{\sigma_i} \in \Lis(\LL_{\sigma_i},\LL_{\sigma_i}')$ when the diameter of the largest element in $\tria_i$ tends to zero.
Furthermore, $S_{\sigma_i \tau_i}$ is uniformly spectrally equivalent to $S_{\sigma_i}$ under a uniform LBB condition.
To achieve the latter, we perform refinements until the triangles $T\in \tria_i$ intersecting the boundary $\gamma$ have  diameters smaller than $3$ times the length of the elements in $\sigma_i$, see \cite{GiraultGlowinski1995}.
At this point, we emphasize that the validity of the LBB condition 
is enforced only to assess the performances of the wavelet preconditioner but is not required for the nested inexact Uzawa algorithm.

The results are collected in Table~\ref{tab1}.
In the first two columns, we
report the number of elements in $\sigma_i$ and $\tria_i$, while the third
and fourth column show the condition numbers of the Schur
complement and its preconditioned version, respectively. 
The last two columns contains the spectral radius of the preconditioned Schur
complement and that of its inverse.
As predicted, the condition number of the unpreconditioned matrices increases by a factor 2 when the level $i$ of refinement is
increased by 1. In contrast, the efficiency of the wavelet preconditioner is confirmed (fourth column) by the nearly constant values of the condition number of the preconditioned
Schur complements. The fact that these condition numbers even decrease with an increasing $\# \sigma_i$ is an artifact caused by the replacement of $A^{-1}$ by $A_{\tria_i}^{-1}$.

It is worth noting that from the quantities 
$\rho(M_{\sigma_i}^{-1} S_{\sigma_i \tau_i})$ and
$\rho(M_{\sigma_i}S_{\sigma_i \tau_i}^{-1})$ reported in Table~\ref{tab1},
it is possible to obtain an estimate for the optimal parameter $\beta$ defined by \eqref{beta.opt}. 
In fact, we observe that
$\rho(M_{\sigma_i}^{-1} S_{\sigma_i \tau_i})
+\rho(M_{\sigma_i}S_{\sigma_i \tau_i}^{-1})^{-1} \approx 0.8$ so from now on we set $\beta=2/0.8$.

\begin{table}[ht!]
\caption{Spectral condition numbers of  the preconditioned and
  unpreconditioned approximate Schur complement.
$\kappa_{S} = \kappa(S_{\sigma_i \tau_i})$,
$\kappa_{M^{-1}S} = \kappa(M_{\sigma_i}^{-1} S_{\sigma_i \tau_i})$,
$\rho_{S}=\rho(S_{\sigma_i})$,
$\rho_{M^{-1} S}=\rho(M_{\sigma_i}^{-1} S_{\sigma_i})$,
$\rho_{S^{-1}}=\rho(S_{\sigma_i}^{-1})$,
$\rho_{MS^{-1}}=\rho(M_{\sigma_i} S_{\sigma_i}^{-1})$.}
\label{tab1}
\begin{tabular}{|c|c|c|c|c|c|}\hline
$\# \sigma_i$ &
$\# \tria_i $ &
$\kappa_{S}$ &
$\kappa_{M^{-1}S}$ &
$\rho_{M^{-1}S}$ &
$\rho_{M S^{-1}}$ \\\hline
8    &  1741  & 6.71 & 6.71 & 0.563 & 11.9 \\ \hline
16   &  2010  & 13.5 & 6.44 & 0.575 & 11.2 \\ \hline
32   &  4770  & 28.0 & 6.04 & 0.587 & 10.3 \\ \hline 
64   & 11326  & 57.8 & 5.83 & 0.593 & 9.83  \\ \hline
128  & 23398  & 118 & 5.74 & 0.596 & 9.64 \\ \hline
256  & 46134  & 238 & 5.69 & 0.597 & 9.54 \\ \hline
512  & 85460  & 489 & 5.67 & 0.598 & 9.48 \\ \hline
1024 & 156092 & 980 & 5.65 & 0.598 & 9.46 \\ \hline
\end{tabular}
\end{table}

\subsection{Performances of the Nested Inexact Uzawa Algorithm}

We now investigated the performances of the nested inexact preconditioned Uzawa iteration (Algorithm~\ref{a:fd}).
The routine $\afem$ given in Algorithm~\ref{a:AFEM} serves as an inner solver in Algorithm~\ref{a:fd} and is driven by the a posteriori error estimator ${\mathcal E}$, see \eqref{ellipt_est}.
Apart from data oscillation terms, it consists of the square root of
the sum of weighted norms of jumps of normal derivatives of the
current approximation for $u$ over the edges of the partition of
$\Omega$. 
The numerical observations in \cite{37} indicate that, ignoring the data oscillations, ${\mathcal E}$ is
approximately a factor $3\sqrt{2}$ larger than the error it estimates
(the factor $\sqrt{2}$ stems from the fact that unlike in \cite{37} our estimator each jump is counted twice). 
Therefore, in the following we scale $\mathcal E$ by a factor $\sqrt{2}/6$ and set the constant $C_{\text{upp}} = 1$.
Note that the same scaling is applied to the quantity $E_\text{inner}$ defined in \eqref{upper}. 
In addition, we set the constant
$\Cst=\Cst(f,g)=\bar{\Cst}\,(\|f\|_{L_2(\Omega)}+\|g\|_{H^1(\gamma)})$
with $\bar{\Cst}=0.1$, 
$K=6$, $\zeta=\sqrt{2}$, $\theta=0.1$ and recall that $\beta$ defined in \eqref{beta.opt} is set to $\beta=2/0.8$ (see Section~\ref{ss:precon}).

Figure~\ref{F:meshes} displays the meshes $\tria_{0,K}=\tria_\bot$ (initial mesh), $\sigma_0=\emptyset$ together with the adaptively or uniformly refined meshes $\tria_{i,K}$, $\sigma_i$ obtained at the first, third
and fifth outer iteration $i=1,3,5$ of Algorithm~\ref{a:fd}.
\begin{remark} To illustrate the point made in Remark~\ref{LBB} about not imposing the LBB condition, we observe 
that for the mesh corresponding to $i=5$ in Figure~\ref{F:meshes}, the triangle that covers the lower-right corner of the L-shaped domain 
 contains 7 elements of the boundary mesh $\sigma_5$.
This implies $\inf_{\mu \in \LL_{\sigma_5}} \sup_{v \in \UU_{\tria_{5,K}}} b(v,\mu)=0$,
 so that the fully discrete saddle point problem on $\UU_{\tria_{i,K}}\times \LL_{\sigma_i}$ is even singular,
and in particular that the LBB condition does not hold.
\end{remark}
Figure~\ref{F:sols3D} shows the approximations
$u_{\tria_{i,K}}^{\lambda_{\sigma_i}^{(K-1)}}$ at the third and sixth
outer iterations $i=3,6$, while Figure~\ref{F:lambda_sol} provides a comparison between the approximation  $\lambda_{\sigma_i}^{(K-1)}$ 
and the $L_2(\gamma)$-orthogonal projection of the exact solution $\lambda$ onto $\LL_{\sigma_{i}}$  for $i=3$ and $6$. 
In Figure~\ref{F:sols} the traces of the
numerical solution $u_{\tria_{i,K}}^{\lambda_{\sigma_i}^{(K-1)}}$ on the boundary $\gamma$ are depicted for $i=1,3,5,6$ in red and
compared to the (zero)  trace of the exact solution.

\begin{figure}[ht!]
\begin{center}
\begin{overpic}[width=0.45\textwidth]%
{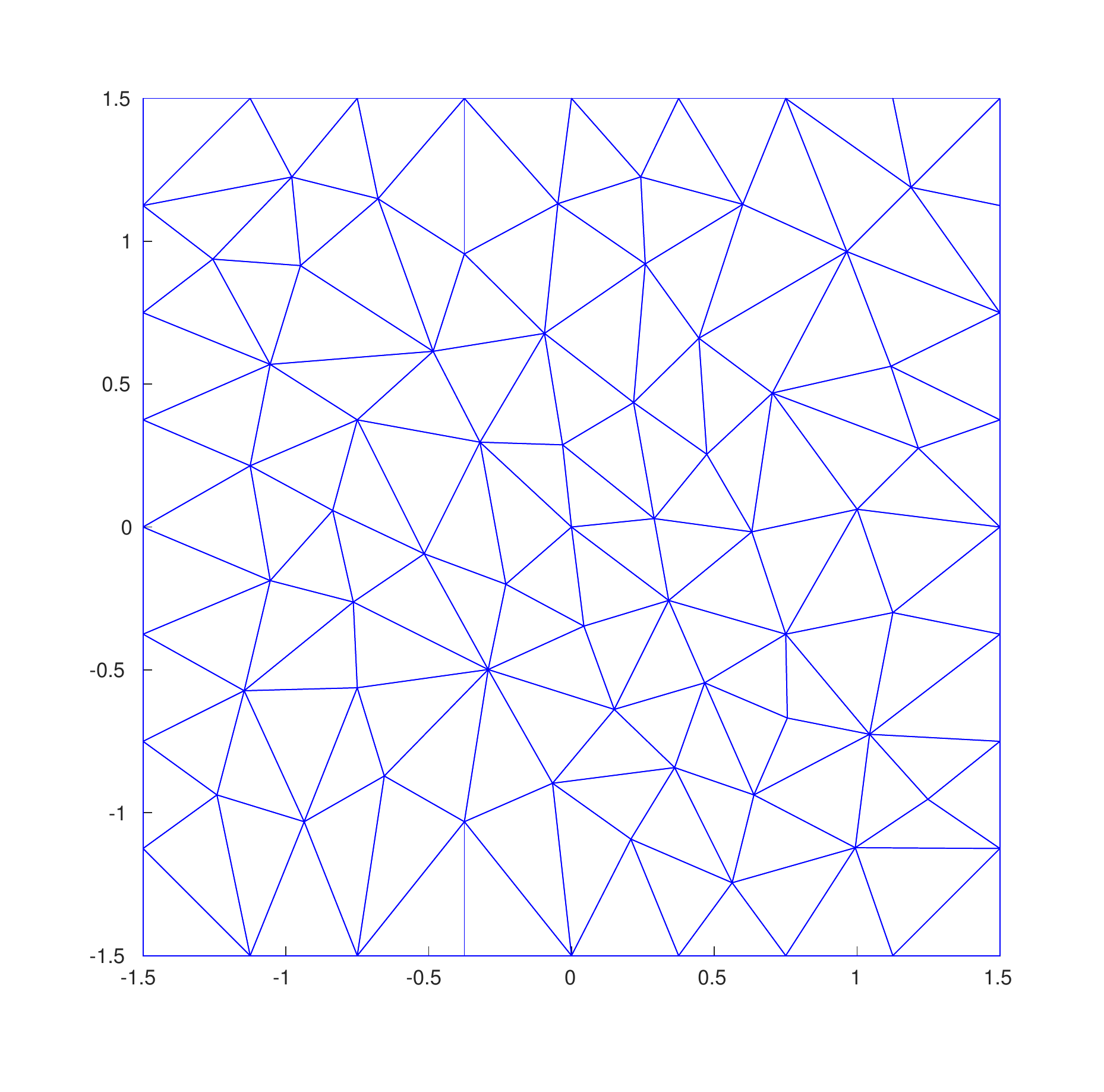}
\put(41.5,32.7){\includegraphics[width=0.235\textwidth]{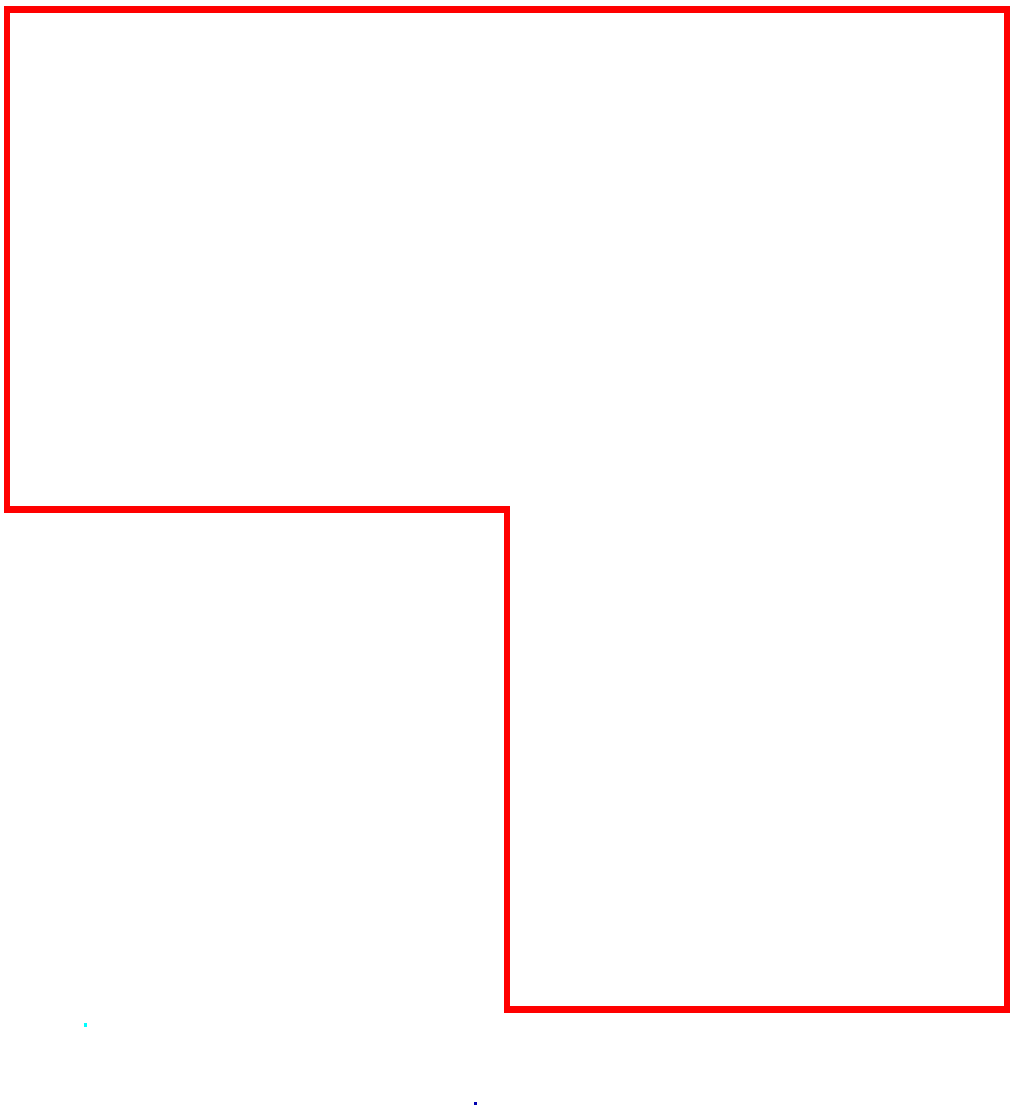}}
\end{overpic}
\begin{overpic}[width=0.45\textwidth]%
{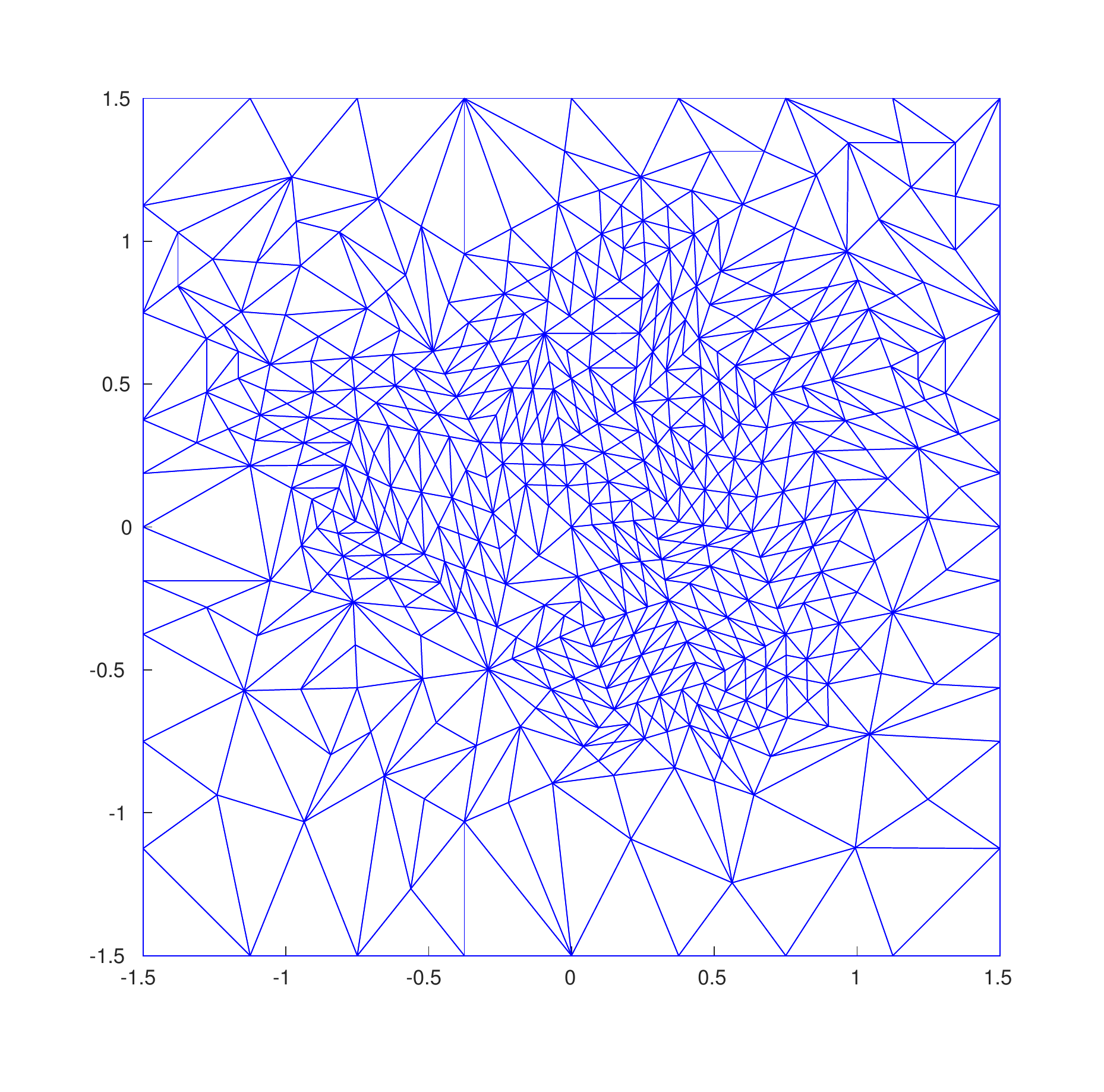}
\put(40,38.8){\includegraphics[width=0.243\textwidth]{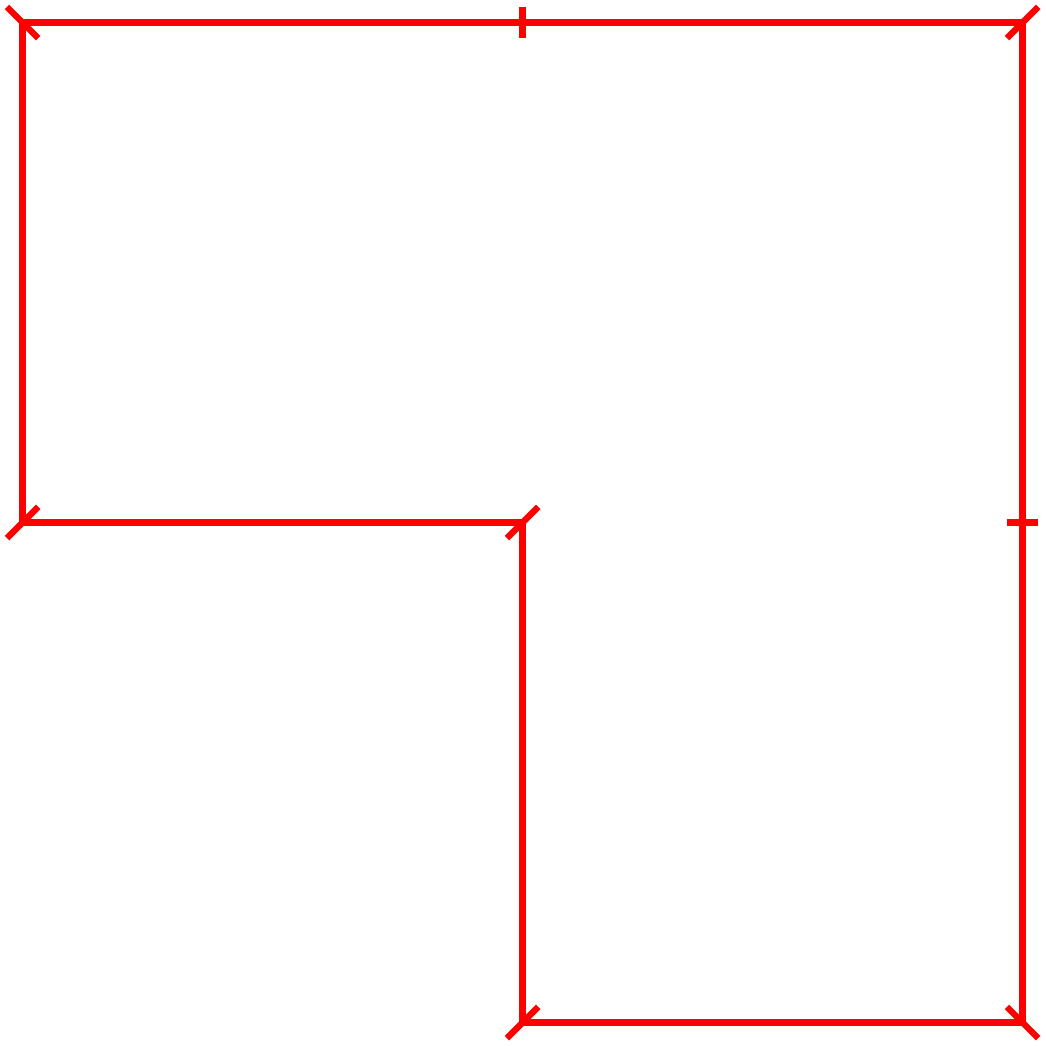}}
\end{overpic}\\
\begin{overpic}[width=0.45\textwidth]%
{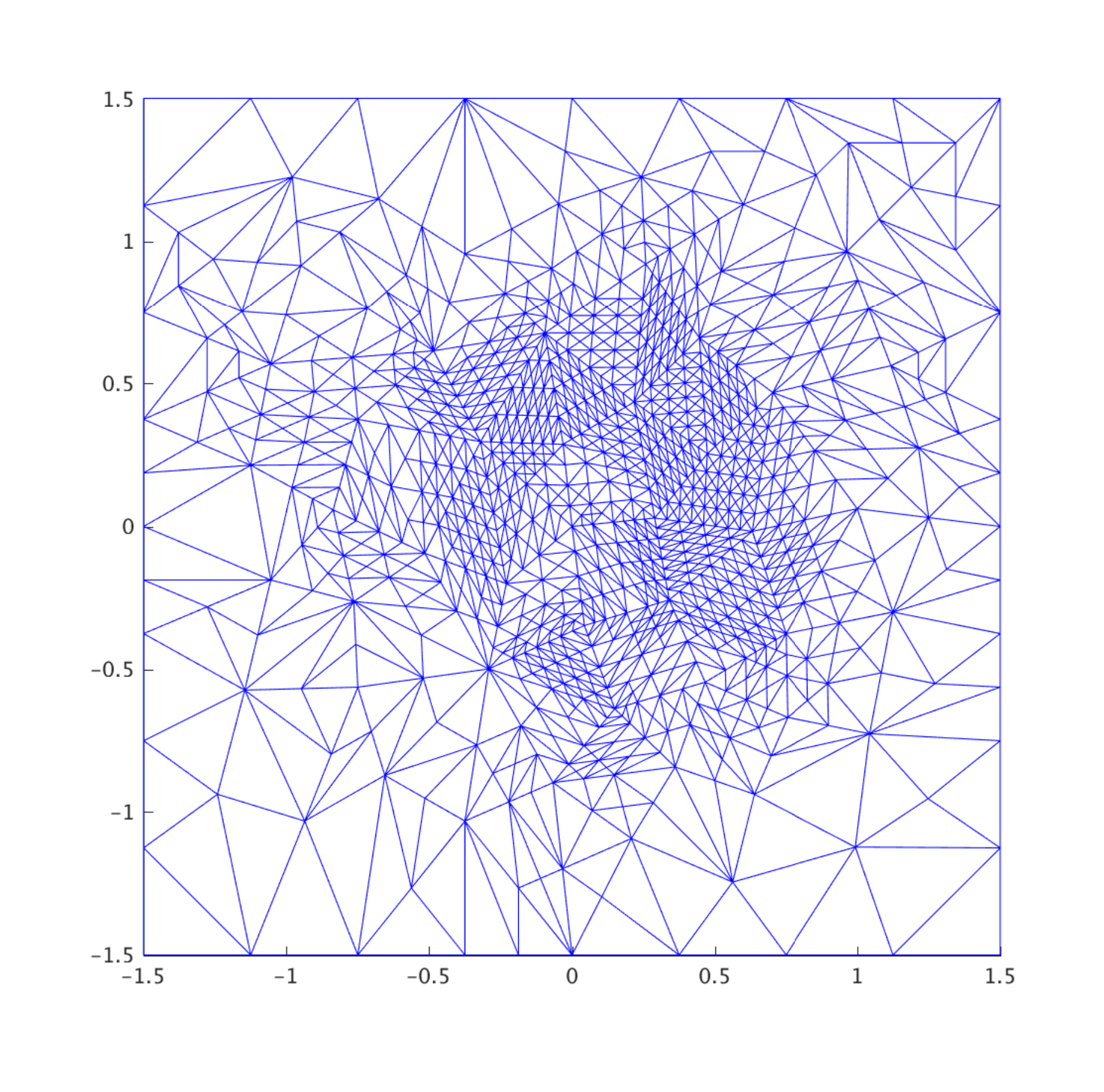}
\put(40,38.8){\includegraphics[width=0.243\textwidth]{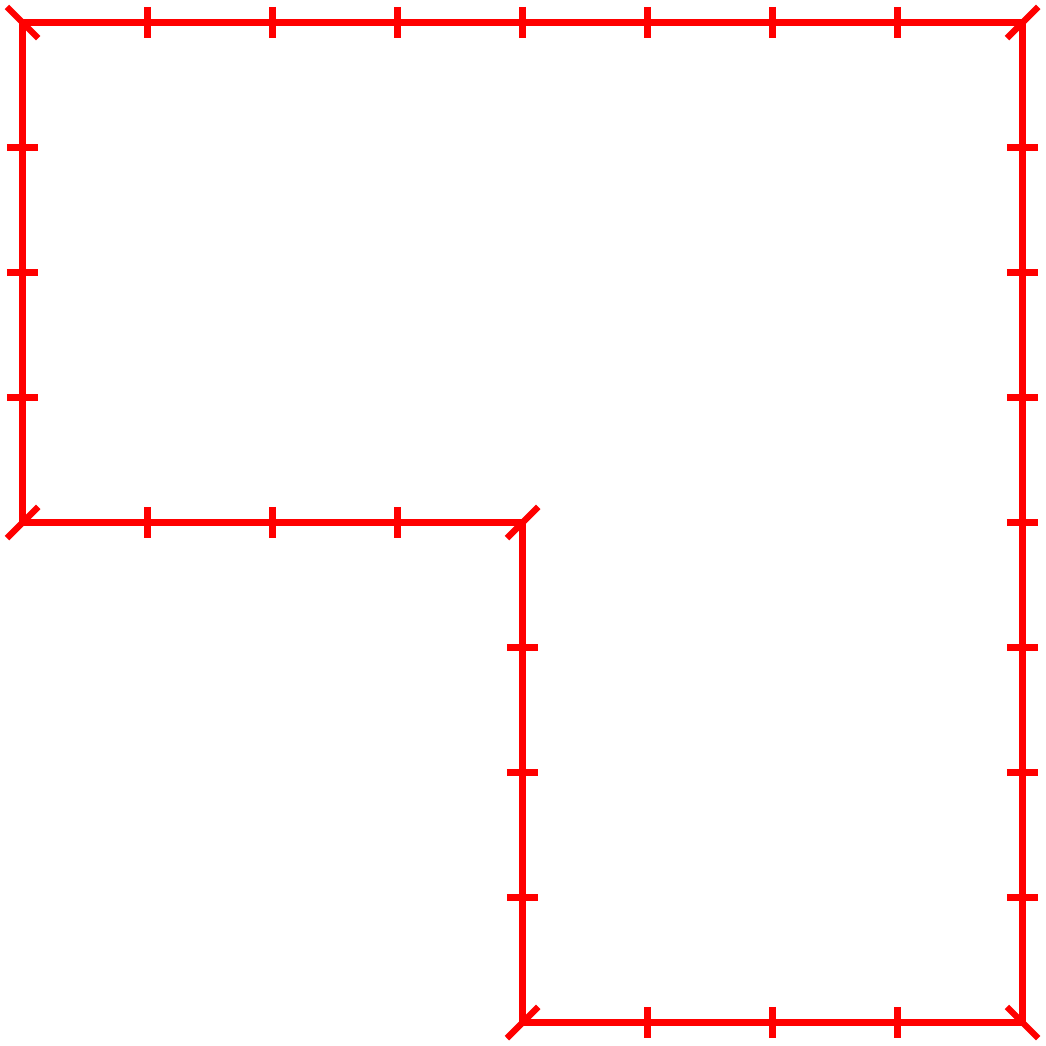}}
\end{overpic}
\begin{overpic}[width=0.45\textwidth]%
{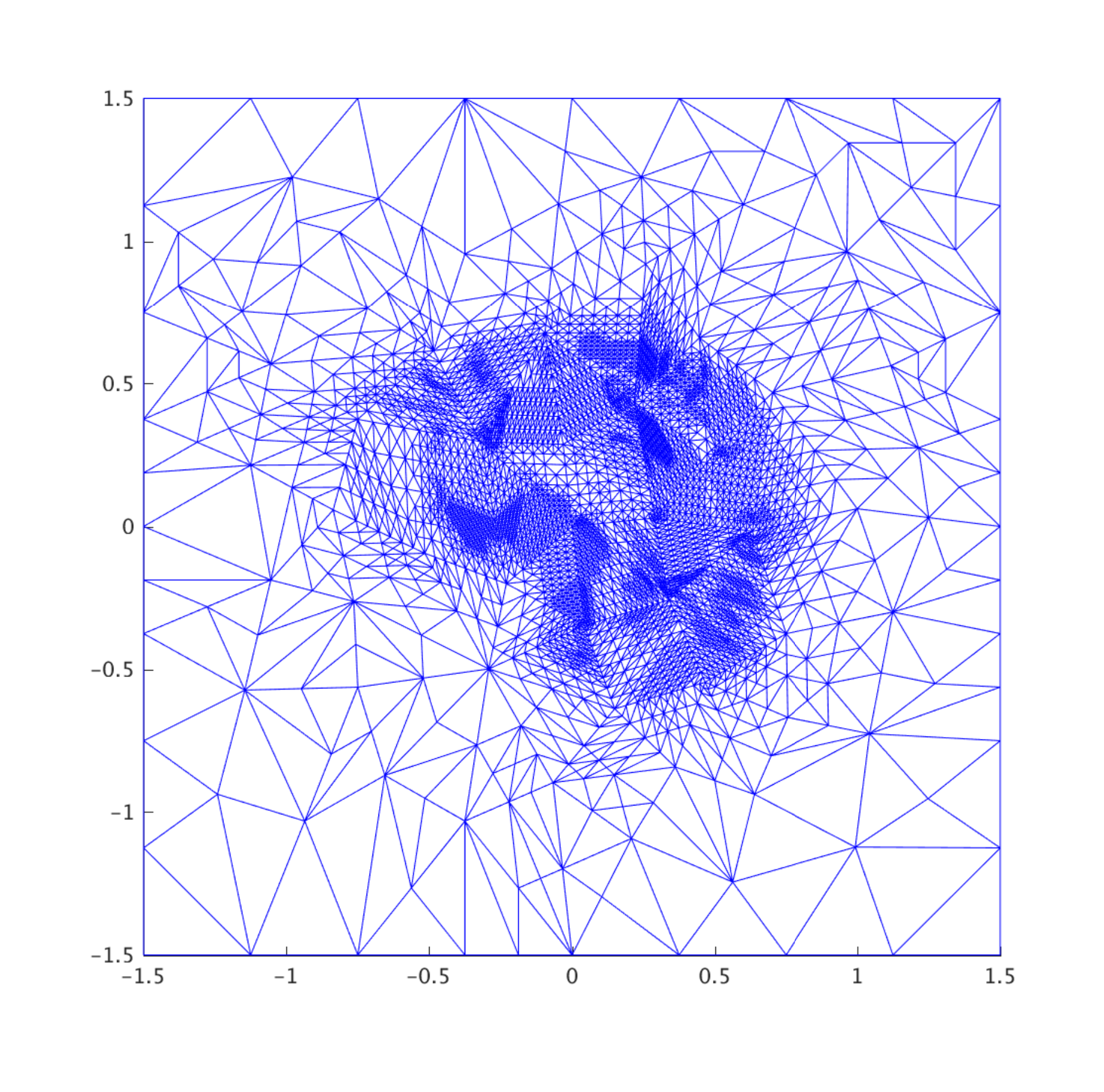}
\put(40,38.8){\includegraphics[width=0.243\textwidth]{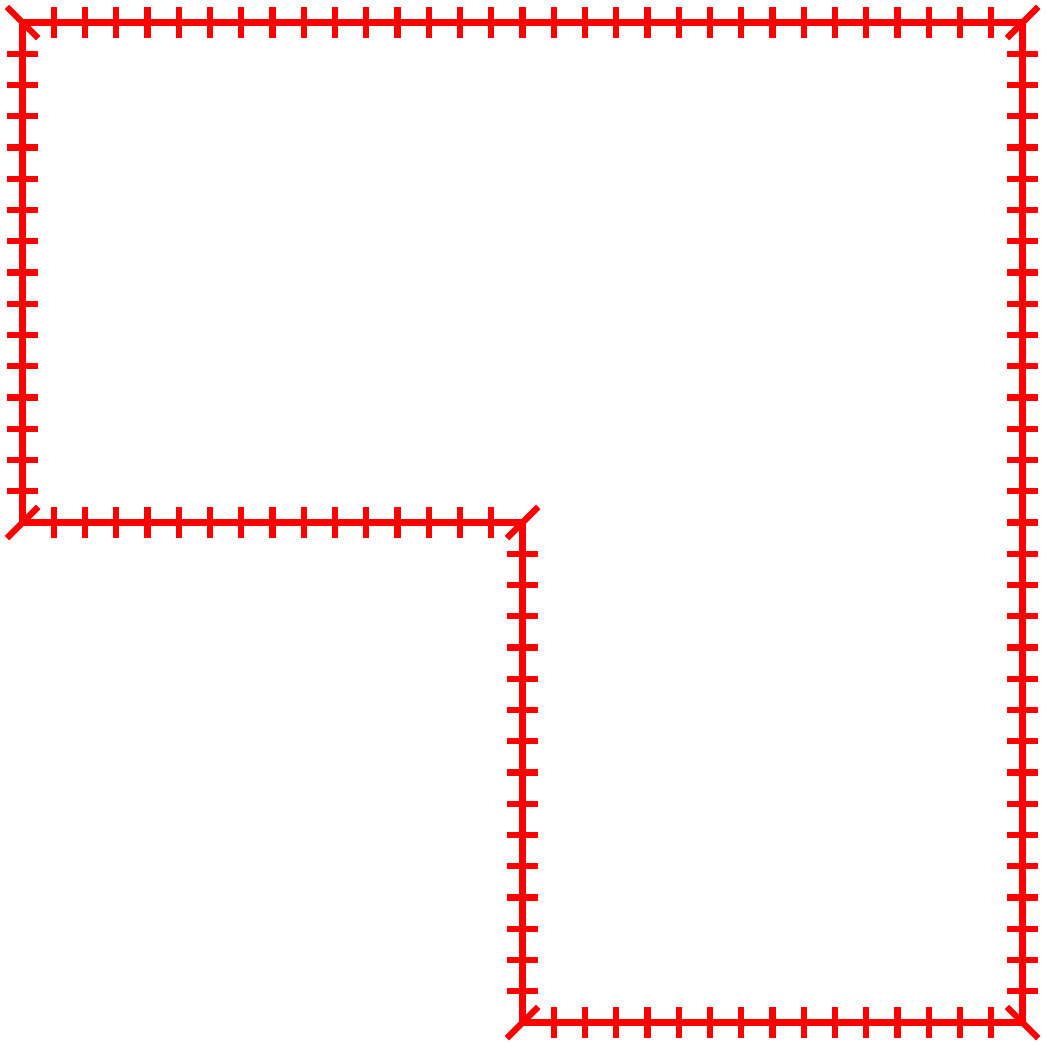}}
\end{overpic}
\caption{Meshes $\tria_{i,K}$, $\sigma_i$ for $i=0,1,3,5$ produced with $K=6$, $\zeta=\sqrt{2}$, and $\theta=0.1$.}
\label{F:meshes}
\end{center}
\end{figure}	


\begin{figure}[ht!]
\begin{center}
\includegraphics[width=0.45\textwidth]{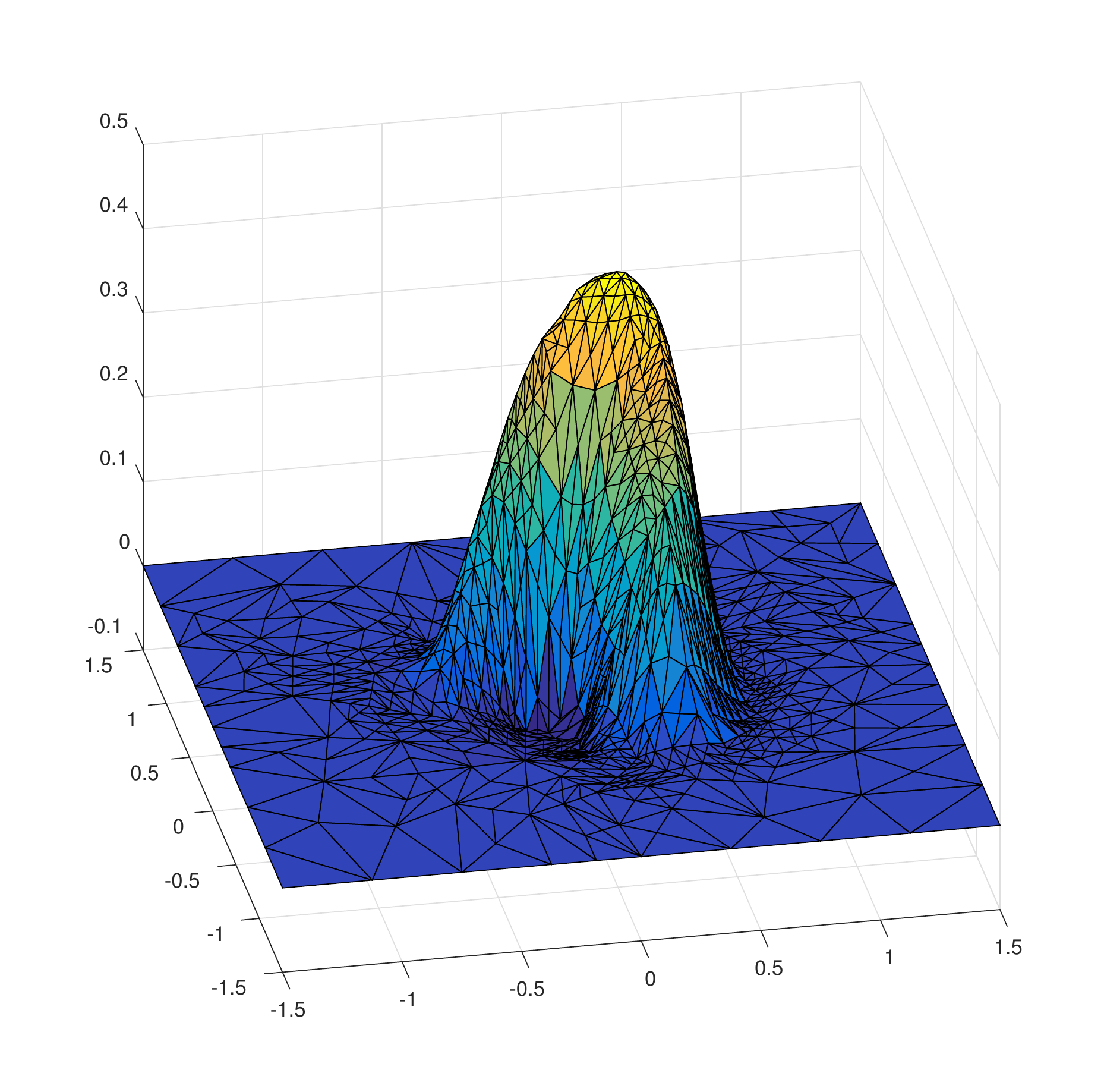}
\includegraphics[width=0.45\linewidth]{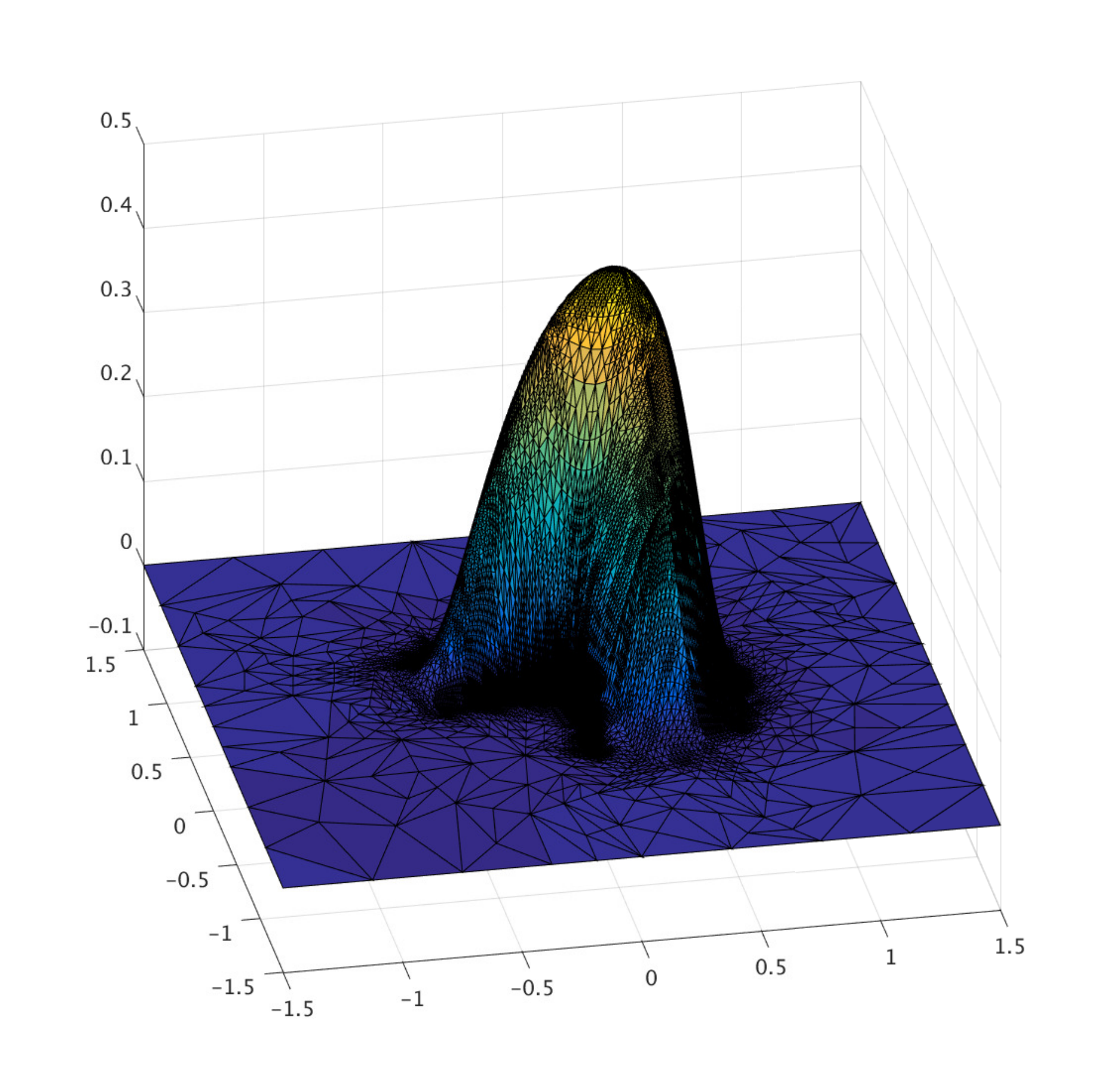}
\caption{Approximations $u_{\tria_{i,K}}^{\lambda_{\sigma_i}^{(K-1)}}$ for $i=3$ (left) and $i=6$ (right) produced with $K=6$, $\zeta=\sqrt{2}$, $\theta=0.1$.}
\label{F:sols3D}
\end{center}
\end{figure}


\begin{figure}[ht!]
\begin{center}
\includegraphics[width=0.45\textwidth]{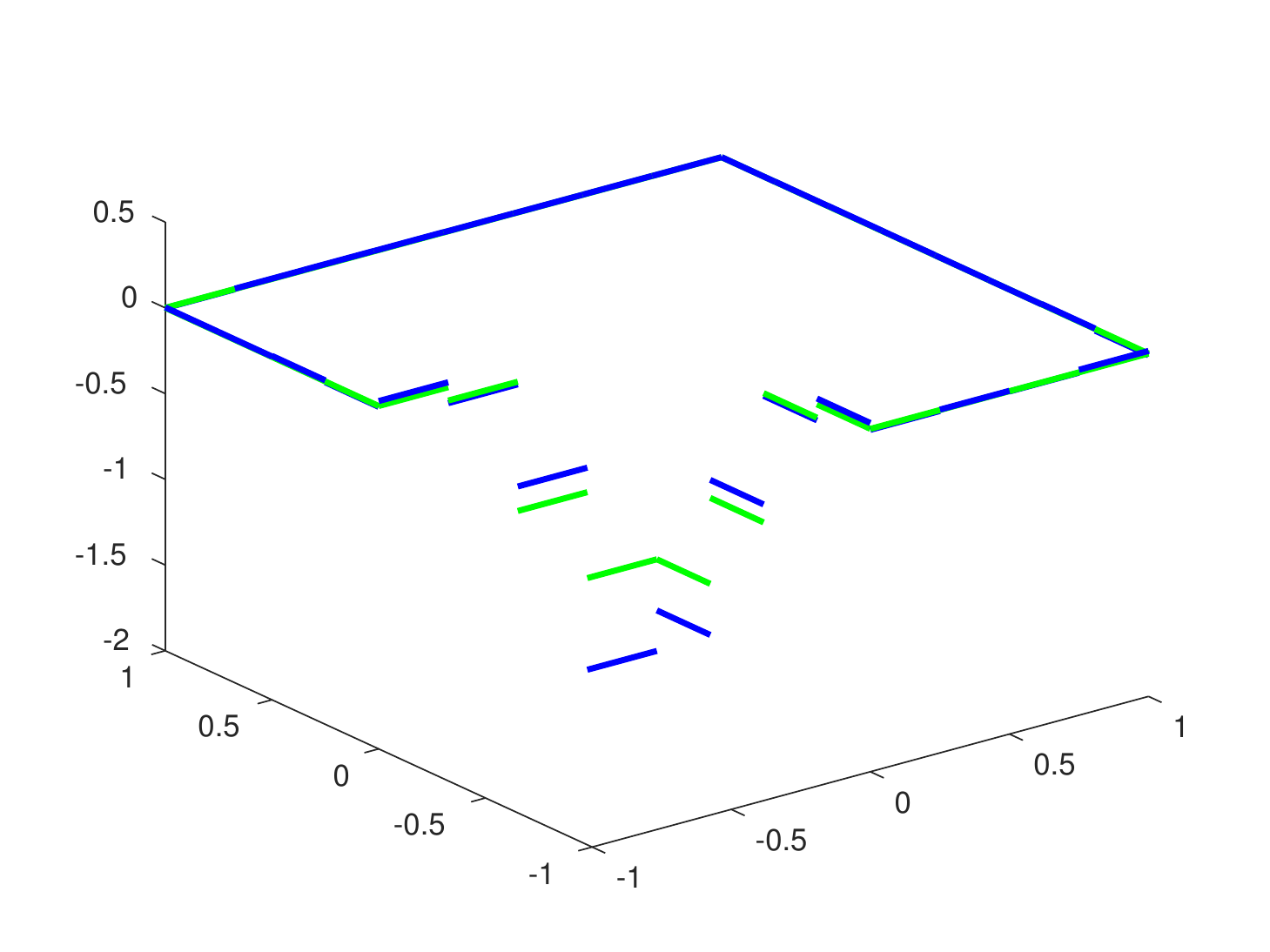}
\includegraphics[width=0.45\textwidth]{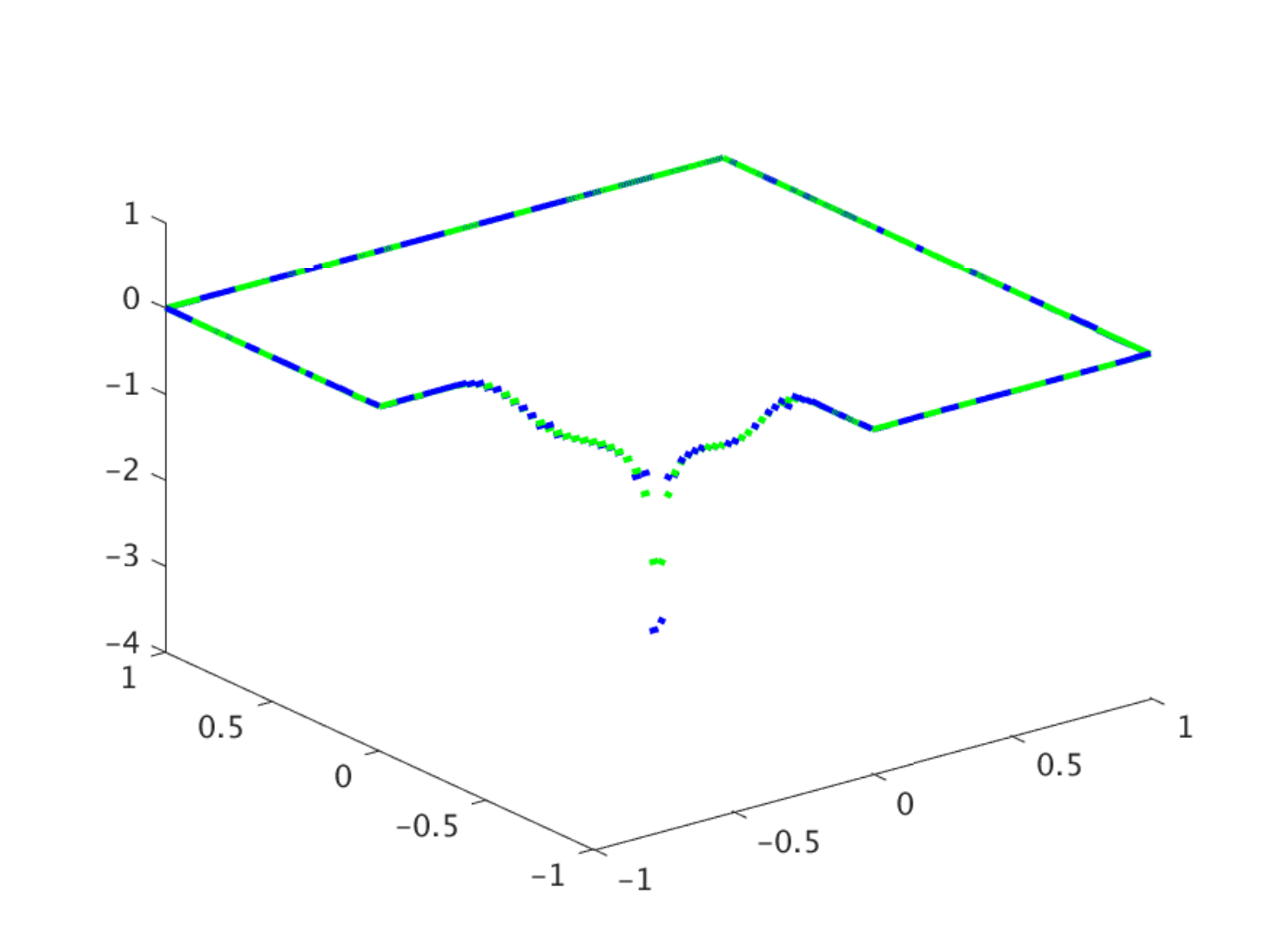}
 \caption{Comparison between $\lambda_{\sigma_i}^{(K-1)}$ (blue) and the $L_2(\gamma)$-orthogonal projection of $\lambda$ onto
    $\LL_{\sigma_{i}}$(green)
    for $i=3$ (left) and $i=6$ (right).}
  \label{F:lambda_sol}
  \end{center}
  \end{figure}

\begin{figure}[ht!]
\begin{center}
\includegraphics[width=0.45\textwidth]{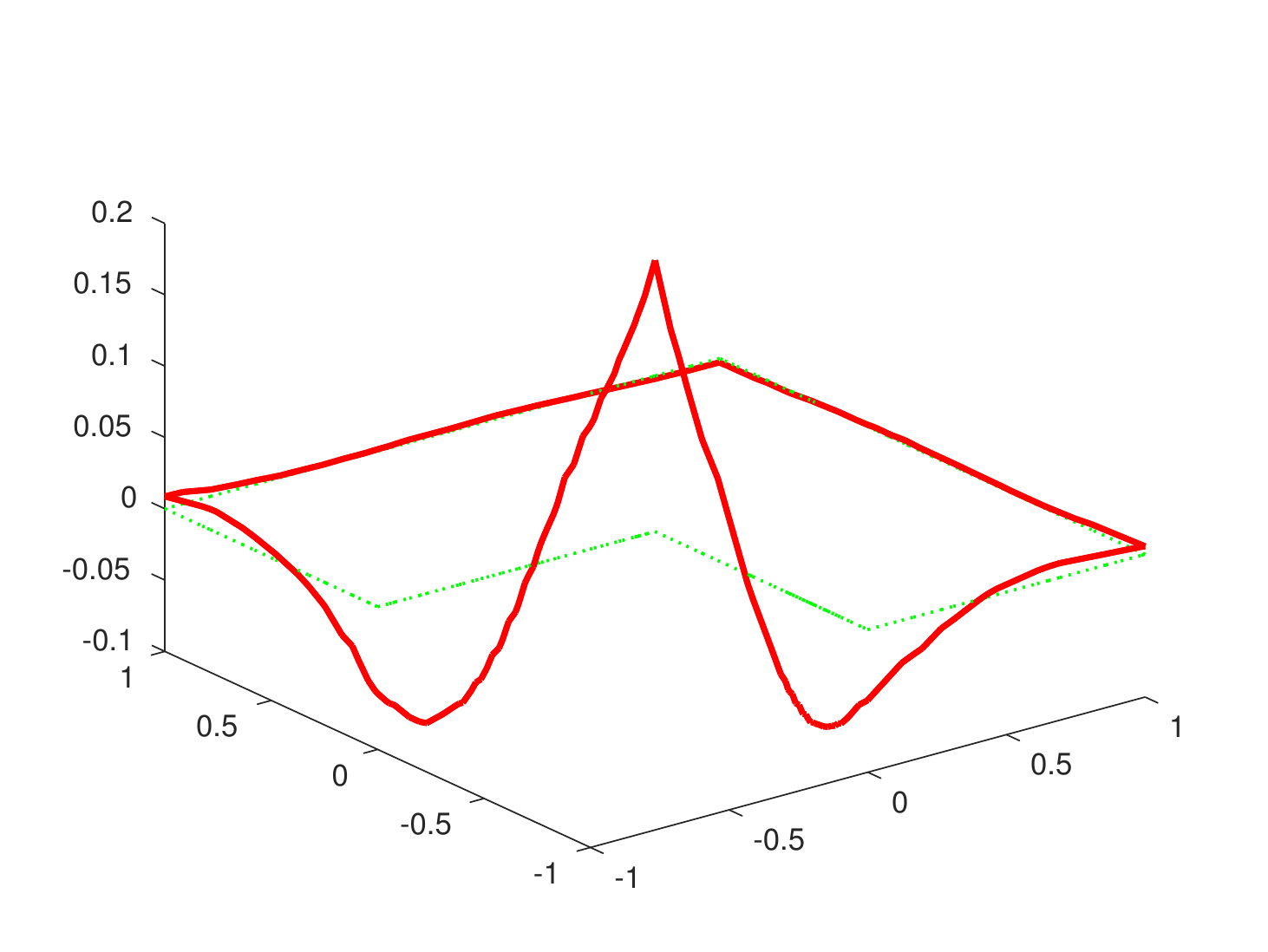}
\includegraphics[width=0.45\linewidth]{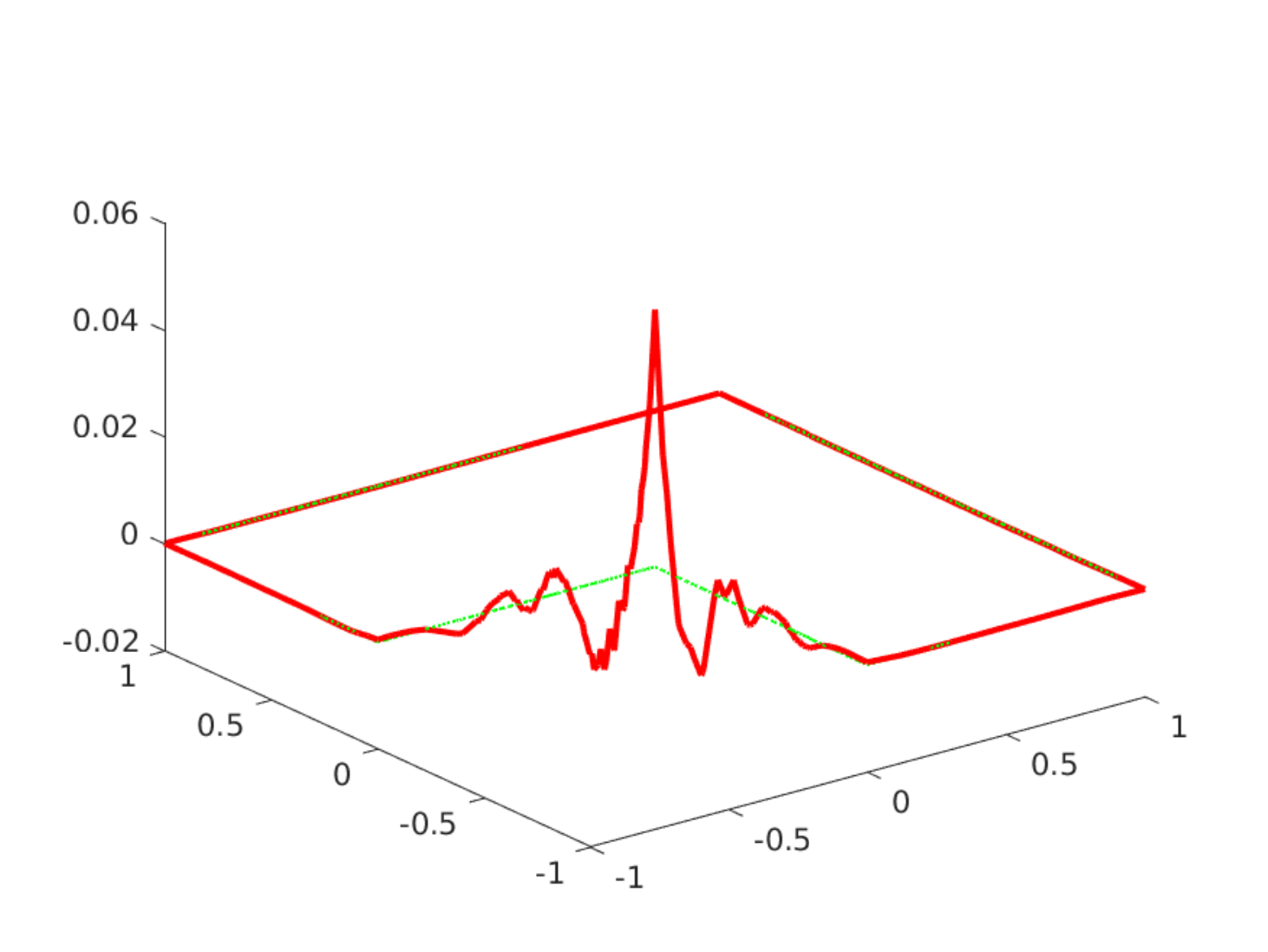}\\
\includegraphics[width=0.45\textwidth]{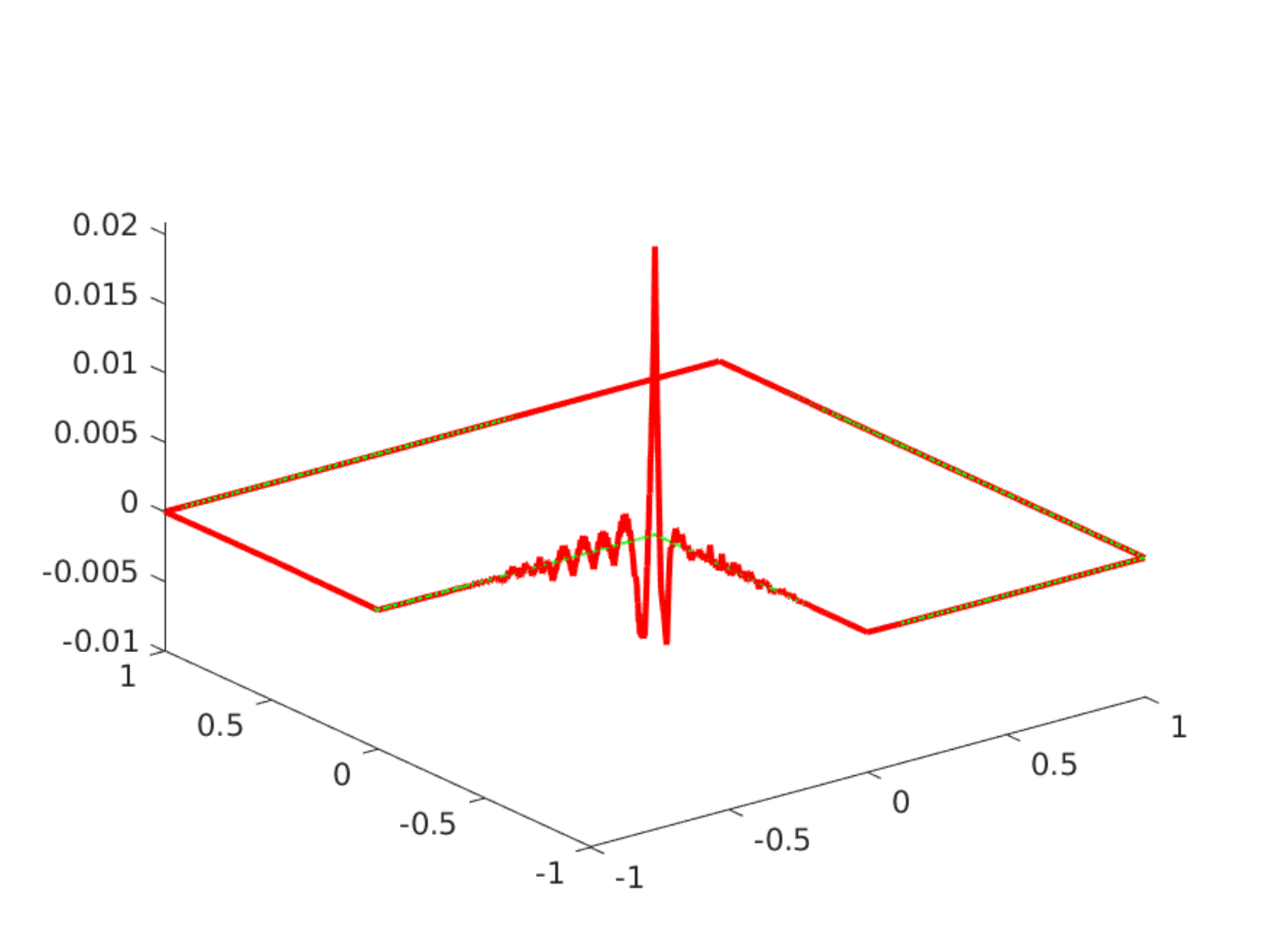}
\includegraphics[width=0.45\linewidth]{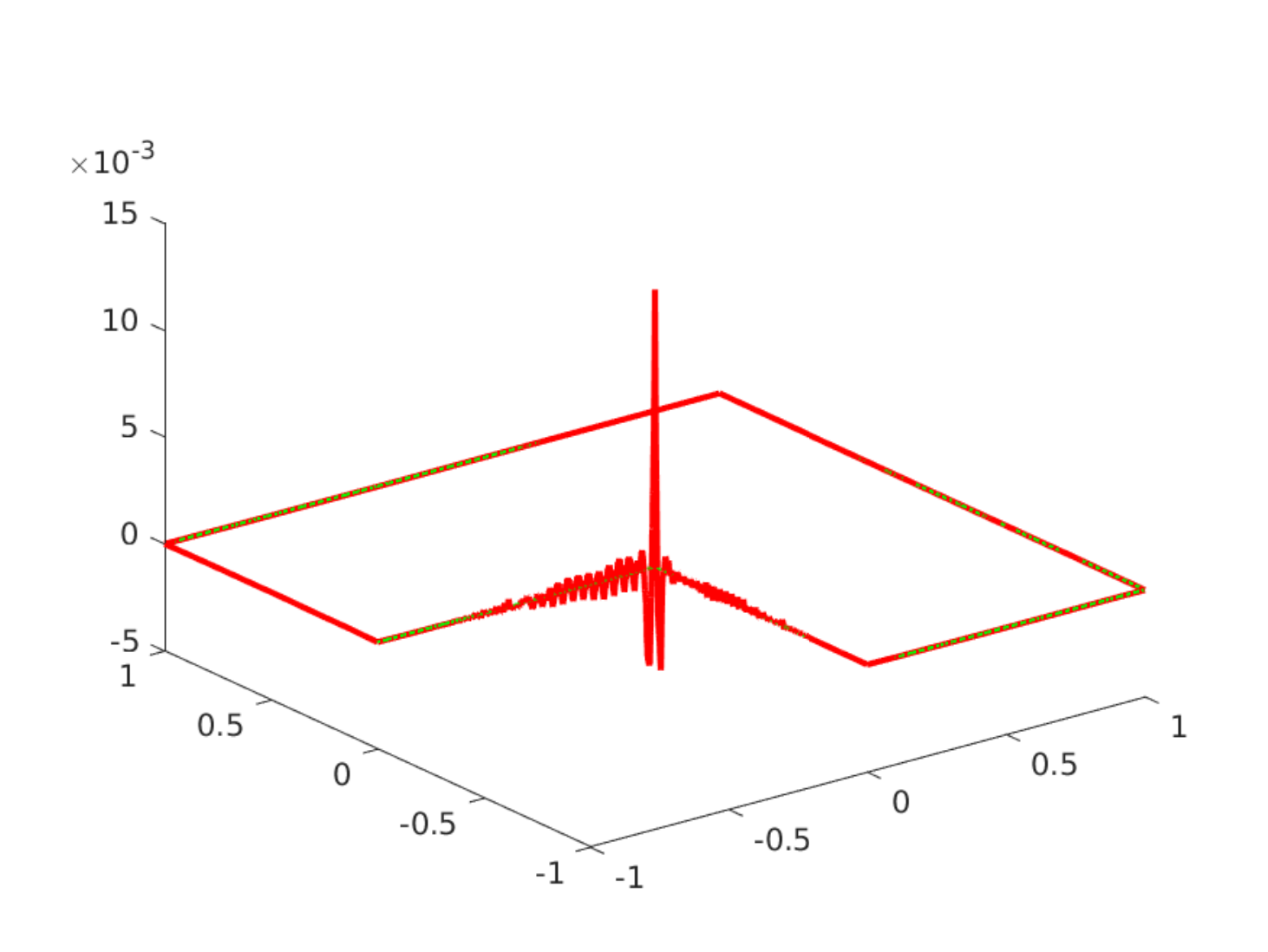}
\caption{Traces $u_{\tria_{i,K}}^{\lambda_{\sigma_i}^{(K-1)}}|_{\gamma}$ (red)
    and $u|_{\gamma} (=0)$ (green) for $i=1,3,5,6$, obtained with $K=6$, $\zeta=\sqrt{2}$, and $\theta=0.1$.}
\label{F:sols}
\end{center}
\end{figure}

In Figure~\ref{F:errors_3}, for $i=1,\ldots,I:=10$, we report the errors
$|u-u_{\tria_{i,K}}^{\lambda_{\sigma_i}^{(K-1)}}|_{H^1{(\Omega)}}$
and $\|\lambda-\lambda_{\sigma_i}^{(K-1)}\|_{H^{-1/2}{(\gamma)}}$, and compare them to the estimators.
We observe a remarkable agreement between the errors and the estimators. In addition, note that  
${E_{\rm outer}}$ and ${E_{\rm inner}}$ exhibit rates of decay comparable with the ones of the errors, whereas ${E_{\rm Uzawa}}$ is in all cases much smaller than the other indicators, displaying a plateau whenever $K>2$ inner iterations are performed. 
For completeness, we mention that the computation of the norm $\|\cdot\|_{H^{-1/2}{(\gamma)}}$ is approximated by 
first building the $L_2(\gamma)$-orthogonal projection $\mu$ of the error $\lambda-\lambda_{\sigma_i}^{(K-1)}$
onto $\LL_{\sigma_{I+2}}$ and then employing \eqref{precond} to get
$\|\lambda-\lambda_{\sigma_i}^{(K-1)}\|_{H^{-1/2}{(\gamma)}}\simeq \sqrt{(M_{\sigma_{I+2}} \mu)(\mu)}$.

\begin{figure}[ht!]
\begin{center}
\includegraphics[width=0.45\textwidth]{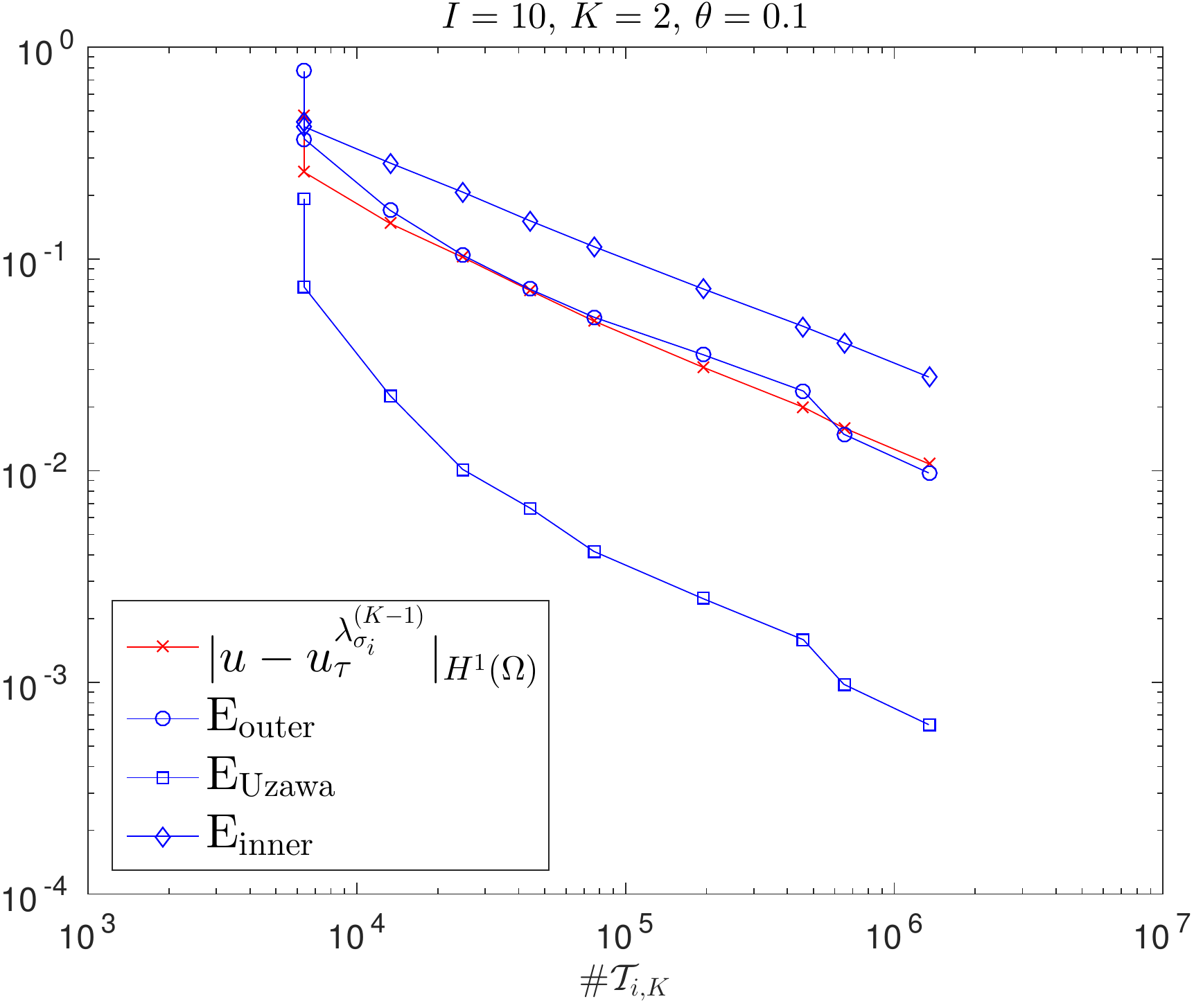}
\includegraphics[width=0.45\textwidth]{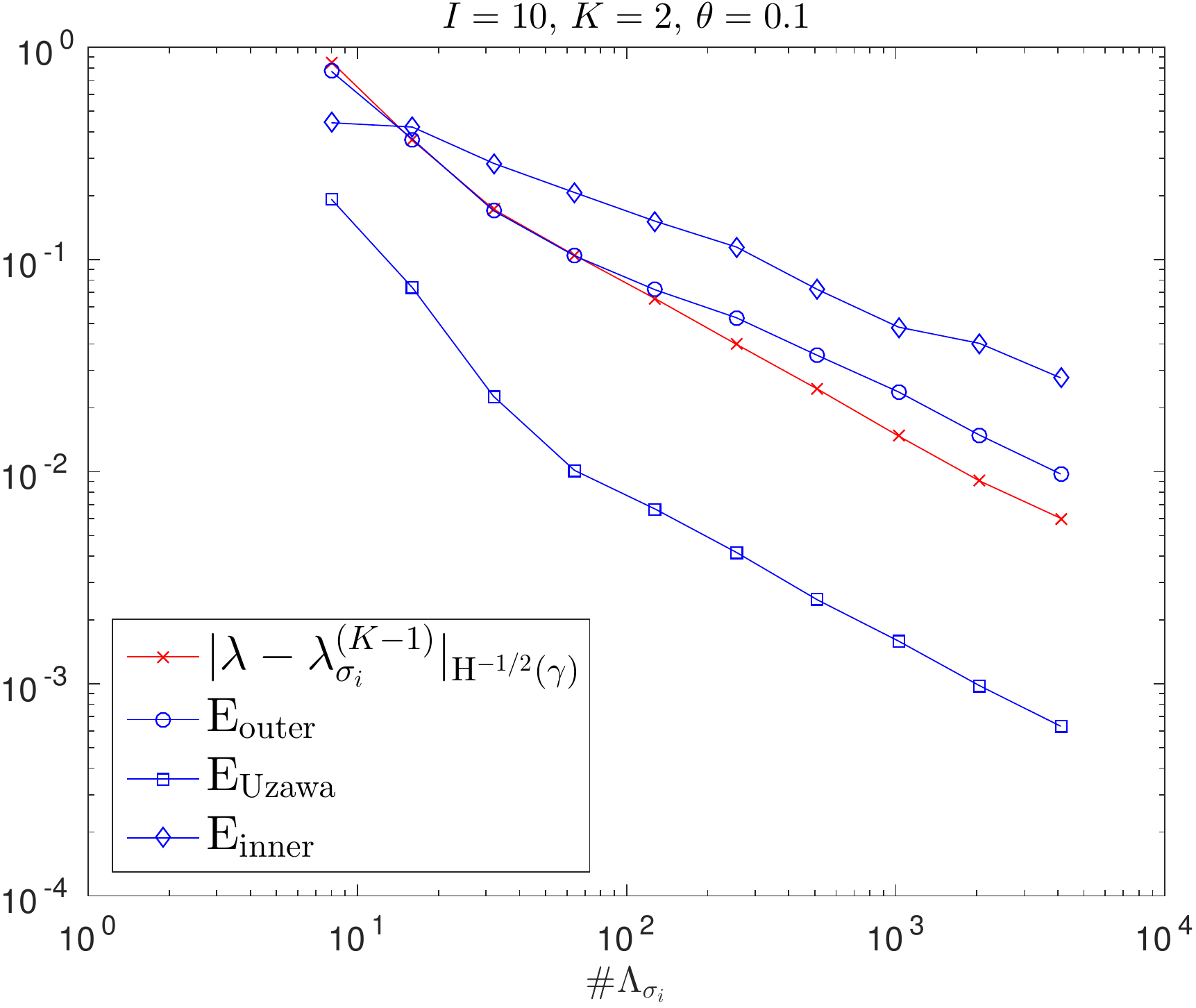} \\
\includegraphics[width=0.45\textwidth]{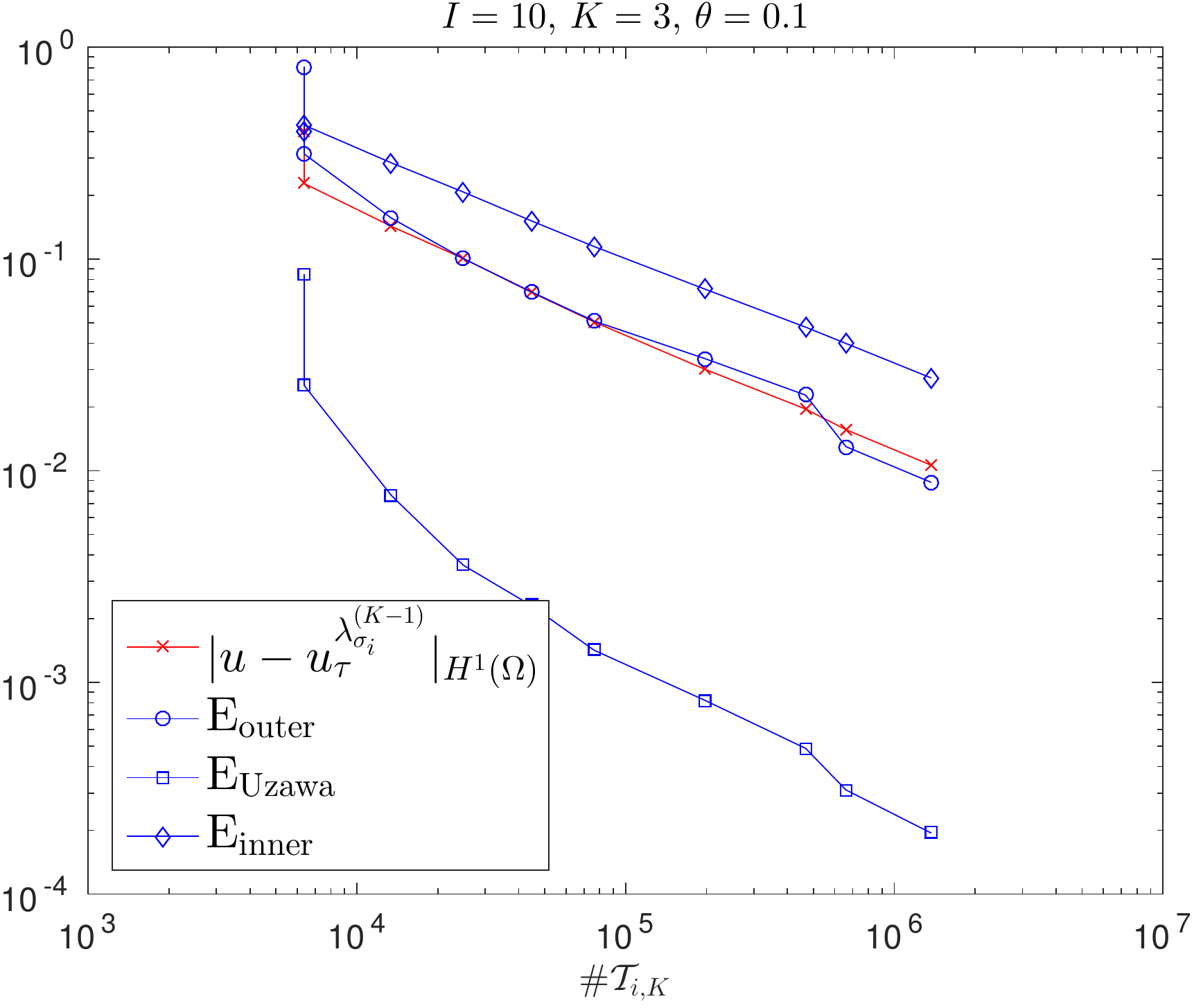}
\includegraphics[width=0.45\textwidth]{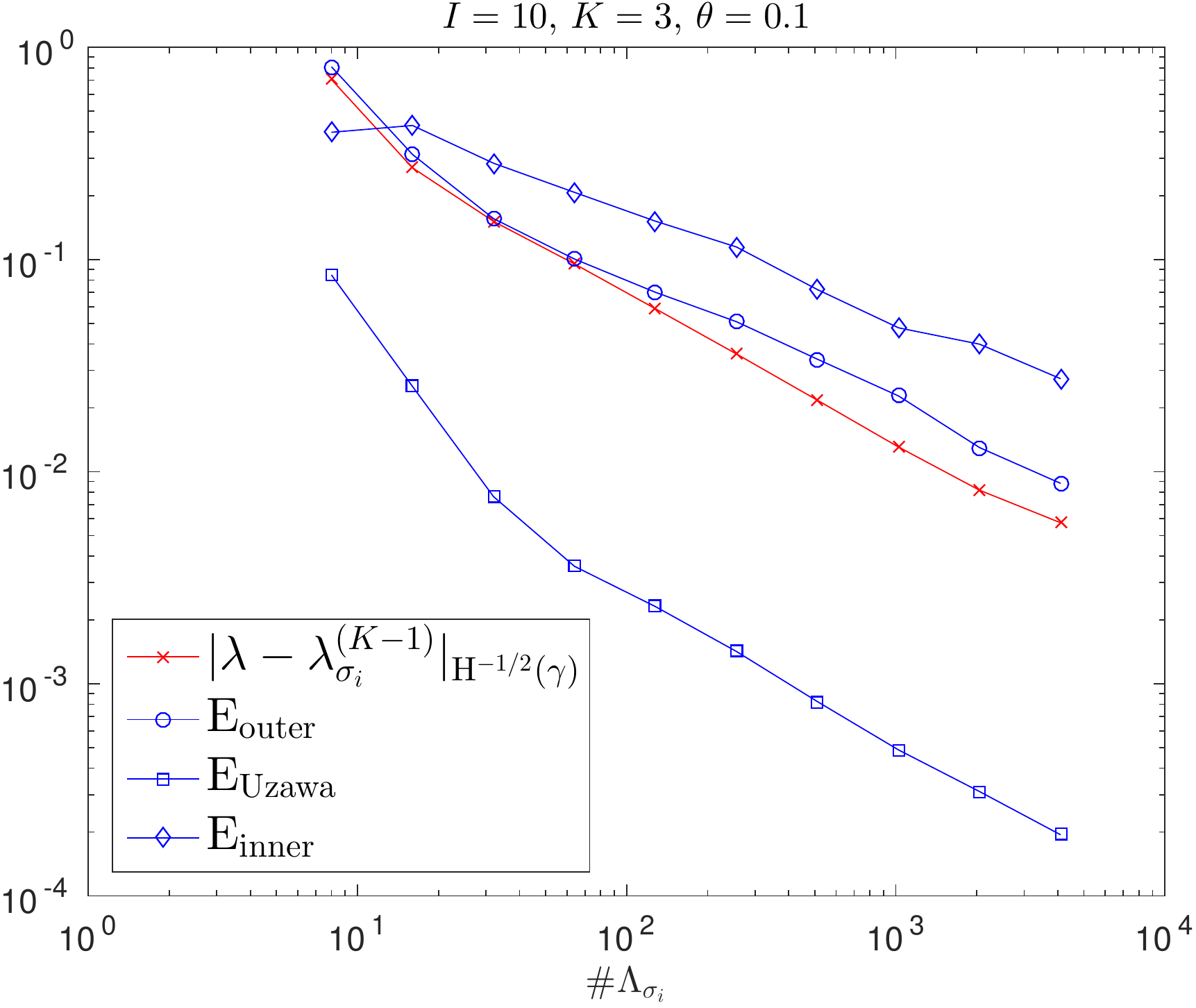}\\
\includegraphics[width=0.45\textwidth]{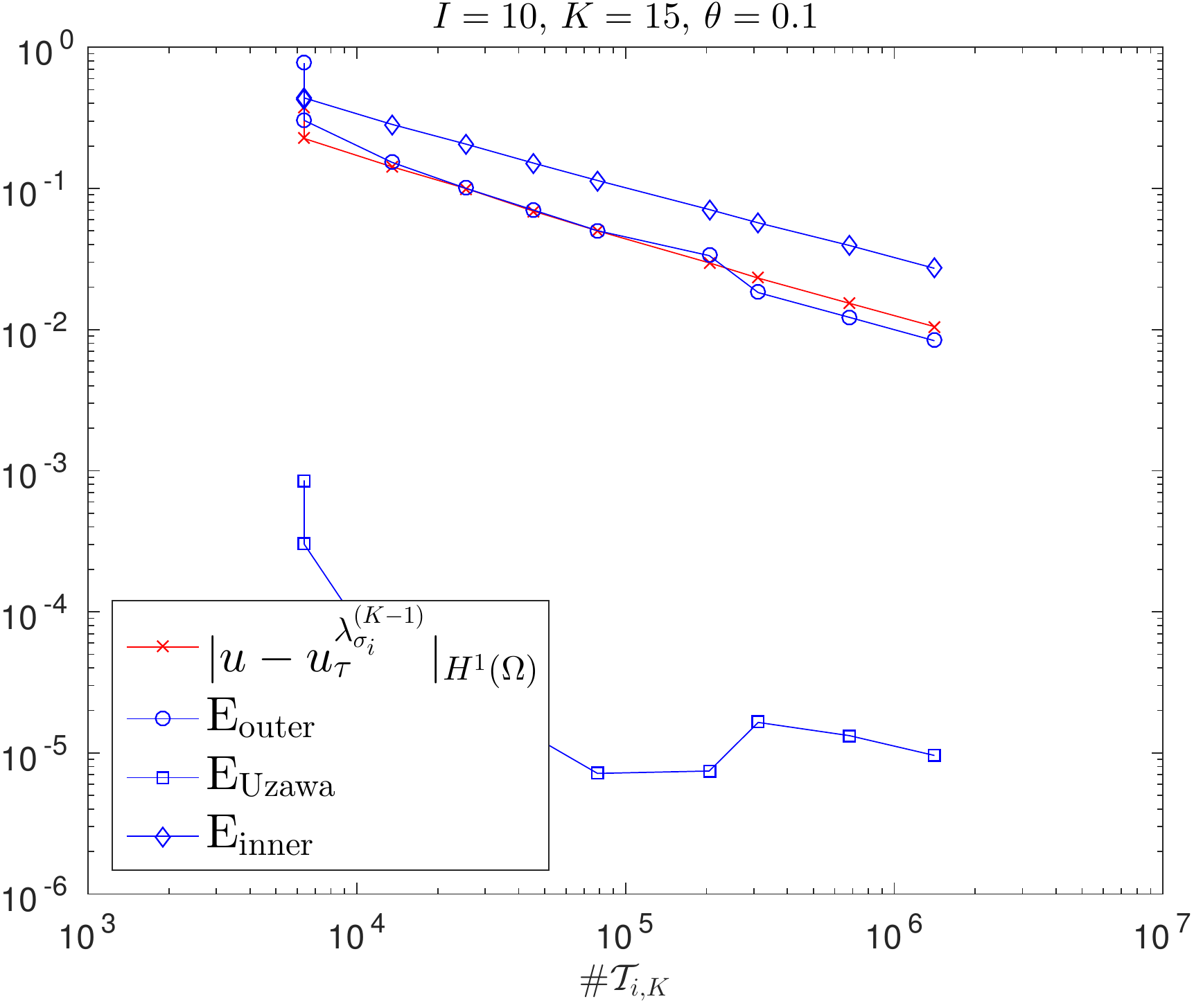}
\includegraphics[width=0.45\textwidth]{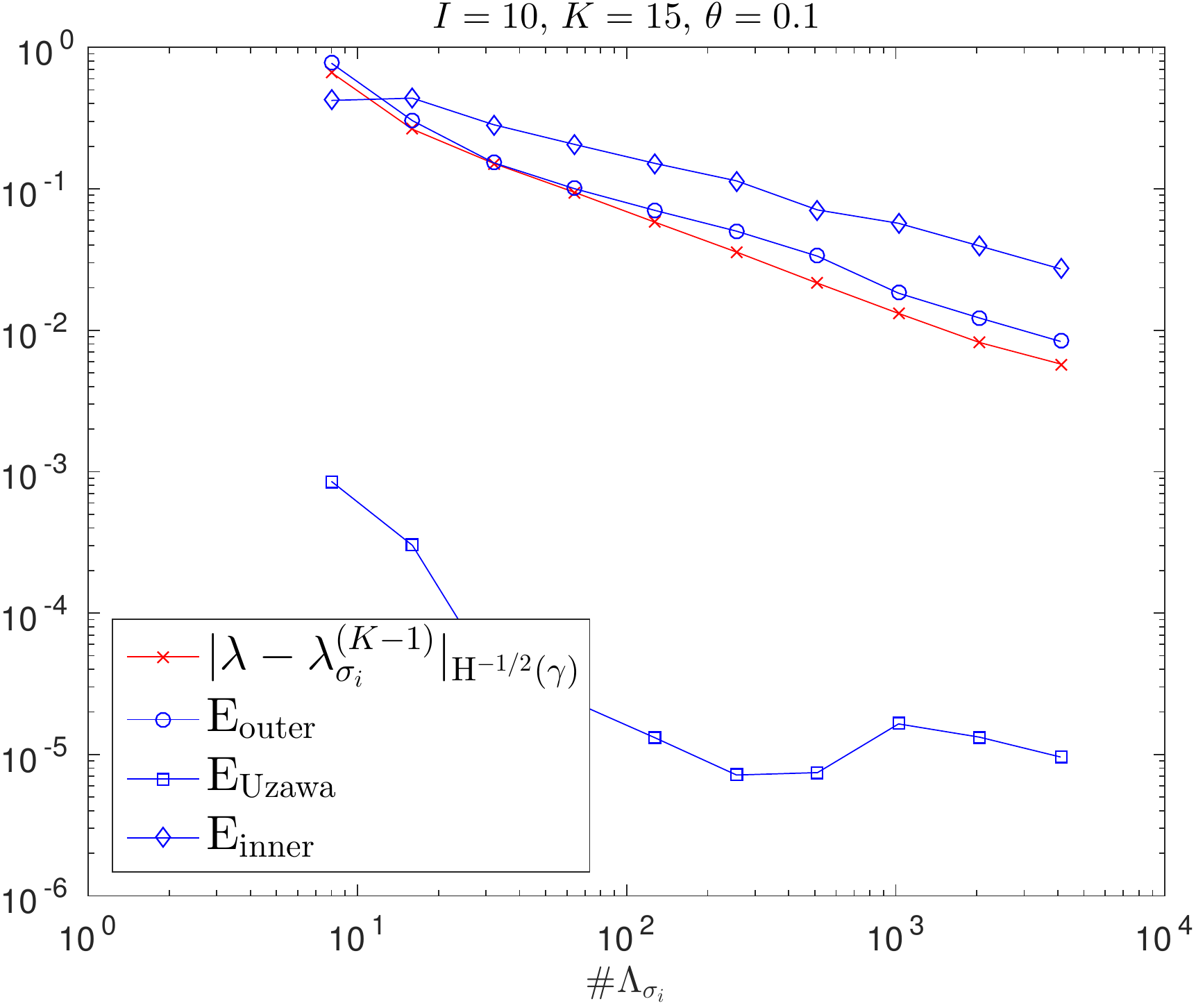}
\caption{Errors $|u-u_{\tria_{i,K}}^{\lambda_{\sigma_i}^{(K-1)}}|_{H^1{(\Omega)}}$ (left) and $\|\lambda-\lambda_{\sigma_i}^{(K-1)}\|_{H^{-1/2}{(\gamma)}}$ (right), and estimators for 
$\zeta=\sqrt{2}$, $\theta=0.1$, and $K=2$ (top), $K=3$ (middle)  and $K=15$ (bottom).}
\label{F:errors_3}
\end{center}
\end{figure}

In Table~\ref{P1_Ntot-errs}, we report the rates of convergence for the errors 
$|u-u_{\tria_{i,K}}^{\lambda_{\sigma_i}^{(K-1)}}|_{H^1{(\Omega)}}$ and
$\|\lambda-\lambda_{\sigma_i}^{(K-1)}\|_{H^{-1/2}{(\gamma)}}$
with respect to $\# \tria_{i,K}$ and $\# \sigma_i$, respectively.
The rates are computed after excluding the first three iterations of the algorithms.
The convergence rate of the $H^1(\Omega)$-error for $u$ is always close to the expected value $0.5$ while 
the convergence rate of the ${H^{-1/2}{(\gamma)}}$-error for $\lambda$ is $0.69$. The latter is in agreement with the theoretical rate $\frac{2}{3}$ expected since $\lambda\in H^s(\gamma)$ for any $s<\frac{1}{6}$. 
Finally, in the last two columns we report the number of elements of $\tria_{I,K}$ and $\sigma_I$ at the last iteration $I=10$.



In Table~\ref{P1_estimators}, we report the rates of convergence of the estimators ${E_{\rm outer}}$, ${E_{\rm Uzawa}}$ and ${E_{\rm inner}}$. The rates observed for $E_{\rm outer}$ are closer to the theoretical value $\frac{2}{3}$ when $K$ increases.
The rates obtained for ${E_{\rm inner}}$ always matche (up to the third significant digit) the theoretical rate expected for $|u-u_{\tria_{i,K}}^{\lambda_{\sigma_i}^{(K-1)}}|_{H^1{(\Omega)}}$. Finally, the low rates exhibited by ${E_{\rm Uzawa}}$ are explained by the appearance of plateaux for larger values of $i$ when $K >3$ inner iterations are performed.

We conclude this section with one additional table focusing on the behavior of the inner adaptive solver.
Recall that in Algorithm ~\ref{a:fd} a fixed number of inner iterations $j=1,..,K$ is performed within each outer iteration $i=1,..,I$.
Each of these inner iterations lead to bulk mesh refinement (Algorithm~\ref{a:AFEM}) whenever $\Upper \cE(f,\chi,\tria_{k}) > \Cst\zeta^{-i}$. 
In Table~\ref{P1_k_afem}, for each outer iteration $i$, we report  the number of times that the bulk mesh refinement is  performed and observe that the refinements are never performed after the second inner iteration.

\begin{table}[ht!]
\caption{Computed rates of convergence of 
$e_{u}:=|u-u_{\tria_{i,K}}^{\lambda_{\sigma_i}^{(K-1)}}|_{H^1{(\Omega)}}$
and
$e_{\lambda}:=\|\lambda-\lambda_{\sigma_i}^{(K-1)}\|_{H^{-1/2}{(\gamma)}}$
w.r.t. 
$\# \tria_{i,K}$ and $\# \sigma_i$, 
for different values of $K$ ($I=10$, $\theta=0.1$ and $\zeta=\sqrt{2}$).}
\begin{tabular}{|c|c|c|c|c|}\hline
  & $e_{u}$
  & $e_{\lambda}$
  & $\# \tria_{I,K}$
  & $\# \sigma_I$ \\ \hline
$K=2$  & 0.56  & 0.70 & 1344310 & 4096 \\ \hline
$K=3$  & 0.55  & 0.69 & 1372266 & 4096 \\ \hline
$K=6$  & 0.56  & 0.69 & 1411114 & 4096 \\ \hline
$K=9$  & 0.56  & 0.69 & 1411274 & 4096 \\ \hline
$K=15$ & 0.56  & 0.69 & 1411254 & 4096 \\ \hline
\end{tabular}
\label{P1_Ntot-errs}
\end{table}

\begin{table}[ht!]
\caption{Computed rates of convergence of the error estimators
$E_{\rm outer}$, $E_{\rm Uzawa}$ and $E_{\rm inner}$, respectively,
for different values of $K$ ($I=10$, $\theta=0.1$ and $\zeta=\sqrt{2}$).}
\begin{tabular}{|c|c|c|c|}\hline
  & $E_{\rm outer}$
  & $E_{\rm Uzawa}$
  & $E_{\rm inner}$ \\ \hline
$K=2$ & 0.58 & 0.70 & 0.50 \\ \hline
$K=3$ & 0.59 & 0.71 & 0.50 \\ \hline
$K=6$ & 0.63 & 0.67 & 0.50 \\ \hline
$K=9$ & 0.63 & 0.38 & 0.50 \\ \hline
$K=15$& 0.63 & 0.10 & 0.50 \\ \hline
\end{tabular}
\label{P1_estimators}
\end{table}

\begin{table}[ht!]
\caption{Number of inner iterations at which bulk mesh refinement is activated, for different values of $K$ ($I=10$, $\theta=0.1$ and $\zeta=\sqrt{2}$).}
\label{P1_k_afem}
\begin{tabular}{|c|c|c|c|c|c|c|c|c|c|c|}\hline
$i$    & 1 & 2 & 3 & 4 & 5 & 6 & 7 & 8 & 9 & 10 \\ \hline
$K=2$  & 2 & 0 & 1 & 1 & 1 & 1 & 1 & 1 & 1 &  1 \\ \hline
$K=3$  & 2 & 0 & 1 & 1 & 1 & 1 & 2 & 0 & 0 &  1 \\ \hline
$K=6$  & 2 & 0 & 1 & 1 & 1 & 1 & 1 & 1 & 1 &  1 \\ \hline
$K=9$  & 2 & 0 & 1 & 1 & 1 & 1 & 1 & 1 & 1 &  1 \\ \hline
$K=15$ & 2 & 0 & 1 & 1 & 1 & 1 & 1 & 1 & 1 &  1 \\ \hline
\end{tabular}
\end{table}


\section{General $d$-dimensional domains and/or higher finite element spaces} \label{Sgeneralizations}
 So far we considered the case of $d=2$ space dimensions, and lowest order approximation, i.e., continuous piecewise linears for $u$, piecewise constants for $\lambda$. We now discuss the case of general $d \geq 2$, and general polynomial orders.
 
First we address the question for which $s>0$, membership of $u$ in $\cA^s$ can be expected when $u$ is approximated from families of continuous piecewise polynomials of order $p \geq 2$.
Since generally $\lambda \neq 0$, the normal derivative of $u$ has a generally non-zero jump over the $(d-1)$-dimensional manifold $\gamma$, generally being not-aligned with any mesh.
 Assuming that apart from this jump, the solution $u$ is smooth, the question of approximability of $u$ in $H^1(\Omega)$  is equivalent to the question of approximability in $L_2(\Omega)$ of a piecewise smooth function, say a piecewise constant one w.r.t. the partition of $\Omega$ into $\wideparen{\Omega}$ and $\Omega \setminus \overline{\wideparen{\Omega}}$, from families of discontinuous polynomials of order $p-1$.
 Taking cells of diameter $h$ that intersect $\gamma$, regardless of the order $p$ the squared $L_2(\Omega)$-norm of the latter approximation error is $\eqsim h^d$ times the number of those cells, being of the order $(1/h)^{d-1}$.
 We infer that in terms of the total number $N$ of elements in the mesh, which satisfies $N \gtrsim (1/h)^{d-1}$, and with a proper refinement towards $\gamma$, even satisfies $N \eqsim (1/h)^{d-1}$, it holds that the $L_2(\Omega)$-norm of this error is $\sqrt{h} \eqsim N^{-\frac{1}{2(d-1)}}$.
 We conclude that generally at best $u \in \cA^{\frac{1}{2(d-1)}}$. 
 
 On the other hand, if the solution $\wideparen{u}$ of our original PDE, posed on $\wideparen{\Omega}$, is approximated from families of continuous piecewise polynomials of order $p$ w.r.t. (isotropic) partitions of $\wideparen{\Omega}$, then under appropriate (Besov) smoothness conditions, $\wideparen{u}$ can be approximated at rate $\frac{p-1}{d}$.
 
 \begin{remark}
 Other than for $d=2$, for $d>2$ and arbitrary Lipschitz domains
 these Besov smoothness conditions are not automatically valid for sufficiently smooth data, in which case this rate $\frac{p-1}{d}$ can only be realized by proper anisotropic refinements.
 \end{remark}
 
 Since for $d>2$ or $p>2$, it holds that $\frac{1}{2(d-1)} < \frac{p-1}{d}$, we conclude that for those $(d,p)$ a price to be paid for the application of the Fictitious Domain Method instead of the usual finite element method is that generally it results in a reduced best approximation rate.
 
\begin{remark} This deficit of the Fictitious Domain Method might be tackled by considering anisotropic refinements allowing  for a more accurate approximation of $\gamma$, by enriching the local finite element space on elements that intersect $\gamma$, or, as we will study in future work, by constructing an extension of $\wideparen{f}$ on $\wideparen{\Omega}$ to $f$ on $\Omega$ that yields a multiplier $\lambda$ that is small or preferably zero, and thus avoids the discontinuity in the normal derivative of $u$ over $\gamma$.
\end{remark}

Knowing that the solution $u$ of the Fictitious Domain Method is at best in $\cA^{\frac{1}{2(d-1)}}$, the straightforward generalization to $d$-dimensions
of the adaptive solution method that we have developed for $d=2$ yields the best possible approximation rate.
Indeed, assuming $f \in L_2(\Omega)$ and $g \in H^1(\gamma)$, it holds that $\lambda \in L_2(\gamma)$ and so its approximation in $H^{-\frac{1}{2}}(\gamma)$ by piecewise constants w.r.t. to uniform meshes converges with rate $\frac{1}{2(d-1)}$.
A direct generalization of \cite[Thms. 7.3-4]{45.47} from $2$ to $d$ dimensions shows that $f \in L_2(\Omega)$ and $\chi \in L_2(\gamma)$ are in the data approximation classes $\cB_\Omega^{\frac{1}{d}}$ and $\cB_\Omega^{\frac{1}{2(d-1)}}$, respectively (cf. Thm.~\ref{rhsclasses}).
Now the generalization of Thm.~\ref{rate} to $d$-dimensions shows that whenever $u \in \cA^s$ for some $s \in (0, \frac{1}{2(d-1)}]$, the sequence of approximations produced by our nested inexact preconditioned Uzawa algorithm converges with this rate $s$.

 Concluding we can  say that in any dimension our adaptive method solves the fictitious domain formulation with the best possible rate. On the other hand, without constructing a very special extension of $\wideparen{f}$, for $d>2$ (or $p>2$) this rate is generally lower that the best possible rate with which the original PDE can be solved with standard finite elements, i.e., w.r.t. to partitions of the original domain.
 
\newcommand{\etalchar}[1]{$^{#1}$}

\end{document}